\documentclass[10pt]{article}

\usepackage[ngerman,french,english]{babel}

\usepackage[left]{showlabels}                                           % show labels

\usepackage{pdfsync}

\usepackage{amsmath}                                                        % need for subequations
\numberwithin{equation}{section}
\usepackage{graphicx}                                                       % need for figures
\usepackage{subfigure}
\usepackage{enumitem}                                                    % need for subfigures
\usepackage{hyperref}
\usepackage{amssymb}                                                        % gives you \mathbb{} font
\usepackage[mathscr]{eucal}                                             % gives you \mathscr font
\usepackage{cancel}                                                             % gives you the ability to visibly cross out terms in equations}
\usepackage[normalem]{ulem}                                                                 % gives you \sout
\usepackage{pstricks}
\usepackage{rotating}
\usepackage{lscape}
\usepackage[paperwidth=8.5in,paperheight=11in,top=1.00in, bottom=1.00in, left=1.00in, right=1.00in]{geometry}
\usepackage{mathtools}                                                      % need for `show only references'
\mathtoolsset{showonlyrefs=true}                                    % only equations which are labeled AND referenced will be numbered.
\usepackage{fixltx2e,amsmath}                                           % Supposedly, this allows one to use \eqref{} in \caption{}.
\MakeRobust{\eqref}

\usepackage{dsfont}

\usepackage{mathdots}
\usepackage{amsthm}                                                             % need for theorem-proof environment
\allowdisplaybreaks

\theoremstyle{plain}
\newtheorem{theorem}{Theorem}
\numberwithin{theorem}{section}

\newtheorem{lemma}[theorem]{Lemma}                              % [theorem] ==> theorems and lemmas will share a counter
\newtheorem{proposition}[theorem]{Proposition}
\newtheorem{keyestimate}[theorem]{Key Inequalities}
\newtheorem{corollary}[theorem]{Corollary}
\theoremstyle{definition}
\newtheorem{definition}[theorem]{Definition}

\newtheorem{notation}[theorem]{Notation}
\newtheorem{remark}[theorem]{Remark}

\def \sign { \varsigma}

\def \s {{\sigma}}

\def \b {{\beta}}

\def \R {\mathbb{R}}

\newcommand\Hc{\mathscr{H}}
\newcommand\Lc{\mathscr{L}}

\newcommand\eps{\varepsilon}

\newcommand\internal{f(\partial B)^{*}}
\newcommand\external{f(\partial B)_{*}}
\newcommand\dd{d}

\def \id {{\bf 1}_\mathbb{G}}

\def \R  {{\mathbb {R}}}
\def \x {{\xi}}

\def \eps {{\varepsilon}}

\def \b {{\beta}}

\def \It\^o {It\^o }
\def \s {{\sigma}}

\def \R {{\mathbb {R}}}
\def \N {{\mathbb {N}}}

\def \x {{\xi}}

\def \eps {{\varepsilon}}

\def \tilde {\widetilde}

\def \Ã  {{\`a }}
\def \Ã¨ {{\`e }}
\def \Ã² {{\`o }}
\def \Ã¹ {{\`u }}

\begin{document}

\title{A Yosida's parametrix approach to Varadhan's estimates for a degenerate diffusion under the weak H\"ormander condition
%Geman-Yor processes% and applications to arithmetic Asian options
}

\author{
Stefano Pagliarani
\thanks{Dipartimento di Matematica, Universit\`a di Bologna, Bologna, Italy.
\textbf{e-mail}: stefano.pagliarani9@unibo.it.} \and Sergio Polidoro
\thanks{Dipartimento di Scienze Fisiche, Informariche e Matematiche, Universit\`a di Modena e Reggio Emilia, Modena, Italy. \textbf{e-mail}: sergio.polidoro@unimore.it.}}

\date{This version: \today}

\maketitle

\begin{abstract}
We adapt and extend Yosida's parametrix method, originally introduced for the construction of the fundamental solution to a parabolic operator on a Riemannian manifold, to derive Varadhan-type asymptotic estimates for the transition density of a degenerate diffusion under the weak H\"ormander condition. This diffusion process, widely studied by Yor in a series of papers, finds direct application in the study of a class of path-dependent financial derivatives known as Asian options. We obtain the Varadhan formula
\begin{equation}
\frac{-2 \log p(t,x;T,y)  }  { \Psi(t,x;T,y) }   \to 1, \qquad \text{as } \quad T-t \to 0^+,  
\end{equation}
where $p$ denotes the transition density and $\Psi$ denotes the optimal cost function of a deterministic control problem associated to the diffusion. We provide a partial proof of this formula, and present numerical evidence to support the validity of an intermediate inequality that is required to complete the proof. We also derive an asymptotic expansion of the cost function $\Psi$, expressed in terms of elementary functions, which is useful in order to design efficient approximation formulas for the transition density. 
\end{abstract}

\noindent
%\textbf{Acknowledgements}: 

%%%%%%%%%%%%%%%%%%%%%%%%%%%%%%%%%%%%%%%
%
%       SECTION: Introduction - NEW
%
%%%%%%%%%%%%%%%%%%%%%%%%%%%%%%%%%%%%%%%

\section{Introduction}

We consider the system of It\^o stochastic differential equations (SDEs)

\begin{equation} \label{eq-model}
\begin{cases}
\dd X^1_s =  \s X^1_s \dd W_s,\qquad &X^1_t = x_1\\
\dd X^2_s = X^1_s \dd s,\qquad &X^2_t = x_2
\end{cases},
\end{equation}
where $\s$ is a positive constant. The study of this equation is motivated by the financial problem of pricing arithmetically averaged Asian options. Indeed, the process $\left(X_s^1\right)_{s \ge t}$ is a geometric Brownian motion starting from $x_1$ at $s=t$, and describes the evolution of the price of some financial asset in the Black-Scholes setting, while $\left(X_s^2\right)_{s \ge t}$ is the time-integral of $\left(X_s^1\right)_{s \ge t}$ and appears whenever the payoff of the option depends on the past average of the asset price. We refer to Section \ref{sec:asian} below for a more detailed discussion of the related applications.

Although the explicit expression of the solution to \eqref{eq-model} is well known, namely
\begin{equation}\label{eq:solution_exp}
\begin{cases}
X^1_s = x_1 e^{ \sigma (W_s-W_t) - \frac{\s^2}{2}(s-t) } \\
X^2_s = x_2 +  \int_t^s X^1_{\tau} \dd \tau %,\qquad &X^2_t = x_2
\end{cases},\qquad s>t,
\end{equation}
the analytical study of its density is more involved. We recall that Yor provided us with the following closed-form expression for the joint density of the process $X=(X^1,X^2)$ with $\sigma=1$ starting from $(1,0)$ at time $0$:
\begin{equation}\label{e-Yor-density}
 p(s,y_1, y_2) = \frac{e^{\frac{\pi^2}{2s}}}{\pi y_2^2 \sqrt{2 \pi s}}\exp\left( -\frac{1+y_1^2}{2y_2}\right)
q\left(s , \frac{{y_1}}{y_2}\right), 
\end{equation}
for $s,y_1,y_2>0$, where
\begin{equation}\label{psi}
  q\left(s ,  \eta \right)=\int_0^{\infty}e^{-\frac{\xi^2}{2 s}}e^{-\eta\cosh(\xi)}
  \sinh{(\xi)}\sin\left( \frac{\pi \xi}{s} \right)d \xi . 
\end{equation}
This expression was first obtained in \cite{yor2001some}. We also refer to \cite[Theorem 4.1]{matsumoto2005exponential} for a more compact statement and a concise proof. 
%We mainly refer to Yor's work (\cite{Yor5}, \cite{Yor4}, \cite{}), in particular to \cite[Theorem 4.1]{Yor1}. 
Note that the density $p$ is trivially null for non-positive values of $y_1$ and $y_2$, as the process $X^1$ is strictly positive. Furthermore, an elementary change of variable (see \eqref{eq:fund_sol_inv}-\eqref{eq:change_sigma} below) allows to write the transition density $p(t,x_1,x_2;s,y_1, y_2)$ of the process $X=(X^1,X^2)$ for a general $\sigma>0$ starting from $(x_1,x_2)$ at time $t$, which is constantly null if $y_1\leq 0$ or $y_2\leq x_2$.
%\cite{Yor5}, we also quote \cite{Yor4} for an exhaustive presentation of the topic. Other works are due by Matsumoto, Geman and Yor \cite{Yor1, Yor2, Yor3}, Carr and Schr\"{o}der \cite{Schroder1}, Bally and Kohatsu-Higa \cite{BallyHiga}.

Although \eqref{e-Yor-density} provides a semi-closed-form expression for the density, the latter is notoriously hard to handle from the numerical and analytical point of view. On the one hand, the numerical integration of the density as written in \eqref{e-Yor-density} is not accurate when the time $s$ is close to $0$, and this results in relevant errors when computing the price of financial options; we refer to Section \ref{sec:asian} for more details and appropriate references about this problem. On the other hand, it is difficult to extract relevant analytical information, such as asymptotic properties of the density for extreme values of $(y_1,y_2)$ and small values of time $s$. We undertake the study of these issues by following a different approach, namely we study the transition density as the fundamental solution of the associated backward Kolmogorov ultra-parabolic operator

% 
% 
% 
%\bibitem{BallyHiga}
% {\sc V.~Bally and A.~Kohatsu-Higa}, {\em Lower bounds for densities of {A}sian
%   type stochastic differential equations}, J. Funct. Anal., 258 (2010),
%   pp.~3134--3164.
% 
% \bibitem{Schroder1}
% {\sc P.~Carr and M.~Schr{\"o}der}, {\em Bessel processes, the integral of
%   geometric {B}rownian motion, and {A}sian options}, Teor. Veroyatnost. i
%   Primenen., 48 (2003), pp.~503--533.
%   
% \bibitem{Yor1}
% {\sc H.~Matsumoto and M.~Yor}, {\em Exponential functionals of {B}rownian
%   motion. {I}. {P}robability laws at fixed time}, Probab. Surv., 2 (2005),
%   pp.~312--347.
% 
% \bibitem{Yor2}
% \leavevmode\vrule height 2pt depth -1.6pt width 23pt, {\em Exponential
%   functionals of {B}rownian motion. {II}. {S}ome related diffusion processes},
%   Probab. Surv., 2 (2005), pp.~348--384.
%   
% \bibitem{Yor3}
% {\sc H.~Geman and M.~Yor}, {\em {B}essel {P}rocesses, {A}sian {O}ptions, and
%   {P}erpetuities}, Mathematical Finance, 3 (1993), pp.~349--375.
% 
% 
% \bibitem{Yor5}
% {\sc M.~Yor}, {\em On some exponential functionals of {B}rownian motion}, Adv.
%   in Appl. Probab., 24 (1992), pp.~509--531.
% 
% \bibitem{Yor4}
% \leavevmode\vrule height 2pt depth -1.6pt width 23pt, {\em Exponential
%   functionals of {B}rownian motion and related processes}, Springer Finance,
%   Springer-Verlag, Berlin, 2001.  

%In this paper we  
\begin{equation}\label{eq:operator_L}
\Lc ={\partial_t + x_1\partial_{x_2}}%_{=:Y}
 + \frac{\s^2 x_1^2}{2} \partial_{x_1 x_1}, \qquad (t,x_1,x_2)\in \R\times D,
\end{equation}
where we set $D: = \R^{+}\times\R$, with $\R^{+}:=]0,+\infty[$. %, and where $\sigma$ is a positive constant. 
The operator $\Lc$ is hypoelliptic in that it can be written in H\"ormander form as 
%\begin{equation}
$\Lc = Y-  \frac{\sigma}{2} Z +  \frac{1}{2} Z^2$
%\end{equation}
with the vector fields 
\begin{equation}\label{eq:field_Y}
Y = \partial_t + x_1\partial_{x_2},\qquad Z = \s x_1 \partial_{x_1}
\end{equation}
%which satisfy 
satisfying the H\"ormander condition (see Section \ref{sec:varadhan_estimates}). Note that the commutator $[Z,Y]$ is necessary in order to span the space $\R^3$. In the Probability community this is often referred to as \emph{weak H\"ormander condition}.

In \cite{cibelli2019sharp} the authors showed that the transition density of $(X^1,X^2)$ coincides with the fundamental solution $p=p(t,x;T,y)$ for $\Lc$, and derived upper and lower bounds in terms of the optimal cost function $\Psi$, defined as the solution of the control problem given by:
\begin{equation}\label{eq:control1}
 \Psi(t,x;T,y) = \min_{\omega%\atop \gamma(0)=x,\, \gamma(1)=y
}  \int_t^T |\omega(s)|^2 \dd s,
\end{equation}
where the minimum is taken over all controls $\omega\in L^2([t,T]) $ such that the problem
\begin{equation}\label{eq:optimal_curves}
\begin{cases}
\dot{\gamma}_1(s) =\sigma \omega(s) \gamma_1(s)\\
\dot{\gamma}_2(s) = \gamma_1(s)
%\Psi(0) = x, \ 
%\Psi(1) = y
\end{cases}
t<s<T,\qquad
\text{and}\quad
\begin{cases}
\gamma_1(t) = x_1, \ \gamma_1(T) = y_1\\
\gamma_2(t) = {x_2}, \ \gamma_2(T) = y_2
\end{cases}
\end{equation}
admits a solution. Such a control exists if and only if $x_2<y_2$. As usual in control theory, we agree to set $\Psi(z;w):=+\infty$ whenever $x_2>y_2$, in accordance with the fact that $p(t,x;T,y)$ is null. 

Roughly speaking, the main results in \cite{cibelli2019sharp} are an explicit representation for the cost function $\Psi$ and the following upper-lower bounds:
\begin{equation}\label{eq:cibellietal_simple}
\frac{c}{\s^2 y_1^2(T-t)^2} e^{-\frac{C}{ 2}\Psi(z;w)} \leq p(z;w)  \leq \frac{C}{\s^2 y_1^2(T-t)^2} e^{-\frac{c}{2}\Psi(z;w)}
\end{equation}
for some positive constants $C>>1$ and $c<<1$, and for $z=(t,x),w=(T,y)$ such that $t<T$ and $x_2<y_2$. 
We improve such bounds by proving that the constants appearing at the exponent in both sides of \eqref{eq:cibellietal_simple} are equal to one, namely 
\begin{equation}\label{eq:sim}
\log p(z;w) \sim -\frac{\Psi(z;w)}{2}, \qquad \text{as } \quad T-t \to 0^+,  
\end{equation}
where we agree to let 
\begin{equation}\label{eq:sim_def}
f\sim g\qquad \text{if}\quad \frac{f}{g}\to 1.
\end{equation}

For clarity we state two separate results for the upper and lower upper bounds, respectively. We anticipate that one step of the proof of such estimates relies on some bounds, precisely the Key Inequalities \ref{prop:estimate_CK} below, which we prove numerically in Section \ref{sec:numerical_evidence}. We refer to Section \ref{sec:Yosida_intro} for a preliminary discussion about the interpretation of such inequalities.
\begin{theorem}[Upper bound]\label{th:main}
%The operator $\Lc$ has a fundamental solution $\Gamma=\Gamma(t,x;T,y)$ in the sense of Definition \ref{} below. Moreover, 
%For any $\tau>0$, t
If the Key Inequalities \ref{prop:estimate_CK} hold, then for any $\tau>0$ there exists a positive constant $C>0$, only dependent on $\tau$ and 
$\sigma$, such that %which the following global upper bounds hold:
\begin{equation}\label{eq:main_estimate}
%\Gamma(z;w)  \leq C\,  \sqrt{ {\bf h}(z;w)}H_1(z;w) , 
p(z;w)  \leq C \, \frac{\sqrt{ {\bf h}(z;w)}+{\bf h}(z;w)}{\s^2 y_1^2(T-t)^2} e^{-\frac{1}{2}\Psi(z;w)},%\qquad z=(t,x,y),w=(T,y_1,y_2)\in \R\times D,\quad 0\leq t<T\leq T_0
\end{equation}
for any $z=(t,x_1,x_2),w=(T,y_1,y_2)\in \R\times D$ such that $T-\tau<
t<T$ and $x_2<y_2$, where 
%\begin{align}
%H_1(z;w) & :=    \frac{e^{-\frac{1}{2}\Psi(z;w)}}{\sqrt{(2\pi)^2\, \textrm{\emph{det}}{\bf C}(y_1,T-t)}}  , \quad {\bf C}(y_1,s) = y^2_1 \begin{pmatrix}
%    s & \frac{s^2}{2} \\
%    \frac{s^2}{2} & \frac{s^3}{3} \
%  \end{pmatrix},\qquad {\bf h}(z;w): = \frac{(T-t)\sqrt{x_1 y_1}}{y_2 - x_2},
%\end{align}
%and 
\begin{equation}\label{eq:h}
{\bf h}(z;w): = \frac{(T-t)\sqrt{x_1 y_1}}{y_2 - x_2},
\end{equation}
and
$\Psi$ is the cost function defined as the solution of the control problem above, whose explicit representation is given in \eqref{eq:Psi_explicit}.
\end{theorem}
\begin{theorem}[Lower bound]\label{th:lower}
If the Key Inequalities \ref{prop:estimate_CK} hold, then for any $\eps,\kappa>0$ % and $K$ compact subset of $D$
 there exists $\tau>0$, only dependent on $\eps$, $\kappa$ and 
$\sigma$, such that %which the following global upper bounds hold:
\begin{equation}\label{eq:main_estimate_lower}
%\Gamma(z;w)  \leq C\,  \sqrt{ {\bf h}(z;w)}H_1(z;w) , 
p(z;w)  \geq (1-\eps) \, \frac{{\bf h}(z;w)}{\s^2 y_1^2(T-t)^2} e^{-\frac{1}{2}\Psi(z;w)},%\qquad z=(t,x,y),w=(T,y_1,y_2)\in \R\times D,\quad 0\leq t<T\leq T_0
\end{equation}
for any $z=(t,x_1,x_2),w=(T,y_1,y_2)\in \R\times D$ such that $T-\tau<
t<T$, $x_2<y_2$ and $\frac{1}{\kappa}\leq  \frac{\sqrt{x_1 y_1}}{y_2 - x_2} < \kappa $.%$x_2<y_2$.
\end{theorem}

The results above are obtained by adapting to the strictly hypoelliptic setting the Yosida's parametrix method (\cite{yosida1953fundamental}) for the construction of the fundamental solution to parabolic-type operators. Specifically, we use the cost function $\Psi$ to define a parametrix function. %, i.e. an approximation of the fundamental solution of $\Lc$. 
We believe that this method might be of separate interest and we refer to Section \ref{sec:Yosida_intro} for further details and related references. 

The interest in the results above is duplex. From the theoretical point of view, they provide Varadhan-type estimates for a degenerate operator satisfying a weak H\"ormander condition. A detailed discussion about this aspect is deferred to Section \ref{sec:varadhan_estimates} below. Secondly, knowing the sharp choice of the constant in the exponent paves the way for developing numerically tractable asymptotic expansions of the transition density: we discuss this possible application in Section \ref{sec:asian}. 

With regard to the latter point, the computational aspects of the cost function $\Psi$ are also important. In \cite{cibelli2019sharp} the Authors solved the control problem \eqref{eq:control1}-\eqref{eq:optimal_curves} and provided an explicit representation of $\Psi$, which we report in Section \ref{sec:preliminaries}. Tough exact, such expression involves the inverse of hyperbolic trigonometric functions. In Section \ref{sec:novel_repres} we derive an expansion for the the cost function $\Psi$ whose terms only contain elementary functions, namely
\begin{align}
\Psi(z;\zeta) = \frac{4}{\s^2(T-t)}\bigg[ &{\bf h}(z,w)\bigg(\sqrt{\frac{y_1}{x_1}} + \sqrt{\frac{x_1}{y_1}} -2\bigg)  \\
& +  \sum_{n=2}^{\infty}   a_n {\bf 1}_{[1,\infty[} \big({\bf h}(z,w)\big)  \frac{ \big({\bf h}(z,w) -1\big)^n}{\big({\bf h}(z,w)\big)^{n-1} } +b_n {\bf 1}_{]0,1[}\big({\bf h}(z,w)\big)  \frac{ \big(-\log {\bf h}(z,w) \big)^n}{\big(1-\log {\bf h}(z,w)\big)^{n-2}}    \bigg],\\ \label{eq:rep_Psi_intro}
\end{align}
with ${\bf h}(z,w)$ as defined in \eqref{eq:h}, and where the coefficients $a_n,b_n$ can be computed recursively (see Proposition \ref{prop:coeff_a_b}). %We refer to Section \ref{sec:asian} for further details about the numerical implications of our results. 
A remarkable feature of such expansion is that it seems both point-wise and asymptotically convergent, in the sense that it approximates the asymptotic behavior of $\Psi(z;\zeta)$ as ${\bf h}(z,w)$ approaches $0$ and $+\infty$. In Section \ref{sec:novel_repres} we managed to prove point-wise convergence for ${\bf h}(z;\zeta)>0$ and asymptotic convergence as ${\bf h}(z,w)\to +\infty$.

\vspace{4pt}
After Sections \ref{sec:varadhan_estimates}, \ref{sec:Yosida_intro} and \ref{sec:asian} below, the remainder of the paper unfolds as follows. Section \ref{sec:preliminaries} contains some preliminaries about the cost function $\Psi$ and the fundamental solution of $\Lc$. In Section \ref{sec:novel_repres} we derive the representation \eqref{eq:rep_Psi_intro} and provide a partial proof of convergence. In Section \ref{sec:yosida} we prove Theorems \ref{th:main} and \ref{th:lower} by extending Yosida's parametrix method. Section \ref{sec:numerical_evidence} contains the numerical evidence supporting the validity of the Key Inequalities \ref{prop:estimate_CK}. Appendices \ref{sec:app_top_lemma} and \ref{sec:proof_lemma_HJB} contain, respectively, a topological lemma needed to prove the results of Section \ref{sec:novel_repres} and the proof of Lemma \ref{lem:HJB_eq} appearing in Section \ref{sec:yosida}.

\subsection{Varadhan-type estimates}\label{sec:varadhan_estimates}

In order to firm our results within the existing literature about fundamental solution, and density, estimates, let us consider the non-divergence-form second-order differential operator 
\begin{equation}\label{eq:opH}
\Hc =  \frac{1}{2} \sum_{i,j=1}^{n} a_{ij} (t,x) \partial_{x_i x_j} + \sum_{i=1}^{n} \mu_{i} (t,x) \partial_{x_i} + \partial_t,\qquad (t,x) \in \R \times \tilde D,
\end{equation}
with $\tilde{D}$ being a domain of $\R^n$. Clearly, the operator $\Lc$ in \eqref{eq:operator_L} is a particular instance of \eqref{eq:opH}. From the probabilistic stand-point, under suitable assumptions, $\Hc $ is the extended generator of a solution to the following It\^o SDE:
\begin{equation} \label{eq-model_gen}
%\begin{cases}
\dd X_s =  \mu(t,X_s) \dd s + \s(t,X_s)  \dd W_s,\qquad %&
X_t = x , %\\
%\dd X^2_s = X^1_s \dd s,\qquad &X^2_t = x_2
%\end{cases}.
\end{equation}
with $W$ being a $n$-dimensional Brownian motion, and $\sigma$ a $n\times n$ matrix such that $\sigma \sigma^\top = a= (a_{i,j})_{i,j=1,\cdots,n}$. Also, the transition density of $X$, hereafter denoted by $p=p(t,x;T,y)$, coincides with the fundamental solution of $\Hc$.
%Hereafter, we also denote by $p=p(t,x;T,y)$ %, for $z=(t,x),w=(T,y)$, 
%a fundamental solution of $\Hc$, which corresponds %, under suitable assumptions%(see for instance \cite{pascucci_book}), 
%to the transition density of the process $X$. 

In the uniformly parabolic case, a fairly general classical result  (see \cite{friedman-parabolic}) states that
\begin{equation}\label{eq:estimates_parametrix_parabolic}
\frac{c}{(T-t)^{n/2}} e^{-\frac{C|y-x|^2}{2 (T-t)}} \leq p(t,x;T,y)  \leq \frac{C}{(T-t)^{n/2}} e^{-\frac{c|y-x|^2}{2 (T-t)}}, \qquad t<T, \quad x,y\in \R^n,
\end{equation}
%for any $z=(t,x),\ w=(t,T) \in \R \times \R^d$ with $t<T$, 
where $c,C$ are positive constants independent of $(t,x),(T,y)$. These estimates can be proved under boundedness and H\"older regularity assumptions on the coefficients of $\Hc$, and assuming a uniform ellipticity condition on the second-order coefficients $a_{ij}$. More precisely, the upper-bound above can be obtained by means of the so-called \emph{parametrix method}. With the same method it is possible to prove the lower bound in \eqref{eq:estimates_parametrix_parabolic}, but only locally in space: the global version can be achieved with the method of the Harnack chains, introduced by Aronson in \cite{aronson1967bounds}, which is based on the Harnack inequality.  Though very general, the bounds in \eqref{eq:estimates_parametrix_parabolic} are not sharp enough to characterize the asymptotic behavior for small times of the fundamental solution $p$ away from the pole. %Indeed, $p$ decays exponentially and the constante
In the groundbreaking paper \cite{varadhan1967behavior} Varadhan proved, for $\mu\equiv 0$ and $a(t,x)= a(x)$, %again in the uniformly parabolic case and with time-independent coefficients, 
that 
%the lower-bound that can be proved with the parametrix method is only local in space, 
%The latter is an iterative method, introduced by Levi in \cite{}, which allows to construct the fundamental solution by successive approximations, starting from a kernel, called \emph{parametrix function}, given by the fundamental solution o
%Comparison with the literature
\begin{equation}\label{eq:varadhan}
\lim_{T-t\to 0^+}2 (T-t)\, {\log p(t,x;T,y)} = - d^2 (x,y),
\end{equation}
uniformly with respect to $x,y$ over a compact subset of $\R^n$, where $d$ represents the geodesic distance with respect to the Riemann metric $a^{-1}(x) dx_i dx_j  $, namely
\begin{equation}\label{eq:riemann_distance}
d(x,y) = \inf_{\gamma: \gamma(0)=x, \, \gamma(1)=y} \int_{0}^{1}  \sqrt{ \big\langle a^{-1}\big(\gamma(s)\big) \dot\gamma(s) , \dot\gamma(s) \big\rangle }\, ds.
\end{equation}
The assumptions in \cite{varadhan1967behavior} are, again, uniform ellipticity and H\"older continuity for the second-order coefficients. Note that \eqref{eq:varadhan} is more precise than \eqref{eq:estimates_parametrix_parabolic}, in that it yields the exact asymptotic behavior of the logarithm of $p$ for small times. In the subsequent paper \cite{varadhan1967diffusion}, building upon \eqref{eq:varadhan}, Varadhan established a large-deviation principle for elliptic diffusions of the form \eqref{eq-model_gen}. The cornerstone contributions \cite{varadhan1967behavior}-\cite{varadhan1967diffusion} initiated a whole stream of literature, dealing with asymptotic expansions of the transition densities (or heat-kernels in the context of PDEs) of diffusion processes on Riemannian manifolds. Such expansions aim at characterizing the full asymptotic behavior of $p(t,x;T,y)$, under various assumptions on the underlying geometry, by adding up additional terms to the leading one given by \eqref{eq:varadhan}. For instance, following the WKB (Wentzel-Kramers-Brillouin) method, one seeks representations of the type
\begin{equation}\label{eq:full_exp}
p(t,x;T,y) = \frac{\exp\Big(\! -\frac{d^2(x,y)}{2(T-t)} \Big) }{(T-t)^{n/2}}%e^{-\frac{d(x,y)}{2(T-t)}} 
\Big(\alpha_0(x,y) + (T-t) \alpha_1(x,y) + (T-t)^2 \alpha_2(x,y) + \cdots \Big).
\end{equation}
We refer, for instance, to \cite{molchanov1975diffusion} and \cite{MR770974} for some relevant contributions to this field, which mainly developed in the years 1970s and 1980s. 
 We also mention the important contribution of Freidlin and Wentzell (e.g. \cite{ventsel1970small}) to the subject of large deviations for diffusion processes. In the last two decades, these asymptotic techniques were employed for the study of volatility models in mathematical finance (see \cite{friz2015large} and the references therein).

In the late 1980s these results were generalized to the hypoelliptic setting, under the so-called \emph{strong H\"ormander condition}. To explain these generalizations, it is useful to write the operator $\Hc$ above (with time-independent coefficients) in H\"ormander form, namely
\begin{equation} \label{e-heatkernel}
\Hc =\partial_t + Z_0 + \frac{1}{2} \sum_{j=1}^m Z_j^2,
\end{equation}
%in $\R\times \R^{N} $, 
where $Z=(Z_0,Z_1, \dots, Z_m)$ is a system of %smooth 
vector fields defined on a domain $\tilde{D} \subset \R^n$ satisfying the strong H\"ormander's condition
\begin{equation}\label{strong_Hor}
 {\rm Lie} (Z_1, \ldots, Z_m)(x) = \R^n, \qquad x \in \tilde D,
\end{equation}
%for every $x \in \tilde D$, 
where ${\rm Lie} (Z_1, \ldots, Z_m)$ is the Lie algebra generated by the vector fields $Z_1, \ldots, Z_m$, namely the vector space spanned by $Z_1, \ldots, Z_m$ and their commutators. Note that, if $\Hc $ represents the generator of $X$ in \eqref{eq-model_gen}, then the vector field $Z_j$, $j=1,\cdots,m$, is exactly identified by the $j$-th column of the diffusion matrix $\sigma$. 
Under the assumption of smooth vector fields, Leandre (\cite{leandre1987majoration}, \cite{leandre1987minoration}) proved that \eqref{eq:varadhan} remains valid in this setting, with $d(x,t)$ being the Carnot-Caratheodory \emph{sub-Riemannian} metric induced by the vector fields $Z_1, \ldots, Z_m$. In \cite{ben1988developpement} also \eqref{eq:full_exp} was extended to this setting.
 A contribution in this direction was also given in \cite{de2018local}.
 
In order to discuss extensions to strictly hypoelliptic settings (when \eqref{strong_Hor} fails), it is crucial the following
\begin{remark}\label{rem:varadhan_cost_func}
Both in the Riemannian and sub-Riemannian case, \eqref{eq:varadhan} can also be written in the form of \eqref{eq:sim}, where $\Psi$ represents here the solution of the optimal control problem given by \eqref{eq:control1}, with the minimum taken over all controls $\omega\in L^2([t,T])$ such that the problem
\begin{equation}\label{eq:optimal_curves_gen}
\begin{cases}
\dot{\gamma}(s) = \omega(s) \sigma \big(\gamma(s)\big),\quad%\\
%\dot{\gamma}_2(s) = \gamma_1(s)
%\Psi(0) = x, \ 
%\Psi(1) = y
%\end{cases}
t<s<T,\\%\qquad
%\text{and}\quad
%\begin{cases}
\gamma(t) = x, \ \gamma(T) = y%\\
%\gamma_2(t) = {x_2}, \ \gamma_2(T) = y_2
\end{cases}
\end{equation}
admits a solution. Indeed,
%\eqref{eq:varadhan} can be written as \eqref{eq:sim}. Indeed, 
it is a standard result (see for instance \cite[Lemma 2.2]{varadhan1967diffusion} for the elliptic case) that 
\begin{equation}\label{eq:rel_distance_Psi}
\frac{d^2(x,y)}{T-t} = \Psi(t,x;T,y).
\end{equation}
\end{remark}
%As already mentioned in the beginning of the introduction, the operator $\Lc$ only satisfies the 
Under the so-called \emph{weak H\"ormander} condition, namely 
\begin{equation}\label{eq:weak_Hor}
{\rm Lie} (\partial_t + Z_0, Z_1, \ldots, Z_m)(x) = \R^{n+1}, \qquad x \in \tilde D,
\end{equation}
%\begin{equation}\label{eq:weak_Hor}
%{\rm Lie} (Z_0, Z_1, \ldots, Z_m, [Z_0,Z_1],\cdots, [Z_0,Z_m])(x) = \R^{d}, \qquad x \in \tilde D,
%\end{equation}
%While this condition ensures 
the hypoellipticity of the operator $\Hc$ is preserved, but %by opposite with respect to the strong H\"ormander condition, 
the second-order vector fields $Z_1, \ldots, Z_m$ and their commutators are no longer enough to span the space $\R^n$. Therefore, there is no sub-Riemannian metric on $\R^n$ and an estimate in the form of  \eqref{eq:varadhan} can be no longer achieved. 
%\begin{equation}
%\exp\bigg( {-\frac{1}{2}\frac{d^2_Z(x,y)}{T-t}} \bigg).
%\end{equation}
%For instance, we refe 
%are also known as heat-kernes in the context of PDEs. 
%Leandre \cite{leandre1987majoration}-\cite{leandre1987minoration}
%\cite{ben1988developpement} full-expansion.
The cost function $\Psi$, however, remains well defined as the solution of the same control problem \eqref{eq:control1} where %with the minimum taken over all controls $\omega\in L^2([t,T])$ such that the problem
\begin{equation}\label{eq:optimal_curves_gen_bis}
\begin{cases}
\dot{\gamma}(s) = \mu\big(  \gamma(s)  \big) +  \omega(s) \sigma \big(\gamma(s)\big),\quad%\\
%\dot{\gamma}_2(s) = \gamma_1(s)
%\Psi(0) = x, \ 
%\Psi(1) = y
%\end{cases}
t<s<T,\\%\qquad
%\text{and}\quad
%\begin{cases}
\gamma(t) = x, \ \gamma(T) = y%\\
%\gamma_2(t) = {x_2}, \ \gamma_2(T) = y_2
\end{cases}
\end{equation}
replaces \eqref{eq:optimal_curves_gen}.   Therefore, in light of Remark \ref{rem:varadhan_cost_func}, \eqref{eq:sim} appears as the natural generalization of \eqref{eq:varadhan} to strictly hypoelliptic settings. %latter thus appear as a natural generalization of 
In this sense, Theorems \ref{th:main} and \ref{th:lower} above, which lead to \eqref{eq:sim}, can be viewed as Varadhan-type estimates for the degenerate parabolic operator $\Lc$. As already mentioned above, \eqref{eq:sim} sharpens the estimates proved in \cite{cibelli2019sharp} for the fundamental solution of $\Lc$. Furthermore, we are not aware of other Varadhan-type formulas in the context of hypoelliptic operators under the weak H\"ormander condition.

Note that the drift coefficient $\mu$ in \eqref{eq:optimal_curves_gen_bis}, which is not controlled, is needed to ensure the existence of a path connecting $x$ and $y$. %system \eqref{eq:optimal_curves_gen_bis} in order to garante that ... 
%For instance, the system \eqref{eq:optimal_curves} associated to the operator $\Lc$ (which satisfies \eqref{eq:weak_Hor}) is controllable and the optimal cost function $\Psi$ was computed in \cite{cibelli2019sharp}. 
%This is a well-known fact, which reflects the different time-scales of the single components of the underlying diffusion.
%As already mentioned in the beginning of the introduction, the operator $\Lc$ only satisfies the weak H\"ormander condition, namely 
%\begin{equation}
%{\rm Lie} (Z_0, Z_1, \ldots, Z_m)(x) = \R^d, \qquad x \in \tilde D.
%\end{equation}
%While this condition ensures the hypoellipticity of the operator $\Hc$, %by opposite with respect to the strong H\"ormander condition, 
%the second-order vector fields $Z_1, \ldots, Z_m$ and their commutators are no longer enough to span the whole space. Therefore, there is no sub-Riemannian metric and \eqref{eq:sim} cannot be cast in the form \eqref{eq:varadhan}. 
Also, the fact that \eqref{eq:rel_distance_Psi} does not hold in general is evident as $\Psi(t,x;T,y)$ can exhibit different rates of explosion, as $T-t \to 0^+$, depending on the choice of $(x,y)$. This is a well-known phenomenom, which reflects the different time-scales of the single components of the underlying diffusion. A stylized example is given by the stochastic Langevin equation 
\begin{equation} \label{eq-model_Langevin}
\begin{cases}
\dd X^1_s = \sigma  \dd W_s,\qquad &X^1_t = x_1\\
\dd X^2_s = X^1_s \dd s,\qquad &X^2_t = x_2
\end{cases}
\end{equation}
with $\sigma$ positive constant,
whose generator is given by 
\begin{equation}\label{eq:kolmogorov_constant}
\Hc = {\partial_t + x_1\partial_{x_2}}%_{=:Y}
 + \frac{a}{2} \partial_{x_1 x_1}, \qquad a=\sigma^2.
\end{equation}
%$\Hc = {\partial_t + x_1\partial_{x_2}}%_{=:Y}
% + \frac{1}{2} \partial_{x_1 x_1}$, 
Its fundamental solution is given exactly by 
 \begin{align}\label{eq:density_langevin}
p_{\sigma}(t,x;T,y)  &=  \frac{1}{2\pi \sqrt{\text{det}\, {\bf C}(\sigma ,T-t)  }}\exp\bigg({ -\frac{1}{2}\Psi_{\sigma}(t,x;T,y)} \bigg),  \\
{\bf C}(\sigma,s) &:= \sigma^2 \begin{pmatrix}
    s & - \frac{s^2}{2} \\
    - \frac{s^2}{2} & \frac{s^3}{3} \
  \end{pmatrix}, \label{eq:matrix_C_intro}
\end{align}
where  
 \begin{align} \label{eq:quadratic_form}
\Psi_{\sigma}(t,x;T,y) & =   \big\langle {\bf C}^{-1}(\sigma,T-t) \big( x_1 - y_1, x_2 - y_2 + (T-t) y_1 \big)      ,  \big(  x_1 - y_1 , x_2 - y_2 + (T-t) y_1  \big)  \big\rangle   \\
 & = \frac{3 \big(2 (y_2 -  x_2) - (T-t) (x_1+y_1)\big)^2}{\sigma^2(T-t)^3}+\frac{(y_1-x_1)^2}{\sigma^2(T-t)}
\end{align}
is the cost function of the associated control problem (see for instance \cite[Example 9.53]{pascucci2011pde}).

%or equivalently
%\begin{equation}\label{eq:varadhan_bis}
%\log p(t,x;T,y) \sim -\frac{d^2(x,y)}{2 (T-t)}\qquad \text{as } T-t\to 0^+ ,
%\end{equation}
%in accordance with \eqref{eq:sim_def}. 
%
%We also consider the optimal control problem, which generalizes -\eqref{eq:optimal_curves}, given by 
%\begin{equation}\label{eq:control1_gen}
% \tilde\Psi(t,x;T,y) = \min_{\omega%\atop \gamma(0)=x,\, \gamma(1)=y
%}  \int_t^T |\omega(s)|^2 \dd s,
%\end{equation}
%where 
%

\subsection{Yosida's parametrix}\label{sec:Yosida_intro}

In order to obtain the bounds in Theorems \ref{th:main} and \ref{th:lower}, we adapt and extend the method introduced by Yosida in his seminal paper \cite{yosida1953fundamental}, where he outlined a geometrical variation of the classical \emph{Levi's parametrix method} for the construction of the fundamental solution of a parabolic-type operator $\Hc$ of type \eqref{eq:opH} on a Riemannian manifold. Such method, which was published 14 years before Varadhan's paper \cite{varadhan1967behavior}, already contained the idea that an expansion like \eqref{eq:full_exp} should hold, and in particular that the logarithm of the fundamental solution $p$ should asymptotically behave like in \eqref{eq:varadhan}. In spite of this, Yosida's method did not become a standard tool in the study of the asymptotic properties of transition densities. We suspect this is partially due the fact that his construction was mainly heuristic: no precise assumptions on the coefficients were given to ensure convergence, nor upper/lower bounds leading to \eqref{eq:varadhan} were proved in \cite{yosida1953fundamental}. In \cite{gatheral2012asymptotics}, Gatheral et al. brought Yosida's method to the attention of the mathematical finance community emphasizing its versatility from the computational point of view, with a particular focus on its ability to deal with time-dependent coefficients. In this paper, inspired by \cite{gatheral2012asymptotics}, we push Yosida's method one step further and adapt it to the strictly hypoelliptic setting. In Section \ref{sec:yosida} we perform the construction of the fundamental solution $p$ for the operator $\Lc$ in \eqref{eq:operator_L}, and prove the estimates \eqref{eq:main_estimate}-\eqref{eq:main_estimate_lower} which lead to \eqref{eq:sim}. Although we treat here a special case, we claim that the principles of this approach can be employed in general for a wider class of degenerate parabolic operators under the weak H\"ormander condition, to establish full expansions analogous to \eqref{eq:full_exp}. The study of general assumptions on the coefficients, possibly time-dependent, under which this method yields %results like \eqref{eq:full_exp} 
rigorous results is subject of ongoing investigation. 
%In future research, we plan to study general assumptions on the coefficients, possibly time-dependent, under which results like \eqref{eq:full_exp} can be established by employing this method. %Furthermore, we plan to consider in the future full expansions like 

Consider the parabolic operator $\Hc$ in \eqref{eq:opH}, and assume for simplicity %$\mu(t,x)\equiv 0$ and ]
$\mu\equiv 0$ and $a_{ij}(t,x)=a(x)$. In the Yosida's parametrix, the key difference with respect to Levi's original method lies in the choice of the \emph{parametrix function} (the starting point of the iterative construction), which is dependent on the geodesic distance $d$ induced by the second-order coefficients of the differential operator. Precisely, in Yosida's method the kernel that defines the parametrix function is set as 
\begin{equation}\label{eq:param_yos}
\frac{(T-t)^{-n/2}}{ \sqrt{ (2\pi)^n\, \text{det}\, a(y)}}\exp\bigg( {-\frac{1}{2}\frac{d^2(x,y)}{T-t}} \bigg),
\end{equation}
where $d$ is the Riemannian distance \eqref{eq:riemann_distance} induced by the coefficients $a_{ij}$. By opposite, in Levi's method the parametrix function 
is the fundamental solution of the constant coefficients operator obtained from $\Hc$ by freezing the coefficients at $y$. In other words, the distance $d$ in \eqref{eq:param_yos} is replaced by the geodesic distance taken with respect to the constant metric tensor $a^{-1}(y) dx_i dx_j$, namely the kernel is
\begin{equation}
\frac{(T-t)^{-n/2}}{ \sqrt{(2\pi)^n\, \text{det}\, a(y)}}\exp\bigg( {-\frac{1}{2}\frac{d_y^2(x,y)}{T-t}} \bigg), \qquad \text{where }\ d^2_y(x,y) = \langle  a^{-1}(y)(y-x), y-x  \rangle. 
\end{equation}
The idea is that the choice \eqref{eq:param_yos} provides us with sharp asymptotic estimates of the fundamental solution for small times, while preserving the correct behavior of the kernel near the singularity (the starting point of the stochastic process). As already mentioned in the previous subsection, the second-order coefficients of the operator $\Lc$ do not induce a distance on $\R^2$. However, proceeding in analogy with \eqref{eq:param_yos}-\eqref{eq:rel_distance_Psi}, our idea is to consider the following kernel for the parametrix function:  
\begin{equation}\label{eq:H1_intro}
H_1(t,x;T,y):= \frac{1}{2\pi \sqrt{\text{det}\, {\bf C}(\s y_1,T-t)  }}\exp\bigg({ -\frac{1}{2}\Psi(t,x;T,y)} \bigg),
\end{equation}
where the matrix ${\bf C}$ is set as in \eqref{eq:matrix_C_intro} 
%\begin{equation}
%{\bf C}(y_1,s) = \sigma^2 y^2_1 \begin{pmatrix}
%    s & - \frac{s^2}{2} \\
%    - \frac{s^2}{2} & \frac{s^3}{3} \
%  \end{pmatrix}
%\end{equation}
in order to yield the right behavior near the singularity. In particular, the choice for ${\bf C}$ can be explained as follows. First note that %consider In particular, denoting by $e^{\cdot Y} z$ 
the integral curve of the vector field $Y$ in \eqref{eq:field_Y}, starting at $z=(t,x)\in \R\times D$, %, which 
is given by
\begin{equation}
e^{\delta Y} z = (t + \delta, x_1, x_2+ \delta x_1),
\end{equation}
which yields (just set $\omega\equiv 0$ in \eqref{eq:optimal_curves})
\begin{equation}
\Psi\big(t, y_1 ,  y_2 - (T-t) y_1; T , y_1, y_2\big) = 0, \qquad T-t>0,\ y\in D.
\end{equation}
As it turns out, expanding now $\Psi(t,\cdot;T,y)$ around $(y_1 ,  y_2 - (T-t) y_1)$, at second order, one obtains 
\begin{equation}\label{eq:taylor_Psi}
\Psi(t,x;T,y) = %\big\langle {\bf C}^{-1}(y_1,T-t) \big( x_1 - y_1, x_2 - y_2 + (T-t) y_1 \big)      ,  \big(  x_1 - y_1 , x_2 - y_2 + (T-t) y_1  \big)  \big\rangle 
\Psi_{\sigma  y_1}(t,x;T,y) 
%\\ & 
+ o\big(   |x_1 - y_1|^2 + |x_2 - y_2 + (T-t) y_1|^2  \big), \qquad \text{as } (x_1,x_2) \to  (y_1 ,  y_2 - (T-t) y_1),
\end{equation}
where $\Psi_{\sigma y_1}(t,\cdot;T,y)$ is the quadratic form in \eqref{eq:quadratic_form}. Therefore, the term 
$$1/ \sqrt{\text{det}\, {\bf C}(\s y_1,T-t)  }$$ in \eqref{eq:H1_intro} ensures a Gaussian behavior for $H_1$ near the diagonal. In particular, one can show (see Proposition \ref{prop:delta_Dirac}) that $H(t,x;T,y)\to \delta_x$ as $t\to T^-$. Note that, replacing $\Psi$ with %the quadratic form in \eqref{eq:taylor_Psi}
$\Psi_{\sigma y_1}$ in \eqref{eq:H1_intro}, the function $H_1$ becomes the fundamental solution of the operator $\Hc$ in \eqref{eq:kolmogorov_constant} with $a = \sigma^2 y^2_1$. Such a kernel can be seen as the counterpart of Levi's parametrix kernel for degenerate \eqref{eq:kolmogorov_constant}-like operators with variable coefficients, and allows to prove (see \cite{polidoro1994class}, \cite{francesco2005class}) Gaussian bounds for the fundamental solution when $a$ is a positive uniformly H\"older-continuous function, which is bonded and bounded away from zero.

The parametrix method is an iterative procedure to constructs the fundamental solution taking a parametrix function as a leading term, and then representing the remainder as a series whose terms can be recursively determined via successive approximations based on Duhamel's principle. A key element to prove the convergence of the series consists in estimating the convolution of the parametrix function with itself. In the original Levi's method for parabolic operators, the parametrix is a Gaussian function and thus one can rely on the Chapman-Kolmogorov equation. In the case of Yosida's parametrix, again in the parabolic case, a locally isometric change of coordinates can transform the kernel \eqref{eq:param_yos} into a Gaussian function and thus the convolutions can be estimated once more by means of the Chapman-Kolmogorov identity. In our case, the bounds we need in order to complete the parametrix construction are those in the Key Inequalities \ref{prop:estimate_CK}, which are estimates that look roughly like
\begin{equation}
\int_{\R^+} \int_{x_2}^{y_2} H_1(t,x;s,\xi) H_1(s,\xi;T,y)  \dd \xi  \leq C_{T-t} H_1(t,x;T,y),\qquad s\in ]t,T[,
\end{equation}
uniformly in $x,y\in ]0,+\infty[\times \R$ with $x_2<y_2$. The fact that this bound holds locally in $x,y$ is clear as $H_1$ tends to a Dirac delta for small times. At the moment we only have numerical evidence (reported in Section \ref{sec:numerical_evidence}) for the fact that it holds uniformly. The techniques used in the Riemannian and sub-Riemannian case to prove the Chapman-Kolmogorov identity seem to fail here. Therefore, although the numerical evidence reported in Section \ref{sec:numerical_evidence} strongly supports the validity of such estimates, the problem of providing a rigorous proof remains an open problem.

%The same approach has been used in the study of heat kernels related to subelliptic operators in the form 
%\begin{equation} \label{e-heatkernel}
% \partial_t + \frac{1}{2} \sum_{j=1}^m Z_j^2
%\end{equation}
%in $\R\times \R^{N} $, where $Z=(Z_1, \dots, Z_m)$ is a system of smooth vector fields defined on an open subset $\Omega \subset \R^N$ satisfying the H\"ormander's conditions
%\begin{equation*}
% {\rm Lie} (Z_1, \ldots, Z_m)(x) = \R^N
%\end{equation*}
%for every $x \in \omega$, where ${\rm Lie} (Z_1, \ldots, Z_m)$ is the Lie algebra generated by the vector fields $Z_1, \ldots, Z_m$, which is the vector space generated by the vector fields $Z_1, \ldots, Z_m$ and their commutators. In this case the Remannian distance is replaced by the \emph{sub-Riemannian}  Carnot-Caratheodory distance $d_Z$, which would replace the Riemannian distance $d$ in the choice of the parametrix kernel \eqref{eq:param_yos}.
%\begin{equation}
%\exp\bigg( {-\frac{1}{2}\frac{d^2_Z(x,y)}{T-t}} \bigg).
%\end{equation}
%When considering our operator $\L$ such distance does not exists as $\L$ only satisfies the weak H\"ormander condition. In this case we replace the geodesic distance with the cost function $\Psi$. Note that, in the case of operators in the form \eqref{e-heatkernel}, we have
%\begin{equation*}
% \Psi(x,t;\xi, \tau) = \frac{d_Z(x, \xi)^2}{\tau - t},
%\end{equation*}
%then $\Psi$ can be considered a generalized distance, however in general it is not symmetric.

\subsection{Application to arithmetic Asian options}\label{sec:asian}

The It\^o process in \eqref{eq:solution_exp} finds a direct application in the problem of pricing and hedging of a class of path-dependent financial derivatives known as arithmetic Asian options. The first component $X^1$, which is a geometric Brownian motion, denotes the evolution price of a risky asset in the well-known Black-Scholes model under the risk-neutral probability measure. For simplicity, it is assumed here zero risk-free interest rate so that the risk-neutral dynamics of the asset is exactly given by \eqref{eq-model}, namely $X^1$ is an exponential Brownian martingale. The process $s\mapsto \frac{1}{s} X^2_s$ represents instead the continuous arithmetic time-average of the risky asset $X^1$. An arithmetic average Asian option is a random variable of the form
\begin{equation}
\varphi\Big( X^1_T, \frac{1}{T} X^2_T\Big),
\end{equation}
where $\varphi$ is a payoff function that determines the value of the financial claim at a given maturity $T>0$. A variety of payoff functions can be considered, some popular choices being for example%, some popular choices are given by the pa
\begin{align}
\varphi(x_1,x_2) &= (x_1 - x_2)^+, && \text{(floating-strike Call)}      \\
\varphi(x_1,x_2) &= (x_2 - K)^+ , \quad K>0.  &&   \text{(fixed-strike Call)}
\end{align}
Within the theory of continuous-time arbitrage pricing, the no-arbitrage price at time $t<T$ of the such derivatives, given the initial values $x=(x_1,x_2)$ for $X=(X^1,X^2)$, are given by the risk-neutral evaluation 
\begin{equation}\label{eq:pricing}
\mathbb{E}_{t,x}\big[\varphi\big( X^1_T,  X^2_T / T\big)\big] = \int_{\R_{>0}} \int_{x_2}^{+\infty}  p(t,x_1,x_2;T,y_1,y_2) \varphi\big( y_1,  y_2 / T\big)    \dd y_2 \dd y_1 .
\end{equation}
While the expected value above determines the price of the claim, its derivatives with respect to the variables $t,x,T$ and the parameter $\sigma$, also known as sensitivities, are related to the hedging strategies of such claim. %Besides the degenerate structure of the diffusion \eqref{eq-model} (the Brownian motion directly acts only on the first component) and the consequent degeneracy in the Kolmogorov operator $\Lc$, 
In the evaluation of \eqref{eq:pricing} one encounters computational difficulties due the involved expression of the transition density $p$. Intuitively, this is largely due to the fact that the problem is somehow \emph{ill-posed}, meaning that $X^2$ is an arithmetic average of a geometric Brownian motion. For instance, the integral representation \eqref{e-Yor-density} given by Yor is of limited practical use in the numerical computation of \eqref{eq:pricing}.

In the last decades, the pursue of the efficient methods to approximate the integral in \eqref{eq:pricing} has challenged many authors, of whom we give here an incomplete account. In \cite{GemanYor1992}, Geman and Yor gave an explicit representation of Asian option prices in terms of the Laplace transform of hypergeometric functions. However, several authors (see for instance \cite{FuMadanWang1998} and \cite{Dufresne2002}) pointed out the difficulty of pricing Asian options with short maturities or small volatilities using the analytical method in \cite{GemanYor1992}. This is also a disadvantage of the Laguerre expansion proposed in \cite{Dufresne2000}. %, \cite{Dufresne2001}
In %\cite{Shaw2000},
\cite{Shaw2003} a contour integral approach was employed to improve the accuracy in the case of low volatilities, though at a higher computational cost. In \cite{Linetsky2004} the problem was tackled using the spectral theory of singular Sturm-Liouville operators, yielding a series formula that gives very accurate results. However, again for low volatilities, the convergence might be slow and the method becomes computationally expensive. We also mention the Monte Carlo approach, which was taken by a number of authors to price efficiently Asian options under the Black-Scholes model (see \cite{GuasoniRobertson2008} among others).  In the Black-Scholes model
and for special homogeneous payoff functions, it is possible to reduce
the study of Asian options to a PDE with only one state variable. We refer to this approach as to PDE
reduction, which was followed in \cite{Ingersoll}, \cite{DewynneShaw2008} and \cite{CaisterOHaraGovinder2010} among other works.

The results of this paper pave the road for deriving sharp asymptotic expansions for prices and sensitivities through an expansion of the transition density $p$ of the form
\begin{equation}\label{eq:full_exp_2}
p(t,x;T,y) = %\frac{\exp\Big(\! -\frac{\Psi(t,x;T,y)}{2} \Big) }{(T-t)^2}%e^{-\frac{d(x,y)}{2(T-t)}} 
H_1(t,x;T,y) {\bf u}(t,x;T,y)
\Big( 1 + (T-t) \alpha_1(t,x;T,y) + (T-t)^2 \alpha_2(t,x;T,y) + \cdots \Big),
\end{equation} 
with $H_1$ as in \eqref{eq:H1_intro}, ${\bf u}$ as determined in Proposition \ref{th:u}, and where the further correction terms $\alpha_1, \alpha_2, \cdots$ have to be determined through recursive procedure as in \cite{yosida1953fundamental}. In order to obtain approximations of \eqref{eq:pricing} that are computationally tractable, the expansion \eqref{eq:rep_Psi_intro} of $\Psi$ in terms of elementary functions, which we derive in Section \ref{sec:novel_repres}, plays an essential role as it allows to avoid the numerical inversion of the hyperbolic trigonometric functions that appear in the closed-form representation \eqref{eq:Psi_explicit}. The expansion \eqref{eq:full_exp_2} can be seen as a refinement of the asymptotic expansions derived in \cite{foschi2013approximations}, \cite{pagliarani2017intrinsic}, which are based on a perturbation procedure that approximates, at leading order, the Black-Scholes dynamics \eqref{eq-model} with some Langevin dynamics like in \eqref{eq-model_Langevin}. The result of this procedure is an expansion for $p$ of the form \eqref{eq:full_exp_2} whose leading term is %not $H_1 {\bf u}$ but 
a Gaussian transition density  as in \eqref{eq:density_langevin}. A similar approach was taken in \cite{gobet2014weak} to obtain analytical approximations by means of Malliavin calculus. 

In general, analytical approaches based on asymptotic expansions exhibit several advantages in that they provide approximations in closed form, which are fast to compute and display an explicit dependence on the parameters.  Other asymptotic methods for Asian options were studied in \cite{ShirayaTakahashi2010} and \cite{ShirayaTakahashiToda2009} by Malliavin calculus techniques. Within this stream of literature we also mention the results in \cite{pirjol2016short}, where sharp asymptotic expansions for \emph{fixed} and \emph{floating-strike} Asian options are derived. The latter approach is related to ours, as it exploits Varadhan's principle of large deviation (\cite{varadhan1967diffusion}), together with a contraction principle to deal with the arithmetic average.

%Among other pricing methods for more general dynamics the price process $X^1$ are those in %\cite{HubalekSgarra2011}  \cite{Dufresne2001CEV} and \cite{}.

%The reduced PDE formulation was used by Dewynne and Shaw \cite{DewynneShaw2008} to derive accurate
%approximation formulae for Asian-rate Call options in the BS model by a matched asymptotic
%expansion.  

%We assume the following hypotheses to be in force. % $a=(a_{i,j})$ being a positive definite $(n\times n)$-matrix satisfying
%\begin{enumerate}
%\item[{[H.1]}] the coefficients $a_{i,j}$ and $b_i$ are measurable functions on $\R\times\R^n$. Moreover, for any $t\in\R$, the functions $a_{i,j}(t,\cdot),b_i(t,\cdot)\in C^2(\R)$ with bounded derivatives;
%\item[{[H.2]}] there exist $\Lambda>\lambda>0$ such that the $(n\times n)$-matrix $a=(a_{i,j})$ satisfies %being a positive definite 
%\begin{equation}
%\lambda |\xi| \leq  \langle  \xi, a(t,x) \xi \rangle  \leq \Lambda |\xi|,\qquad (t,x)\in\R\times\R^n,\ \xi\in\R^n.
%\end{equation}
%\end{enumerate} 
\section{Preliminaries}\label{sec:preliminaries}
Throughout the paper we will use the following

\begin{notation}
Let $x=(x_1,x_2),y=(y_1,y_2)\in D$, we write
\begin{equation}
x\prec y
\end{equation}
if $x_2<y_2$. Furthermore, given $z=(t,x),w=(T,y)\in \R\times D$, we write 
\begin{equation}
z \prec w%\qquad \text{if $t<T$ and $x_2<y_2$.}
\end{equation}
if $x\prec y$ and $t<T$.
\end{notation}

\begin{definition}\label{def:fund_sol}
A fundamental solution for the operator $\Lc$ is a function $p(z,w)$ defined for any $z,w\in\R\times D$ with $z\prec w$ such that:
\begin{itemize}
\item[i)] for any $w\in\R\times D$, the function $p(\cdot;w)$ solves 
\begin{equation}\label{eq:Lsolved}
\Lc u(z) = 0, \qquad z\in\R\times D,\ z\prec w,
\end{equation} 
in the sense of classical derivatives;
\item[ii)] for any $z=(T,x)\in\R\times D$, we have that $p(t,x;T,\cdot)\to \delta_{z}$ as $t\to T^-$ in the following sense:
\begin{equation}
\lim_{\substack{(t,x')\to z\\ t<T} } \int_{x'\prec y} p(t,x';T,y) \varphi(y) \dd y = \varphi(x),\qquad \varphi\in C_b(D).
\end{equation}
\end{itemize}
\end{definition}

\begin{remark}
Given a fundamental solution $p=p(z,w)$ for $\Lc$, one can consider its zero extension to the whole space $D\times D$ minus the diagonal $\{z=w\}$. This makes sense if we interpret the function $p(t,x;T,\cdot)$ as the transition density of the process $X_T$ starting at time $t$ from the point $x$. Indeed the second component $X^2$ of the process is strictly increasing due to the fact that the first component $X^{1}$ is a geometric Brownian motion, and as such it strictly positive. 
\end{remark}

It is useful to observe that the operator $\Lc$ is left-invariant with respect to the non-commutative group law %on $\R\times D$:
\begin{equation}
z \circ w = (t+T,x_1y_1, x_2 + y_2 x_1), \qquad  z=(t,x), w=(T,y)\in \R\times D.
\end{equation}
Precisely:
\begin{equation}\label{eq:left_invariance}
\Lc u(z) = 0 \Leftrightarrow \Lc u^w(z)=0,\qquad \text{with }u^w(z) :=u( w\circ z ).
\end{equation}
%with $u^w(z) := w\circ z $. 
As it turns out, $\mathbb{G}=(\R\times D,\circ)$ is a Lie Group with identity and inverse given by 
\begin{equation}
\id=(0,1,0),\qquad z^{-1} = \big(-t,x_1^{-1},-x_2 x_1^{-1}\big),
\end{equation}
respectively. By the left-invariance of $\Lc$, it is straightforward to see that  
\begin{equation}\label{eq:fund_sol_inv}
y_1^2 p(z;w) = %\frac{1}{x_1^2}
 p(w^{-1}\circ z;\id) = p\Big( t-T, \frac{x_1}{y_1}, \frac{x_2-y_2}{y_1} ;  \id  \Big),  %= \frac{1}{} p(w^{-1}\circ z,\id), 
%\qquad z=(t,x),w=(T,y)\in\R\times D, \ z\prec w.
\end{equation}
for any $z=(t,x),w=(T,y)\in\R\times D$ with $z\prec w$. Also, denoting by $\bar{p}$ the fundamental solution when $\sigma=1$, it is easy to check that 
\begin{equation}\label{eq:change_sigma}
p(z; \id) = \sigma^2 \bar{p}(\sigma^2 t, x_1,\sigma^2 x_2; \id),
\end{equation}
for any $z=(t,x)\in\R\times D$ with $z\prec \id$.

The following result was proved in \cite{cibelli2019sharp}. %At the best of our knowledge, it provides the sharpest lower/upper bounds on the fundamental solution $p$ of $\Lc$ that are known in the literature, up to this mo. 
\begin{theorem}
The operator $\Lc$ has a unique fundamental solution, which is positive, smooth ($C^{\infty}$), and for which the following upper/lower bounds hold: for any arbitrary $\eps \in ]0, 1[$ and $T_0>0$, there exist two positive constants $c_{\eps}^-,
C_{\eps}^+$ depending on $\eps$ and $T_0$, and two positive universal constants $C^-,
c^+$ such that 
\begin{equation} \label{e-twosidedbounds1}
\begin{split}
  \frac{ c_{\eps}^-}{\s^2 y_1^2 (T-t)^2} \exp & \Big(- C^- \Psi\big(t+\eps(T-t),x_1,x_2+ y_1 \eps(T-t);T,y_1,y_2\big) \Big)
\le\\
  %\\
   %& \qquad \qquad \qquad \qquad \qquad \quad
   & \quad  \quad \Gamma(t,x_1,x_2;T,y_1,y_2)  \le \\
   %\\
   & \quad \qquad \qquad \qquad
   \frac{ C_{\eps}^+}{\s^2 y_1^2 (T-t)^2} \exp \Big(- c^+\Psi\big(    t- \eps,x_1,x_2- y_1 \eps;T,y_1,y_2\big) \Big),
\end{split}
\end{equation}
for every $(t,x_1,x_2),(T,y_1,y_2)\in \R^+ \times D$ such that $(t+\eps(T-t),x_1,x_2+ y_1 \eps(T-t)) \prec (T,y_1,y_2)$ and $T-t<T_0$.
\end{theorem}

We recall here the explicit representation for the cost function $\Psi$, as it was computed in \cite{cibelli2019sharp}.
For given $z=(t,x),w=(T,y)\in \R\times D$ with $z\prec w$, set
\begin{align}\label{def:E}
 E=E_{z,w}:  &= \frac{4}{(T-t)^2} g^{-1}\Big( \frac{1}{{\bf h}(z;w)}  \Big),\\
\sign=\sign_{z,w}:  &=\text{sgn} \bigg(E_{z,w}+ \frac{ \pi^2 }{ (T-t)^2 } \bigg) = \text{sgn} \bigg(4 g^{-1}\Big( \frac{1}{{\bf h}(z;w)}  \Big)+ \pi^2 \bigg),\\ 
\end{align}
with 
\begin{equation}\label{def:g}
g(r) =  \frac{ \sinh (\sqrt{r}) }{ \sqrt{r} }  = 
\begin{cases}
\frac{ \sinh (\sqrt{r}) }{ \sqrt{r} }     , &  r>0              \\  
1, & r=0\\
\frac{ \sin(\sqrt{-r}) }{ \sqrt{-r} }, & -\pi^2 < r<0 
\end{cases}
\end{equation}
and ${\bf h}$ as defined in \eqref{eq:h}.
Sometimes, here and throughout the paper, the notation $E$ will be preferred to $E_{z,w}$ when the dependence on $(z,w)$ is clear from the context. 
The optimal cost is then given by 
\begin{align}\label{eq:Psi_explicit}
\sigma^2 \Psi(z;w) & =    E_{z,w}(T-t) + \frac{ 4 (x_1 + y_1) }{  y_2 - x_2  }  - 4 \sign_{z,w} \sqrt{E + \frac{ 4 x_1 y_1 }{ ( y_2 - x_2)^2 } },
\end{align}
or equivalently
\begin{align}\label{eq:represent_psi}
\Psi(z;w) & =\frac{4}{\s^2(T-t)}\bigg[ g^{-1}\Big( \frac{1}{{\bf h}(z;w)}  \Big)+  {\bf h}(z;w) \bigg(\sqrt{\frac{x_1}{y_1}} + \sqrt{\frac{y_1}{x_1}}\,\bigg)  - 2 \sign_{z,w} \sqrt{g^{-1}\Big( \frac{1}{{\bf h}(z;w)}  \Big) + {\bf h}^{2}(z;w) }     \bigg].
  %\\
%& =
%\begin{cases}
%   E(T-t) + \frac{ 4 (x_1 + y_1) }{  y_2 - x_2  }  - 4 \sqrt{E + \frac{ 4 x_1 y_1 }{ ( y_2 - x_2)^2 } }   & \text{if } E\geq - \frac{ \pi^2 }{ (T-t)^2 }     ,       \\  
%   E(T-t) + \frac{ 4 (x_1 + y_1) }{  y_2 - x_2  }  + 4 \sqrt{E + \frac{ 4 x_1 y_1 }{ ( y_2 - x_2)^2 } }     & \text{if } - \frac{ 4 \pi^2 }{ (T-t)^2 } < E < - \frac{ \pi^2 }{ (T-t)^2 },
%\end{cases}
\end{align}
\begin{remark}\label{rem:invariance}
%(T,y_1,y_2)^{-1} \circ (t,x_1,x_2) = \Big(t-T,\frac{x_1}{y_1},\frac{x_2-y_2}{y_1}\Big),
It is straightforward to check that 
\begin{align}\label{eq:simmetry1}
{\bf h}(z;w) = {\bf h}( w^{-1}\circ z;\id) & =  {\bf h} \Big( t-T, \frac{x_1}{y_1}, \frac{x_2-y_2}{y_1} ;  0,1,0  \Big)\\
& =  {\bf h} \Big(  0,1,0 ;  T-t, \frac{y_1}{x_1}, \frac{y_2-x_2}{x_1}  \Big) = {\bf h}( \id; z^{-1}\circ w),
\end{align}
and thus
\begin{align}\label{eq:Psi_invariance}
\Psi(z,w) = \Psi ( w^{-1}\circ z; \id) &= \Psi \Big( t-T, \frac{x_1}{y_1}, \frac{x_2-y_2}{y_1} ;  0,1,0  \Big)\\
& =  \Psi\Big(  0,1,0 ;  T-t, \frac{y_1}{x_1}, \frac{y_2-x_2}{x_1}  \Big) = \Psi( \id; z^{-1}\circ w) .
\end{align}
This, together with \eqref{eq:fund_sol_inv} and \eqref{eq:change_sigma}, implies that the estimates of Theorems \ref{th:main} and \ref{th:lower} only need to be proved for $w=\id$ and $\sigma=1$. 
\end{remark}

\section{A novel representation for the cost function}\label{sec:novel_repres}

In this section we present a convergent expansion for the cost function $\Psi$, which allows to write the latter in terms of elementary functions, and preserves its asymptotic properties as ${\bf h}(z;w)$ tends to zero and infinity. We start off by re-writing the cost function as
\begin{equation}\label{eq:rep_psi}
\Psi(z,w) =\frac{4}{\s^2(T-t)}\bigg[ {\bf h}(z,w)\bigg(\sqrt{\frac{y_1}{x_1}} + \sqrt{\frac{x_1}{y_1}} -2\bigg) +    G\big( {\bf h}(z,w) \big) \bigg],
\end{equation}
where we set
\begin{align}\label{eq:G}
G(\eta)&:=2 \eta -2 \text{sgn} \big(4 g^{-1}( \eta^{-1} )+ \pi^2 \big) \sqrt{\eta ^2+g^{-1}\left(\eta^{-1}\right)}+g^{-1}\left(\eta^{-1}\right) \\
& = 2 \eta -2 \Big(\eta^{-1} - \frac{2}{\pi}\Big) %\text{sgn} \big(4 g^{-1}( \eta^{-1} )+ \pi^2 \big) 
\sqrt{\big(\eta ^2+g^{-1}\left(\eta^{-1}\right)\big)\Big( \eta^{-1}-\frac{2}{\pi}\Big)^{-2}}+g^{-1}\left(\eta^{-1}\right)
 , \qquad \eta>0.
\end{align}
The last equality above stems from the following relation %direct computation shows that
%\begin{equation}
%4 g^{-1}( \eta^{-1} )+ \pi^2 \geq 0 \Leftrightarrow \eta \leq \frac{\pi}{2},
%\end{equation}
%and thus we have
\begin{equation}\label{eq:sgn_expl}
\text{sgn} \big(4 g^{-1}( \eta^{-1} )+ \pi^2 \big) = \begin{cases}
1,\qquad &\text{if } \eta < \frac{\pi}{2}\\
0,\qquad &\text{if } \eta = \frac{\pi}{2}\\
-1,\qquad &\text{if } \eta > \frac{\pi}{2}
\end{cases}.
\end{equation}

The representation \eqref{eq:rep_psi} for the cost function is particularly meaningful as the function $G$ is strictly convex on $\R^+$ and has global minimum  %is significant in that the quantity
\begin{equation}
\min G = G(1) = 0.
\end{equation}
In particular, $G(\eta)$ is strictly increasing for $\eta>1$ and strictly decreasing for $\eta\in ]0,1[$ (see Figure \ref{fig:ste2}).
\begin{figure}[htb]
%\title{\centering Difference between $G(\eta)%-\sum_{n=2}^N G_n(\eta)
%$ and its $N$-th order approximation\\ ${}$}
\centering
\includegraphics[width=0.5\textwidth,height=0.25\textheight]{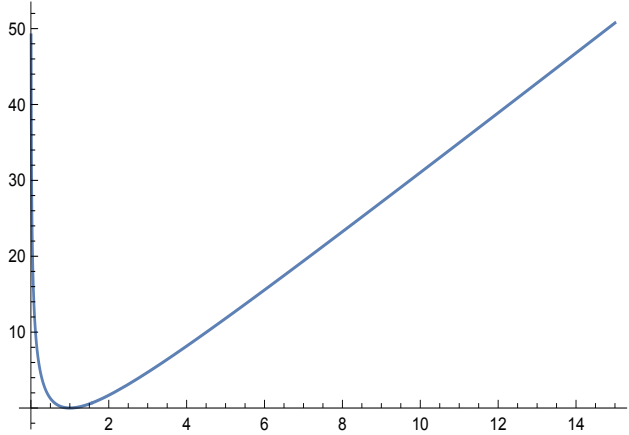} 
\caption{Plot of the function $G$ in \eqref{eq:G}.} 
\label{fig:ste2}
\end{figure} 
Therefore, the function $\Psi$ appears as a strictly increasing function of the non-negative quantities 
\begin{equation}
{\bf h}(z,w), \qquad \sqrt{\frac{y_1}{x_1}} + \sqrt{\frac{x_1}{y_1}} -2,
\end{equation}
with zero minimum at
\begin{equation}\label{eq:h_x}
\bigg( {\bf h}(z,w),  \sqrt{\frac{y_1}{x_1}} + \sqrt{\frac{x_1}{y_1}} -2 \bigg) = (1,0).
\end{equation}
\begin{remark}
Solving \eqref{eq:h_x} for $w=(T,y)$, we obtain
%\begin{equation}
%\begin{cases}
%y_1 = x_1\\
%y_2 - x_2 = (T-t) \sqrt{x_1 y_1} 
%\end{cases},
%\end{equation}
%which yields 
\begin{equation}
\begin{cases}
y_1 = x_1, \\
y_2 = x_2 + (T-t) x_1,
\end{cases}
\end{equation}
which is coherent with the control problem \eqref{eq:control1}-\eqref{eq:optimal_curves}. Indeed, a solution to \eqref{eq:optimal_curves} for $y=(y_1,y_2)$ as above is given by $\omega \equiv 0$, which clearly yields $\Psi(t,x;T,y)=0$. 
\end{remark}

The computation of the function $G$ in \eqref{eq:G} requires the numerical inversion of trigonometric functions. This makes the use of $\Psi$ numerically challenging, especially in regards to the computation of numerical integrals of the type
\begin{equation}
\int_{x\prec y} e^{-\frac{\Psi(t,x;T,y)}{2}} f(y) \dd y,
\end{equation}
which are in turn related to the evaluation of arithmetic Asian options (see \eqref{eq:pricing}).
For this reason, we look for an expansion of the function $G$ that makes the numerical computation of $\Psi$ more manageable. We would also like such expansion to preserve the asymptotic properties of $G$ as $\eta$ tends to $0^+$ and $+\infty$.

We introduce the following 
\begin{definition}\label{def:N_approx}
For any $N\in\N$ with $N\geq 2$, we define the \emph{$N$-th order expansion of $\Psi$} as
\begin{equation}
\Psi_N(z,w): =\frac{4}{\s^2(T-t)}\bigg[ {\bf h}(z,w)\bigg(\sqrt{\frac{y_1}{x_1}} + \sqrt{\frac{x_1}{y_1}} -2\bigg) +  \sum_{n=2}^{N}  G_n\big( {\bf h}(z,w) \big) \bigg],  \qquad N\geq 2,
\end{equation}
where
\begin{equation}\label{eq:def_Gn}
G_n(\eta): = a_n {\bf 1}_{]1,\infty[} (\eta)  \frac{ (\eta -1)^n}{\eta^{n-1} } +b_n {\bf 1}_{]0,1[}(\eta)  \frac{ (-\log \eta )^n}{\big(1-\log \eta\big)^{n-2}} ,
\end{equation}
and where the coefficients $a_n,b_n$ are recursively determined by solving the equations
\begin{equation}\label{eq:def_a_b}
\frac{\dd^n}{\dd \eta^n} G(\eta)\Big|_{\eta=1} = \frac{\dd^n}{\dd\eta^n} \sum_{k=2}^{n}  G_k (\eta ) \Big|_{\eta=1}, \qquad n\geq 1 .
\end{equation}
\end{definition} 
Before focusing on the convergence results and the asymptotic properties of the approximation $\Psi_N$, we discuss some aspects related to its computation. We first point out that the functions $G_n$ %in \eqref{eq:def_Gn} 
are written in terms of elementary functions. Therefore, the numerical evaluation of $\Psi_N$ is straightforward save computing the coefficients $a_n,b_n$ in \eqref{eq:def_Gn}. In regard to this matter, the characterization of the coefficients in Proposition \ref{prop:coeff_a_b} below comes in handy.  For the sake of completeness, the first $38$ coefficients $a_n,b_n$ are reported in Table \ref{tab:coeff}. 
\begin{table}
\centering
\begin{tabular}{c|c|c}
$n$  &  $a_n$  & $b_n$\\
\hline
2& 3  & 3
\\
3&0.6  & -0.6
\\
4&0.205714 & -0.444286
\\
5&0.0891429& -0.322
\\
6&0.044144&-0.227237
\\
7&0.0238419&-0.15493
\\
8&0.0136883&-0.100762
\\
9&0.00822413&-0.0610761
\\
10&0.00511786&-0.0328045
\\
11&0.00327527&-0.0133935
\\
12&0.00214449&-0.00074011
\\
13&0.00143104&0.00686845
\\
14&0.0009704&0.0108089
\\
15&0.000667139&0.0121717
\\
16&0.00046414&0.0118058
\\
17&0.000326287&0.0103592
\\
18&0.000231492&0.00831401
\\
19&0.000165581&0.00601859
\\
20&0.000119303&0.00371439
\\
21&0.0000865255&0.00155955
\\
22&0.0000631266&-0.000351322
\\
23&0.0000463045&-0.00197063
\\
24&0.0000341329&-0.00328434
\\
25&0.0000252745&-0.00430119
\\
26&0.000018793&-0.00504427
\\
27&0.0000140273&-0.00554463
\\
28&0.0000105073&-0.0058366
\\
29&$7.89662*10^-6$&-0.00595453
\\
30&$5.95282*10^-6$&-0.00593069
\\
31&$4.50037*10^-6$&-0.00579403
\\
32&$3.41143*10^-6$&-0.00556962
\\
33&$2.59248*10^-6$&-0.00527856
\\
34&$1.97478*10^-6$&-0.00493825
\\
35&$1.5076*10^-6$&-0.00456279
\\
36&$1.15335*10^-6$&-0.00416351
\\
37&$8.84092*10^-7$&-0.00374953
\\
38&$6.78961*10^-7$&-0.00332826
%\\
%&$5.2235*10^-7$&
%\\
%&$4.02539*10^-7$&
%\\
%&$3.10705*10^-7$&
%\\
%&$2.40187*10^-7$&
%\\
%&$1.85942*10^-7$&
%\\
%&$1.44147*10^-7$&
%\\
%&$1.11894*10^-7$&
%\\
%&$8.69664*10^-8$&
%\\
%&$6.76738*10^-8$&
%\\
%&$5.27218*10^-8$&
%\\
%&$4.11187*10^-8$&
%\\
%&$3.21032*10^-8$&
%\\
%&$2.509*10^-8$&
%\\
%&$1.9628*10^-8$&
%\\
%&$1.53695*10^-8$&
%\\
%&$1.20458*10^-8$&
%\\
%&$9.44911*10^-9$&
%\\
%&$7.4184*10^-9$&
%\\
%&$5.82883*10^-9$&
%\\
%&$4.58343*10^-9$&
%\\
%&$3.60685*10^-9$&
%\\
%$b_n$ & 0.177\% & 0.471\% & 0.888\%  & 1.55\%  & 5.2\%  & 10.7\%  
\end{tabular}
    \caption{Values of the coefficients $a_n,b_n$ up to $n=38$.}
\label{tab:coeff}
\end{table}

\begin{proposition}\label{prop:coeff_a_b}
For any $n\in\N$ with $n\geq 2$, the coefficients $a_n,b_n$ in \eqref{eq:def_Gn} equal to
\begin{align}\label{eq:an}
a_n &= (-1)^n \bigg( \beta_n + \beta_{n-1}  - \frac{2}{n!} \sum_{h=0}^n  \frac{(-1)^h (2h)!}{(1-2h) (h!) 4^h}
B_{n,h}\big( c_1, 2! c_2 , \cdots , (n-h+1)!  c_{n-h+1} \big)
 \bigg), \\ \label{eq:bn}
b_n &= d_{n-2}-2 d_{n-1}+d_{n}+2 (e_{n-2}-2 e_{n-1}+e_{n}-f_{n-2}+2 f_{n-1}-f_{n})   ,
\end{align}
with
%\begin{equation}
%d_{i,0} = (-1)^{i}  \b_0^i,\qquad d_{i,j} =- \frac{ \sum_{k=1}^{j} (k i - j + k ) (-1)^{k+1}  \big( \b_{k+1} + 2 \beta_k + \b_{k-1}   \big)  d_{i,j-k}  }{j \b_0 }, \qquad 0\leq j \leq i ,
%\end{equation}
\begin{align}
c_n &= \beta_n + 2 \beta_{n-1} + \beta_{n-2},\\%\qquad k\in\N,\\
d_n & = \frac{1}{n!} \sum_{h=1}^{n}  L(n,h)  \sum_{k=1}^h \beta_k k!  {h\brace k}, \label{eq:dn} \\
e_n & = \frac{1}{n!}\sum _{h=1}^n (-1)^h L(n,h),\qquad \tilde{e}_n  = \frac{1}{n!}\sum _{h=1}^n (-2)^h L(n,h), \label{eq:en}\\
f_n & = \frac{1}{n!}\sum _{h=1}^n \frac{(-1)^h (2h)!}{(1-2 h) (h!) 4^h} B_{n,h}\big( d_1 + \tilde{e}_1, 2! (d_2 + \tilde{e}_2  ) , \cdots , (n-h+1)!  (d_{n-h+1} + \tilde{e}_{n-h+1}) \big),\\
\end{align}
and 
\begin{equation}\label{eq:betas}
\b_{-1}=\b_0=0,\qquad \b_1 =6 , \qquad 
 \b_k = - \frac{6^k}{k!} \sum_{h=1}^{k-1} h! b_h B_{k,h}\bigg( \frac{1}{3!} , \frac{2!}{5!},  \cdots, \frac{(k-h+1)!}{[2(k-h+1) + 1]!}  \bigg),\ k\geq 2.
\end{equation}
Above, the functions $B_{k,h}$ represent the exponential Bell polynomials, $L(n,h)$ denote the Lah numbers, and ${h\brace k}$ the Stirling numbers of the second kind. 
\end{proposition}
\begin{proof}
We first prove \eqref{eq:an}. By the change of variable $\xi : = \frac{\eta-1}{\eta}$, we obtain that \eqref{eq:def_a_b} is equivalent to 
\begin{equation}\label{eq:def_a_b_bis}
\frac{1}{n!}\frac{\dd^n}{\dd \xi^n} F(\xi)\Big|_{\xi=0} = \frac{1}{n!} \frac{\dd^n}{\dd\xi^n} \sum_{k=2}^{n} a_k \xi ^k \Big|_{\xi=0} =  a_n , \qquad n\geq 2,
\end{equation}
where 
\begin{equation}\label{eq:change_var_def}
F(\xi):= (1-\xi)G\Big(  \frac{1}{1-\xi}  \Big) =  2 -2  \sqrt{1+(1-\xi)^2 g^{-1}\left(1-\xi\right)}+(1-\xi)g^{-1}\left(1-\xi\right) , \qquad |\xi|<<1  . 
\end{equation}
We now go on to compute the power series of $F$ at $\xi=0$. %we make repetitive use of Fa\`a di Bruno formula. 
%It is thus enough to compute the derivatives of $F$ at $\xi=0$. 
%In particular $g'(0)=\frac{1}{6}$ and thus $g$ is locally invertible in $r=0$, with $g^{-1}$ is holomorphic in a neighborhood of $g(0)=1$. In particular, denoting by 
Denoting by $(\b_n)_{n\in\mathbb{N}_0}$ the coefficients of the powers series of $g^{-1}$ in $1$, a direct application of Fa\`a di Bruno formula yields
\begin{align}
\b_0 &=0 ,\\
\b_1 &=\big({g'(0)}\big)^{-1} ,\\
n! \b_n &= - \big({g'(0)}\big)^{-n} \sum_{h=1}^{n-1} h! \b_h B_{n,h}\big( g'(0) , g''(0), \cdots, g^{(n-h+1)}(0)  \big),\qquad n\geq 2,
\end{align}
where $B_{n,h}$ represent the exponential Bell polynomials. %More explicitly, 
Note now that, by \eqref{eq:series_g}, we obtain
%A direct computation shows that 
%\begin{equation}
%g(r) = \sum_{n=0}^{\infty} \frac{r^n}{(2n+1)!},
%\end{equation}
%which implies 
\begin{equation}
g^{(n)}(0) = \frac{n!}{(2n+1)!}, \qquad n\in \N_0,
\end{equation}
which yields \eqref{eq:betas}.
%\begin{align}
%\b_1 &=6 ,\\
%n! \b_n &= - 6^n \sum_{h=1}^{n-1} h! \b_h B_{n,h}\bigg( \frac{1}{3!} , \frac{2!}{5!},  \cdots, \frac{(n-h+1)!}{[2(n-h+1) + 1]!}  \bigg),\qquad n\geq 2.
%\end{align}
Thus we obtain 
\begin{align}
(1-\xi ) g^{-1}(1-\xi) &= \sum_{n=1}^\infty (-1)^n( \beta_n + \beta_{n-1} ) \xi^n ,\\
(1-\xi )^2 g^{-1}(1-\xi) &= \sum_{n=1}^\infty  (\beta_n + 2 \beta_{n-1} + \beta_{n-2}) \xi^n,
\end{align}
for any $\xi$ close to $0$. Eventually, applying once more Fa\`a di Bruno formula, together with
\begin{equation}\label{eq:square_root_tay}
\sqrt{1+x} = 1 + \sum_{n=1}^{\infty}   \frac{(-1)^n (2n)!}{(1-2n) (n!)^2 4^n} x^n, \qquad |x|<<1,
\end{equation}
yield \eqref{eq:def_a_b_bis} with $a_n$ as given by \eqref{eq:an}. 

To prove \eqref{eq:bn} we use an analogous argument. By the change of variable $\xi : = \frac{-\log \eta}{1 - \log \eta}$, we obtain that \eqref{eq:def_a_b} is equivalent to 
\begin{equation}\label{eq:def_a_b_ter}
\frac{1}{n!}\frac{\dd^n}{\dd \xi^n} F(\xi)\Big|_{\xi=0} = \frac{1}{n!} \frac{\dd^n}{\dd\xi^n} \sum_{k=2}^{n} b_k \xi ^k \Big|_{\xi=0} =  b_n , \qquad n\geq 2,
\end{equation}
where
\begin{equation}\label{eq:change_var_def_bis}
F(\xi):=(1-\xi)^2 G\Big(  e^{\frac{\xi}{\xi-1} }  \Big)= (1-\xi)^2 \Big( 2 e^{\frac{\xi}{\xi-1}}  -2  \sqrt{e^{\frac{2\xi}{\xi-1}}+g^{-1}\big( e^{\frac{\xi}{1-\xi}} \big)}+g^{-1}\big( e^{\frac{\xi}{1-\xi}} \big) \Big), \qquad |\xi|<<1  . 
\end{equation}
By applying now Fa\`a di Bruno formula we obtain
\begin{align}
e^{\frac{\xi}{\xi-1}} -1  & =  \sum_{n=1}^{\infty}  \frac{\xi^n}{n!} \sum _{h=1}^n  B_{n,h}\big(-1,-2!,\cdots,-(n-h+1)!\big)  =  \sum_{n=1}^{\infty}  \frac{\xi^n}{n!} \sum _{h=1}^n (-1)^h \binom{n-1}{h-1}\frac{n!}{h!} =  \sum_{n=1}^{\infty} e_n {\xi^n},\\ %L(n,h)      \\
e^{\frac{2\xi}{\xi-1}} - 1 &  =   \sum_{n=1}^{\infty}  \frac{\xi^n}{n!} \sum _{h=1}^n  B_{n,h}\big(-2,-2\cdot 2!,\cdots,-2\cdot (n-h+1)!\big)  =  \sum_{n=1}^{\infty} \tilde{e}_n {\xi^n} , \\
g^{-1}\big( e^{\frac{\xi}{1-\xi}} \big)   & =     \sum_{n=1}^{\infty}  \frac{\xi^n}{n!} \sum _{h=1}^n  B_{n,h}\big(1, 2!,\cdots,  (n-h+1)!\big) \sum_{k=1}^h k! \beta_k B_{n,k}(1,\cdots,1)  =  \sum_{n=1}^{\infty}{d}_n {\xi^n}  ,
\end{align}
for any $\xi$ close to $0$, where $d_n,e_n,\tilde{e}_n$ are as in \eqref{eq:dn}-\eqref{eq:en}. Eventually, applying again Fa\`a di Bruno formula together with \eqref{eq:square_root_tay} yield \eqref{eq:def_a_b_ter} with $b_n$ as in \eqref{eq:bn}.
%\xi : = \frac{-\log \eta}{1 - \log \eta}$ we have that \eqref{eq:def_Gn} is equivalent to
%\begin{equation}
% (1-\xi)^2 G\Big(  e^{\frac{\xi}{\xi-1} }  \Big)   =   \sum_{n= 2}^{\infty} b_n  \xi^n , \qquad \xi\in ] 0 , 1  [ .
%\end{equation}
\end{proof}
We now address both point-wise and asymptotic convergence of $\Psi_N$ to $\Psi$. Concerning the former one, the desired result would be %Our numerical evidence seems to suggest that
\begin{equation}\label{eq:conv_G}
\sum_{n=2}^N G_n(\eta) \to G(\eta)\qquad \text{as }N\to +\infty, %\qquad \eta>0.
\end{equation}
for any $\eta>0$, which would in turn imply
\begin{equation}
\Psi_N (z,w) \to \Psi(z,w)\qquad \text{as }N\to \infty,
\end{equation}
for any $z,w\in \R\times D$ with $z\prec w$.
However, unfortunately, we were able to provide a rigorous proof of \eqref{eq:conv_G} only for $\eta>1$ (see Theorem \ref{th:represent_Psi} below). Despite of this, we point out that strong numerical evidence suggests that \eqref{eq:conv_G} is satisfied also for $\eta\in]0,1[$. To support this claim, in Figure \ref{fig:ste1} we compare the plot of the truncated series $\sum_{n=2}^N G$ with that of $G$, for different values of $N$.
\begin{figure}[htb]
%\title{\centering Difference between $G(\eta)%-\sum_{n=2}^N G_n(\eta)
%$ and its $N$-th order approximation\\ ${}$}
\centering
\includegraphics[width=1\textwidth,height=0.55\textheight]{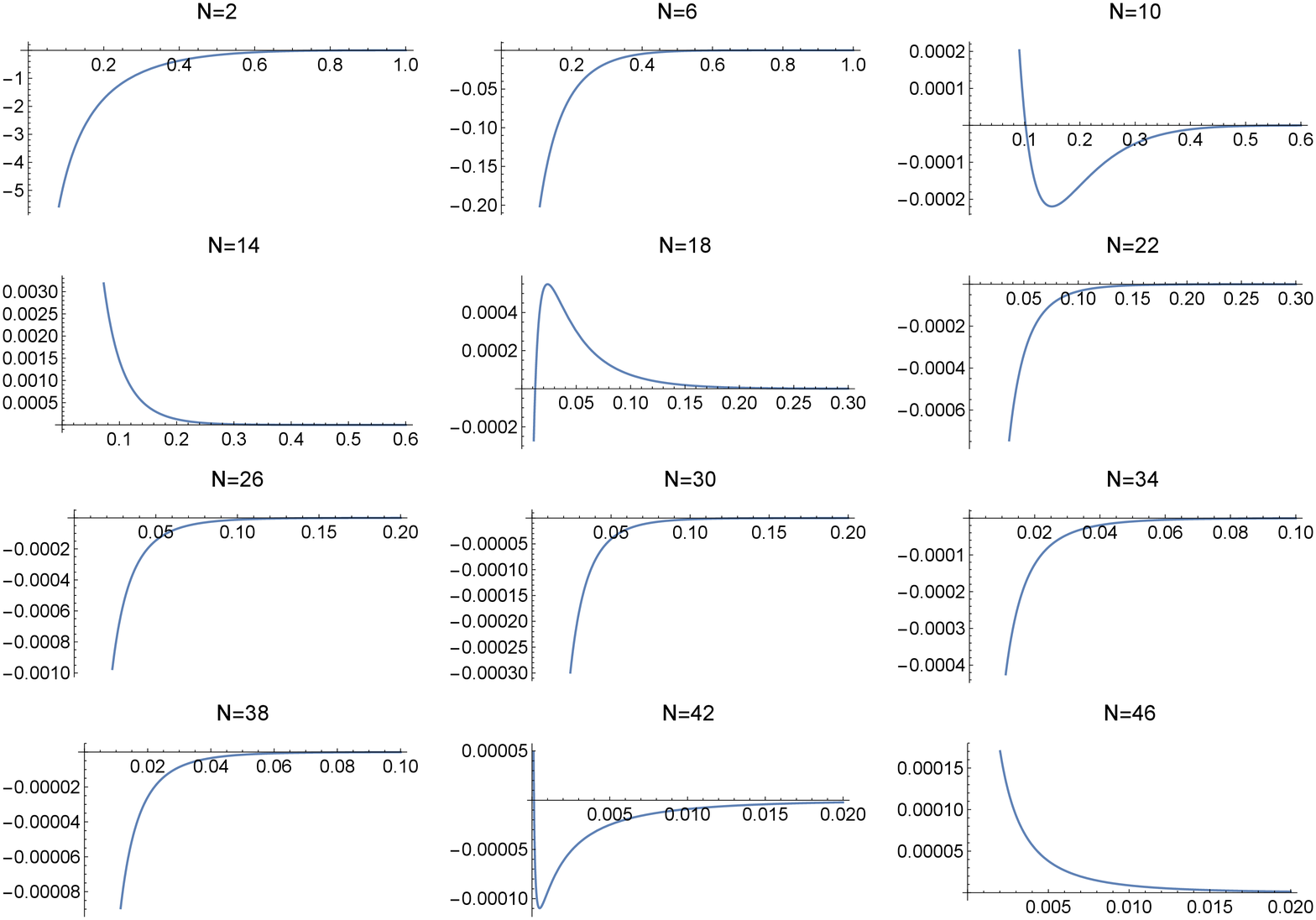} 
\caption{Plot of the difference $G(\eta)-\sum_{n=2}^N  G_n(\eta)$, with $G_n$ as in Definition \ref{def:N_approx}, for values of $\eta$ close to $0$.} 
\label{fig:ste1}
\end{figure} 
\begin{theorem}\label{th:represent_Psi}
%We have
For any $\eta>1$ the limit \eqref{eq:conv_G} holds true.
%\begin{equation}
%G(\eta) = \sum_{n=2}^{\infty} G_n(\eta),\qquad \eta>1.
%\end{equation}
\end{theorem}
The proof of Theorem \ref{th:represent_Psi} is deferred until Section \ref{subsec:proof}. 

We now turn our attention to the asymptotic convergence. By definition of $g$ we can easily obtain 
\begin{equation}\label{eq:asymp_ginv}
g^{-1}(h) \sim \log^2 h \qquad \text{ as } h\to +\infty, 
\end{equation}
where, as already stated in the introduction, we adopted the notation
%Above, we used the notation 
\begin{equation}
f \sim g \Leftrightarrow \frac{f}{g} \longrightarrow 1.
\end{equation}
The definition of $G$ together with \eqref{eq:asymp_ginv} then yield
\begin{align}\label{eq:asymptotic_G_right}
  G(\eta) \sim 4 \eta \qquad &\text{as } \eta\to +\infty,\\ \label{eq:asymptotic_G_left}
  G(\eta) \sim \log^2 \eta  \qquad &\text{as } \eta\to 0^+ .
\end{align}
The idea behind the approximation $\Psi_N$ is to expand the function $G$ by means of suitable basis functions so as to obtain
\begin{align}\label{eq:asympt_sum_right}
{\sum_{n=2}^{N}  G_n (\eta )}\sim \eta \sum_{n=2}^N a_n\qquad &\text{as } \eta\to \infty, \\ \label{eq:asympt_sum_left}
{\sum_{n=2}^{N}  G_n (\eta )}\sim (\log^2 \eta)  \sum_{n=2}^N b_n   \qquad &\text{as } \eta\to 0^+.
\end{align}
Indeed, the latter easily stem from 
\begin{align}
\frac{ (\eta -1)^n}{\eta^{n-1} } &\sim \eta \qquad \text{as } \eta\to +\infty,\\
\frac{ (-\log \eta )^n}{\big(1-\log \eta\big)^{n-2}}  &\sim \log^2 \eta \qquad \text{as } \eta\to 0^+ .
\end{align}
Ideally, we would like to show that
\begin{align}\label{eq:asympt_behav}
\lim_{N\to\infty} \sum_{n=2}^N a_n =& 4, \\ \label{eq:asympt_behav_log}
  \lim_{N\to\infty}\sum_{n=2}^N b_n = &1,%, \qquad \text{as } N\to \infty. %\\
%  \sum_{n=2}^N b_n \to 1  \qquad &\text{as } N\to \infty.
\end{align}
which means the asymptotic behavior of $\sum_{n=2}^N G_n(\eta)$ converges to the asymptotic behavior of $G(\eta)$ as $\eta$ tends to $0^+$ and $+\infty$. Note that this property is in general not granted for convergent expansions. We start by considering the case $\eta>1$. %The values reported in Table \ref{} seem to suggest that the coefficients $a_n$ are all positive. 
In Corollary \ref{cor:asympt} below, we show that \eqref{eq:asympt_behav} is satisfied provided that the coefficients $a_n$ are all positive. Unfortunately, we were not able to rigorously prove the positiveness of the coefficients $a_n$, thought the numerical values reported in Table \ref{tab:coeff} strongly support this conjecture. 
\begin{corollary}\label{cor:asympt}
Under the assumption that the coefficients $(a_n)_{n\geq 2}$ are all positive, the limit \eqref{eq:asympt_behav} holds true.
\end{corollary}
\begin{proof}
Let $\bar{a}:=\lim_{N\to\infty} \sum_{n=2}^N a_n$ and assume $\bar{a}< 4$. By definition \eqref{eq:def_Gn}, and by \eqref{eq:asymptotic_G_right}, there exists $\bar{\eta}>1$ such that
\begin{equation}\label{eq:bound001}
\sum_{n=2}^N G_n(\bar{\eta}) < \bar{a} \bar{\eta} < \frac{\bar{a}+4}{2}\bar{\eta} < G(\bar{\eta}) ,\qquad N\geq 2. 
\end{equation}
This violates Theorem \ref{th:represent_Psi} and %states that $\sum_{n=2}^N G_n(\bar{\eta})\to G(\bar{\eta})$ as $N\to+\infty$, which is a 
yields a contradiction. Therefore, we have $\bar{a}\geq 4$. 

Assume now that $\bar{a}> 4$. Then there exits $\bar{N}$ such that %for $N$ suitably large we have  
\begin{equation}
\sum_{n=2}^{\bar{N}} a_n > \frac{\bar{a}+4}{2} .%, \qquad \bar{N}>\bar{N}.
\end{equation}
On the other hand, the positiveness of the coefficients $a_n$, together with \eqref{eq:asymptotic_G_right} and  \eqref{eq:asympt_sum_left}, implies that there exists $\bar{\eta}>1$ such that
\begin{equation}
\sum_{n=2}^N G_n(\bar{\eta}) > \sum_{n=2}^{\bar{N}} G_n(\bar{\eta}) > \frac{\bar{a}+4}{2} \bar\eta >  G( \bar\eta ), \qquad N>\bar{N},
\end{equation}
which again violates Theorem \ref{th:represent_Psi} and %states that $\sum_{n=2}^N G_n(\bar{\eta})\to G(\bar{\eta})$ as $N\to+\infty$, which is a 
yields a contradiction. This proves that $\bar{a}\leq 4$ and concludes the proof. 

\end{proof}
For $\eta\in]0,1[$ the situation is less clear. On the one hand, the values in Table \ref{tab:coeff} suggests that the sign of the coefficients $b_n$ oscillates as $n$ grows, and thus we cannot rely on the positiveness (or negativeness) of the coefficients $b_n$ to infer \eqref{eq:asympt_behav_log}. On the other hand, the plot in Figure \ref{fig:ste2} shows that $\sum_{n=2}^N b_n$ is close to $1$ for (not too) large values of $N$. Unfortunately, the numerical complexity of the representation \eqref{eq:bn} makes it very difficult to compute the summation above for larger values of $N$. Therefore, the question whether  \eqref{eq:asympt_behav_log} is actually true remains open.
\begin{figure}[htb]
%\title{\centering Difference between $G(\eta)%-\sum_{n=2}^N G_n(\eta)
%$ and its $N$-th order approximation\\ ${}$}
\centering
\includegraphics[width=0.5\textwidth,height=0.25\textheight]{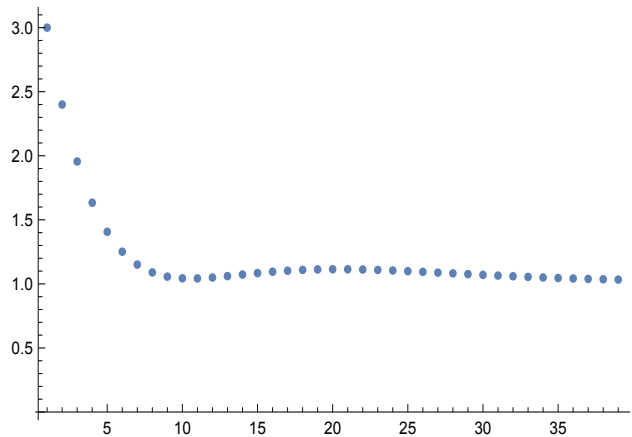} 
\caption{Plot of the sum $\sum_{n=2}^{N} b_n$ for $N$ up to $40$.} 
\label{fig:ste2}
\end{figure} 

\subsection{Proof of Theorem \ref{th:represent_Psi}}\label{subsec:proof}
This section is devoted to the proof of Theorem \ref{th:represent_Psi}, which is preceded by some preliminary results. Hereafter, for any $z\in\mathbb{C}$ and $\delta>0$, we denote by $B_{\delta}(z)$ the complex ball centered at $z$ with radius $\delta$.
\begin{lemma}\label{lemma:prelim}
The function $g$ defined by \eqref{def:g} is holomorphic on $\mathbb{C}$, and 
\begin{equation}\label{eq:der_g_nonzero}
g'(z)\neq 0 ,\qquad z \in {B_{\pi^2}(0)}.
\end{equation}
Furthermore, the restriction of $g$ to $\partial B_{\pi^2}(0)$ %the circumference centered at $0$ with radius $\pi^2$ 
is a homeomorphism, and its image surrounds $B_1(1)$.
%Finally, we have
%and
%\begin{equation}
%B_1(1) \subset g\big( B_{\pi^2}(0) \big).
%\end{equation}
%where $B_r(z)$ denotes the complex ball centered at $z$ with radius $r$.
\end{lemma}
\begin{proof}
Owing to the Taylor series expansion of $\sinh z$ around $z=0$, we directly obtain
\begin{equation}\label{eq:series_g}
g(z) = \sum_{n=0}^{\infty} \frac{z^n}{(2n+1)!},\qquad r\in\mathbb{C},
\end{equation}
which shows that $g$ is holomorphic on $\mathbb{C}$.

We now prove that $g'\neq 0$ on $B_{\pi^2}(0)$. Differentiating \eqref{def:g} we obtain %that \eqref{eq:der_g_nonzero} is equivalent to 
\begin{equation}
g'(z) = \frac{%\sqrt{z} \cosh (\sqrt{z}) - \sinh (\sqrt{z})
f(\sqrt{z})}{2 z^{{3}/{2}}}, \qquad z\in \mathbb{C},
\end{equation}
where 
%$f(z) = $
\begin{equation}\label{eq:f_series}
f(z): = z \cosh z - \sinh z = \sum_{n=1}^{\infty}  \frac{2n }{(2n+1)!}z^{2n+1} .
\end{equation}
Therefore, owing to the fact that $g'(0)=1/6$ (by \eqref{eq:series_g}), and to $f(\bar{z}) = \overline{f(z)}$, we have that \eqref{eq:der_g_nonzero} is equivalent to
\begin{equation}\label{eq:fnonnull}
%z \cosh z - \sinh z \neq 0
f(z)\neq 0, \qquad z\in \tilde{B}_{\pi}(0)\setminus \{0\},
\end{equation}
where 
\begin{equation}
\tilde{B}_{\pi}(0) := \{ x+i y \in {B}_{\pi}(0) : x,y \geq 0 \}. 
\end{equation}
We  first observe, that
\begin{equation}\label{eq:fnotzero}
f(x)\neq 0, \qquad x\in\R,\ x\neq 0,
\end{equation}
which follows directly from the series representation in \eqref{eq:f_series}.
Finally, a direct computation shows that 
\begin{align}
|f(z)|^2 &= \frac{1}{2} \Big(\left(x^2+y^2-1\right) \cos (2 y)+\left(x^2+y^2+1\right) \cosh (2 x)-2 x \sinh (2 x)-2 y \sin (2 y)\Big) \\
& > %\big(\sinh (x)-x \cosh (x)\big)^2 = 
|f(x)|^2 %\qquad z=x+iy \in \tilde{B}_{\pi}(0),
\end{align}
for any $z=x+iy \in \tilde{B}_{\pi}(0)$ with $y>0$. This, together with \eqref{eq:fnotzero}, proves \eqref{eq:fnonnull}.

We now prove that $g|_{\partial B_{\pi^2}(0)}$ is a homeomorphism: it is enough to show that it is injective. Owing to 
\begin{equation}\label{eq_conjugate_prop}
\overline{g(z)}=g(\overline{z}), \qquad z\in\mathbb{C},
\end{equation} 
the latter property can be proved by checking that the function $\theta \mapsto |\sinh (\pi e^{i \theta}) |$ 
%\begin{equation}
%|\sinh (\pi e^{i \theta}) |
%\end{equation}
is strictly decreasing on $[0,\pi/2]$, and that the imaginary part of $e^{-i \theta} \sinh (\pi e^{i \theta})$ is strictly positive for $\theta\in]0,\pi/2[$. We omit the details for brevity. %{\blue (aggiungere la spiegazione del perch\'e fa almeno un giro.)}

Finally, we prove that the image of $g|_{\partial B_{\pi^2}(0)}$ surrounds $B_1(1)$. Owing again to \eqref{eq_conjugate_prop}, 
%\begin{equation}\label{eq_conjugate_prop}
%\overline{g(z)}=g(\overline{z}), \qquad z\in\mathbb{C},
%\end{equation}
it is enough to prove that
\begin{equation}\label{eq:circonf}
\Big|  \frac{\sinh( \pi e^{i \theta} )}{\pi e^{i \theta}}  -1  \Big|%= \frac{1}{\pi} \big|  \sinh(\pi e^{i \theta}) - \pi e^{i \theta} \big| = 
%\frac{1}{\pi} \Big| \sum_{n=1}^{\infty} \frac{ ( \pi e^{i \theta} )^{2n+1}}{(2n+1)!}  \Big| 
\geq 1, \qquad \theta \in [0,\pi/2].
\end{equation}
We have
\begin{align}
\Big|  \frac{\sinh( \pi e^{i \theta} )}{\pi e^{i \theta}}  -1  \Big|= %= \frac{1}{\pi} \big|  \sinh(\pi e^{i \theta}) - \pi e^{i \theta} \big| = 
\frac{1}{\pi} \Big| \sum_{n=1}^{\infty} \frac{ ( \pi e^{i \theta} )^{2n+1}}{(2n+1)!}  \Big|  \geq \frac{1}{\pi} | I_1(\theta) - I_2 |,
\end{align}
with 
\begin{equation}
I_1(\theta) = \Big| \sum_{n=1}^{3} \frac{ ( \pi e^{i \theta} )^{2n+1}}{(2n+1)!}  \Big|  , \qquad I_2 = \sum_{n=4}^{\infty} \frac{ \pi^{2n+1}}{(2n+1)!} .
\end{equation}
Now, it is not difficult to prove that there exists $\varsigma >0 $ such that
\begin{equation}
I_1(\theta)  >   \varsigma   +  \pi  >  I_2 + \pi,   \qquad \theta \in [0,\pi/2],
\end{equation}
which proves \eqref{eq:circonf}.
\end{proof}
%\begin{lemma}
%%Let $f:B_1(1)\to \mathbb{C}$ be 
%The function $f$ defined by %as %The function defined by
%\begin{equation}
%f(z):=\frac{z ^{-2}+g^{-1}(z)}{z-\frac{2}{\pi}}
%\end{equation}
%is holomorphic on $B_1(1)$, and its image $f\big(B_1(1)\big)$ does not contain non-negative real numbers.
%\end{lemma}
We are now in the position to prove Theorem \ref{th:represent_Psi}.
\begin{proof}[Proof of Theorem \ref{th:represent_Psi}]
%We need to prove that the function $G$ in \eqref{eq:G} can be represented as
%\begin{equation}\label{eq:exp_G}
%G(\eta) = \sum_{n=2}^{\infty}  G_n  (\eta), \qquad \eta>0.
%\end{equation}
%We first prove \eqref{eq:expansion_psi} for $\eta\geq 1$. 
By the change of variable $\xi : = \frac{\eta-1}{\eta}$, %we have that \eqref{eq:expansion_psi} is equivalent to 
it suffices to prove
\begin{equation}\label{eq:change_var}
F(\xi):= (1-\xi)G\Big(  \frac{1}{1-\xi}  \Big)   =   \sum_{n= 2}^{\infty} a_n \xi^n , \qquad \xi\in ] -1 , 1  [  .
\end{equation}
We prove \eqref{eq:change_var} by showing that $F$ %the function defined by
is holomorphic on $B_1(0)$. Explicitly, we have %the complex ball centered at $0$ with radius $1$. 
\begin{equation}
F(\xi):= %(1-\xi)G\Big(  \frac{1}{1-\xi}  \Big) = 
(1-\xi) \bigg( \frac{2}{1-\xi} -2 (1-\xi - 2/\pi ) \sqrt{\frac{(1-\xi)^{-2}+g^{-1}\left(1-\xi\right)}{(1-\xi - 2/\pi )^2}}+g^{-1}\left(1-\xi\right) \bigg).
\end{equation}
Lemma \ref{lemma:prelim} combined with Lemma \ref{lem:lem_topol} imply that $g|_{B_{\pi^2}(0)}$ is a biholomorphism on its image. In particular, the inverse $g^{-1}$ is well defined and holomorphic on $B_1(1)$. %Moreover, a direct computation shows that 
Setting now
\begin{equation}
f(z):={z ^{-2}+g^{-1}(z)}, \qquad z\in B_1(1),%{z-\frac{2}{\pi}}
\end{equation}
it is straightforward to check that 
\begin{equation}
f\Big(\frac{2}{\pi} \Big)  = 0 = f' \Big(\frac{2}{\pi} \Big).
\end{equation}
The first equality above follows directly from \eqref{eq:sgn_expl}, whereas the second stems from \eqref{eq:sgn_expl} together with
\begin{equation}
f'(2/\pi) = -2 z^{-3}  +  \frac{2 g^{-1}(z)}{\sqrt{ 1+ z g^{-1}(z) } - z}\bigg|_{z=2/\pi}, \qquad z\in\R.
\end{equation}
It follows that
\begin{equation}
z\mapsto \frac{f(z)}{(z - 2/\pi )^2}
\end{equation}
is holomorphic on $B_1(1)$. it is also possible to observe that $f\big(B_1(1)\big)$ does not contain non-negative real numbers, which finally yields holomorphicity of $F$ on $B_1(0)$.
%which in turn is equivalent to $\tilde{G}$ being analytic on $[0,1[$.
%On the other hand, for $\eta \in ]0,1[$, by the change of variable $\xi : = \frac{-\log \eta}{1 - \log \eta}$ we have that \eqref{eq:exp_G} is equivalent to
%\begin{equation}
% (1-\xi)^2 G\Big(  e^{\frac{\xi}{\xi-1} }  \Big)   =   \sum_{n= 2}^{\infty} b_n  \xi^n , \qquad \xi\in ] 0 , 1  [ .
%\end{equation}
\vspace{5pt}

\end{proof}

%On the other hand, for $\eta \in ]0,1[$, by the change of variable $\xi : = \frac{-\log \eta}{1 - \log \eta}$ we have that \eqref{eq:def_Gn} is equivalent to
%\begin{equation}
% (1-\xi)^2 G\Big(  e^{\frac{\xi}{\xi-1} }  \Big)   =   \sum_{n= 2}^{\infty} b_n  \xi^n , \qquad \xi\in ] 0 , 1  [ .
%\end{equation}

\section{Yosida's parametrix construction}\label{sec:yosida}

For the sake of clarity, before dwelling on the details of our extension of Yosida's parametrix method, we outline the general parametrix construction of the fundamental solution $p$, for a general choice of the {parametrix} function. 

\subsection{General parametrix construction}\label{sec:general_parametrix}

Consider a given function $H(z;w)$, hereafter referred to as \emph{parametrix function}, or simply \emph{parametrix}, which enjoys suitable regularity and boundedness properties, together with the Dirac delta property ii) in Definition \ref{def:fund_sol}, and such that
\begin{equation}\label{eq:new}
H(t,x_1,x_2;T,y_1,y_2)\to 0 \qquad \text{as } y_2\to x_2^+.
\end{equation}
Following the standard approach, we look for $p=p(z;w)$ of the form
%we use $H$ as a parametrix and look for a fundamental solution $p(z;w)$ of $\Lc$, at $z=\id$, of the form
%to construct the fundamental solution $p(z;w)$ of $\Lc$ in $w=\id$, by standard iterative procedure namely
\begin{equation}\label{eq:ansaz_p}
p(z, w) = H(z, w)+ %\hspace{-50pt} 
\int_{z \prec \zeta\prec w}%\limits_{\quad\qquad\qquad]0,T[\times \R^+ \times ]0,y_2[}  \hspace{-35pt} 
 H(z;\zeta)  \Phi(\zeta;w)  \dd \zeta ,\qquad z\prec w.
\end{equation}
By formally applying Point i) in Definition \ref{def:fund_sol}, we obtain
%with%where the function %$\Phi(\cdot;\id)$ is given by the standard iterative procedure
\begin{align}
0=\Lc p(z,w) = \Lc H(z, w) + \Lc \int_{z \prec \zeta\prec w}%\limits_{\quad\qquad\qquad]0,T[\times \R^+ \times ]0,y_2[}  \hspace{-35pt} 
 H(z;\zeta)  \Phi(\zeta;w)  \dd \zeta 
 \intertext{(owing to the Dirac delta property $H(t,x;t,\xi) = \delta_x$ and to \eqref{eq:new})}
 = \Lc H(z, w) + \int_{z \prec \zeta\prec w}   \Lc  H(z;\zeta)  \,   \Phi(\zeta;w)  \dd \zeta   - \Phi(z,w),
\end{align}
which yields 
\begin{equation}\label{eq:picard_PHI}
\Phi(z,w) =  \Lc H(z, w) + \int_{z \prec \zeta\prec w}   \Lc  H(z;\zeta)  \,   \Phi(\zeta;w)  \dd \zeta, \qquad z\prec w.
\end{equation}
Now by Picard iteration we look for a solution to \eqref{eq:picard_PHI} of the form
\begin{equation}\label{eq:parametrix_series}
 \Phi(z;\zeta) 
   =   \sum_{n=0}^{\infty}  %(-1)^{n+1}  
 K_n(z;\zeta), \qquad z\prec w,%(\zeta,\cdot),
\end{equation}
where the functions $K_n$ are defined through the recursion 
\begin{align}\label{eq:param_series}
 K_0 (z;w)  =    \Lc H(z;w)         ,  %\\
&&  K_n (z;w)    =    
%-  \int_{[t,T]}\int_{\R^+}\int_{]x_2,y_2[}   \Lc H(z;\zeta)  K_{n-1}(\zeta;w)  \dd \zeta        ,
%\hspace{-50pt} 
\int_{z \prec \zeta\prec w}
%\limits_{\quad\qquad\qquad]0,T[\times \R^+ \times ]0,y_2[}  \hspace{-35pt}  
\Lc H(z;\zeta)  K_{n-1}(\zeta,w)  \dd \zeta        ,
\qquad n\geq1.
\end{align}
%\begin{remark}
%The key difference with respect to the classical parametrix method, introduced by Levi, lies in the choice of the parametrix function $H$. In the classical case $H$ is given by the fundamental solution of the differential operator $\Lc_0$ obtained by freezing the variable coefficients of the  
%\end{remark}

%In force of Remark \ref{rem:invariance}, it is enough to consider $w=\id$. 
In order for the construction above to be formalized, one needs to find a parametrix $H$ satisfying suitable estimates so that:
\begin{itemize}
\item[a)] the series in \eqref{eq:parametrix_series} is convergent;
\item[b)] the properties i) and ii) of Definition \ref{def:fund_sol} can be verified.
%\item[c)] upper and lower bounds for $p$ and its derivatives can be derived by \eqref{eq:ansaz_p}
\end{itemize}
Furthermore, as a by-product of \eqref{eq:ansaz_p} and of the estimates for $H$ and $\Phi$, one can prove bounds for $p$ and its derivatives. In our case, thanks to a specific choice of $H$, we will be able to prove the estimates in Theorem \ref{th:main} and Theorem \ref{th:lower}.
\subsection{Generalized Yosida's parametrix}
Following Yoshida \cite{yosida1953fundamental}, we introduce our parametrix $H$ in two steps. We first define a \emph{pre-parametrix} $H_1$ as a function having the expected asymptotic behavior both away from the pole, and close to the pole. As we will see, the function $H_1$ is not suitable to define the integrals appearing in \eqref{eq:param_series}. Hence, in the second step, we introduce a correction term $\bf{u}$ that makes the functions $K_n$'s well defined if we use $H := {\bf u} H_1$ as a parametrix.

%(Definiamo in due fasi..)Instead of starting with the ``right" definition of the parametrix $H$, which will be used in the iterative construction reported above, we prefer to show the intermediate steps that lead to its construction. 

Hereafter throughout the whole section we will set $\sigma=1$ without loss of generality (see Remark \ref{rem:invariance}).
For any $z=(t,x), w=(T,y)\in \R\times D$ with $z\prec w$, %$t>T$ and $y_2>x_2$, 
define the \emph{pre-parametrix} $H_1$ as in \eqref{eq:H1_intro}. Explicitly, 
%\begin{equation}
%H_1(z;w): =  \frac{e^{-\frac{1}{2}\Psi(z;w)}}{\sqrt{(2\pi)^2\,  \text{det}\,{\bf C}(y_1,T-t)}}  ,% \quad {\bf C}(y_1,s) = \s^2 y^2_1 \begin{pmatrix}
%\end{equation}
%where 
%\begin{equation}
%{\bf C}(y_1,s) =  y^2_1 \begin{pmatrix}
%    s & \frac{s^2}{2} \\
%    \frac{s^2}{2} & \frac{s^3}{3} \
%  \end{pmatrix}.
%\end{equation}
%Explicitly, we have
\begin{equation}\label{eq:parametrix_pre_explicit}
H_1(z;w) =  \frac{\sqrt{12}}{2\pi (T-t)^2 y_1^2} e^{-\frac{1}{2}\Psi(z;w)}.
\end{equation}

The idea behind this definition is the following. On one hand, the cost function $\Psi$ at the exponential is supposed to provide the exact asymptotic behavior of the fundamental solution away from the pole. On the other hand, the square root of the determinant of ${\bf C}$ in the denominator yields the same singular behavior near the pole as the fundamental solution of the \emph{frozen} operator %$\tilde\Lc$,  
\begin{equation}
\tilde\Lc_y ={\partial_t + x_1\partial_{x_2}}%_{=:Y}
 + \frac{ y_1^2}{2} \partial_{x_1 x_1}, \qquad (t,x_1,x_2)\in \R^3.
\end{equation}
%whose fundamental solution is
%\begin{equation}
%\tilde{p}(t,x;T,y) = \frac{e^{-\frac{1}{2}\tilde\Psi(z;w)}}{\sqrt{(2\pi)^2\,  \text{det}{\bf C}(y_1,T-t)}},
%\end{equation}
%with $\tilde\varphi$ being the cost function associated 
In practice, this specific choice for the denominator ensures that $H_1(z,w)$ enjoys the Dirac Delta property ii) in Definition \ref{def:fund_sol}. The exact statement for this is in Proposition \ref{prop:delta_Dirac} below.

Now, if we were to choose $H_1$ as a parametrix, the first step towards proving the convergence of the series \eqref{eq:parametrix_series} would be to bound $\Lc H_1$ from above. To do so, one first computes 
%where $u$ is a function that will be later specified, and $\Psi$ is the optimal cost function previously defined. 
\begin{align}
\partial_t H_1 &=  \bigg( \frac{2}{T-t} -\frac{1}{2} \partial_t\Psi \bigg)    H_1,\\
 \partial_{x_1} H_1 &=  -\frac{1}{2} \big( \partial_{x_1}\Psi   \big)  H_1,\\
 \partial_{x_2} H_1 &=  -\frac{1}{2} \big( \partial_{x_2}\Psi   \big)  H_1 ,\\
  \partial_{x_1 x_1} H_1 &=  \bigg( \frac{1}{4}\big(\partial_{x_1}\Psi\big)^2 -\frac{1}{2} \partial_{x_1 x_1}\Psi \bigg)     H_1,
\end{align}
and obtains
\begin{equation}\label{eq:LH_bis}
\Lc H_1(z;w) =\bigg( \Big( \frac{x_1}{2} \partial_{x_1}  \Psi(z;w)  \Big)^2 - Y\Psi(z;w)  + f(z,w)     \bigg) H_1(z;w),
\end{equation}
with
\begin{equation}\label{eq:fg_bis}
f(t,x;T,y) := \frac{2}{T-t} - \frac{x_1^2  \partial_{x_1 x_1}\Psi(t,x;T,y)}{4} .%, \qquad g(t,x;T,y) =\frac{ x_1^2}{2}  \partial_{x_1}\Psi(t,x;T,y).
\end{equation}
At first glance, a twofold problem appears: 
\begin{itemize}
\item[-] the function $ f(z,w)$ is singular of order $\frac{1}{T-t}$;
\item[-] the functions
\begin{equation}
\Big( \frac{x_1}{2} \partial_{x_1}  \Psi(z;w)  \Big)^2\quad \text{and}\quad Y\Psi(z;w)
\end{equation}
are singular of order $\frac{1}{(T-t)^2}$.
\end{itemize}
In light of the space-time convolutions that define the terms $K_n$ in \eqref{eq:param_series}, the presence of these singular terms seems to undermine the construction using $H_1$. However, the next lemma shows that the singular terms of order $\frac{1}{(T-t)^2}$ cancel each other out.
\begin{lemma}[HJB equation]\label{lem:HJB_eq}
%For any $t\in\R$ and $y\in\R^n$, 
For any $w=(T,y)\in \R\times D$, the function $\Psi(\cdot,\cdot;T,y)%\in C^2(\R)
$ %and 
satisfies 
%The HJB equation for the control problem in Remark \ref{rem:control_prob} yields 
\begin{equation}\label{eq:fund_identity_dim_n}
%\begin{cases}
%Y\Psi(z;w) = \Big( \frac{x_1}{2} \partial_{x_1}  \Psi(z;w)  \Big)^2
%, \qquad t<T \\ %\qquad (t,x)\in\R\times\R^n,
%\Psi(T,x;T,x) = 0
%\end{cases}.
Y\Psi(z;w) = \Big( \frac{x_1}{2} \partial_{x_1}  \Psi(z;w)  \Big)^2
, \qquad z=(t,x)\in \R\times D, \ z\prec w . %\qquad (t,x)\in\R\times\R^n,
\end{equation}
Furthermore, the optimal control $\omega%=\omega_{z,w}
$ for the problem \eqref{eq:control1}-\eqref{eq:optimal_curves} satisfies
\begin{equation}\label{eq:control}
 \omega%_{z,w}
 (s) = - \frac{\gamma_{1}(s)}{2} \partial_{\gamma_1}  \Psi\big(s,\gamma(s);w\big), \qquad s\in[t,T].
\end{equation}
%in the sense that, for any $\tau\in ]0,T[$, the following relation holds:
%\begin{equation}
%\Psi(t,x;T,y) = \Psi(\tau,x;T,y) -  \frac{1}{4} \int_t^{\tau}   \big| \nabla_x \Psi(s,x;T,y) \sigma(s,x)\big|^2   \dd s, \qquad t\in]0,\tau],\ x\in\R^n .
%\end{equation}
\end{lemma}
%(spiegazione) 
The proof is deferred until Appendix \ref{sec:proof_lemma_HJB}.

By applying \eqref{eq:fund_identity_dim_n} to \eqref{eq:LH_bis} we obtain
\begin{equation}\label{eq:LH_quat}
\Lc H_1(z;w) = f(z,w)   H_1(z;w).
\end{equation}
In order to get rid of the singularity of order $\frac{1}{T-t}$, we define the \emph{parametrix} $H$ as 
%\begin{remark}
%A simple change of variable shows that 
%\begin{equation}
%p(z;w) = p(z^{-1}\circ w;{\bf 0}), \qquad z,w\in \R\times D , \ z\prec w.
%\end{equation}
%\end{remark}
%\begin{lemma}
%There exists $\eps>0$ such that, for any $t\in\R$ and $y\in\R^n$, the function $\Psi(t,\cdot;T,y)$ is twice differentiable on $B_{\eps}(y)$ with bounded derivatives, uniformly with respect to $t\in\R$ and $y\in\R^n$.  
%\end{lemma}

%\begin{equation}
%\Psi(x,y) := \min_{\substack{\gamma\in C^1[0,1] \\ \gamma(0)=x,\, \gamma(1)=y}}  \int_0^1 \frac{1}{a\big( \gamma(s) \big)} \big( \dot{\gamma}(s) \big)^2 \dd s.
%\end{equation}
%\begin{equation}
%\begin{array}{ll}
%a  & := f(....)\\
%b \! & := G(...)
%\end{array}
%\end{equation}

%We define the parametrix
\begin{equation}\label{eq:true_parametrix}
H(z;w): =  H_1(z;w) {\bf u}(z;w),
\end{equation}
where 
%To prove the convergence of the series in \eqref{eq:parametrix_series}, we seek 
${\bf u}$ is a regular function, suitably chosen so as to obtain a uniform bound for $\Lc H(z,w)$. 
%In practice, we seek ${\bf u}$ such that 
%\begin{equation}\label{eq:Lparametrix}
%\Lc H(z;w) = H_1(z;w) {\bf c}(z;w), \qquad z,w \in \R\times D, \ z\prec w, 
%\end{equation}
%with ${\bf c}$ bounded with respect to $T-t$, and such that the initial condition% such that 
%\begin{equation}\label{eq:identity_u}
%{\bf u}(w;w) = 1%,\\
%\end{equation}
%is verified in some sense. Intuitively, the latter is required in order to preserve the Dirac Delta property, as $(T-t)\to 0^+$, for $H$. 
By \eqref{eq:LH_quat} we obtain
\begin{align}
\Lc H(z;w)& = \Big({f(z,w) {\bf u}(z;w) - g(z,w) \partial_{x_1} {\bf u}(z;w)  }  + \Lc {\bf u}(z;w)\Big) H_1(z;w)\\
& = \Big(  \big(Y - g(z;w) \partial_{x_1} + f(z,w) \big)  {\bf u}(z;w)  +\frac{ x_1^2}{2}  \partial_{x_1 x_1} {\bf u}(z;w) \Big) H_1(z;w) ,\label{eq:LH}
\end{align}
with $f$ as given in \eqref{eq:fg_bis} and
\begin{equation}\label{eq:fg}
%f(t,x;T,y) = \frac{2}{T-t} - \frac{x_1^2  \partial_{x_1 x_1}\Psi(t,x;T,y)}{4} , \qquad 
g(z;w): =\frac{ x_1^2}{2}  \partial_{x_1}\Psi(z;w).
\end{equation}
The idea is then to find ${\bf u}$ that verifies
\begin{equation}\label{eq:transport}
\big(Y - g(z;w) \partial_{x_1} \big)  {\bf u}(z;w)  +  f(z;w) {\bf u}(z;w) = 0 ,\qquad  z\prec w,
\end{equation}
and such that 
\begin{equation}
\frac{ x_1^2}{2}  \partial_{x_1 x_1} {\bf u}(z;w)
\end{equation}
is bounded with respect to $T-t$. We also impose that 
that the initial condition% such that 
\begin{equation}\label{eq:identity_u}
{\bf u}(w;w) = 1%,\\
\end{equation}
is verified in some sense. Intuitively, the latter is required in order for the Dirac Delta property, as $(T-t)\to 0^+$, to be transferred from $H_1$ to $H$. We seek a solution to \eqref{eq:transport}-\eqref{eq:identity_u} by employing the method of the characteristic curves. We have the following crucial   
\begin{lemma}\label{lem:radial_field}
For any $w=(T,y)\in\R\times D$, the integral curves of the vector field 
\begin{equation}
z\mapsto Y -  g(z,w) \partial_{x_1}  , \qquad z=( t, x)\in\R\times D,\ z\prec w,
\end{equation}
%is radial w.r.t. $(T,y)$, meaning that its integral curves 
are the optimal curves $\big(s,\gamma(s)\big)$ for the optimal control problem \eqref{eq:control1}-\eqref{eq:optimal_curves} with $\gamma(T)=y$.
\end{lemma}
\begin{proof}
% realizing the $\min$ in \eqref{def:geodesic_distance}. 
By \eqref{eq:control}, we have %\eqref{eq:control2} reads
\begin{equation}\label{eq:optimal_curves_bis}
\begin{cases}
\dot{\gamma}_{%z,w,
1}(s) = -  g\big(s,\gamma%_{z,w}
(s);w\big) \\
\dot{\gamma}_{%z,w,
2}(s) = \gamma_{%z,w,
1}(s)
%\Psi(0) = x, \ 
%\Psi(1) = y
\end{cases},\qquad
t<s<T,
%\qquad
%\text{and}\quad
%\begin{cases}
%\gamma_{z,w,1}(t) = x_1, \ \gamma_{z,w,1}(T) = y_1\\
%\gamma_{z,w,2}(t) = {x_2}, \ \gamma_{z,w,2}(T) = y_2
%\end{cases}
\end{equation}
%together with
%\begin{equation}\label{eq:optimal_curves_ter}
%%\begin{cases}
%%\dot{\gamma}_{z,w,1}(s) = -  g\big(s,\gamma_{z,w}(s);w\big) \\
%%\dot{\gamma}_{z,w,2}(s) = \gamma_{z,w,1}(s)
%%%\Psi(0) = x, \ 
%%%\Psi(1) = y
%%\end{cases},\qquad
%%t<s<T,\qquad
%%\text{and}\quad
%\begin{cases}
%\gamma_{z,w,1}(t) = x_1, \ \gamma_{z,w,1}(T) = y_1\\
%\gamma_{z,w,2}(t) = {x_2}, \ \gamma_{z,w,2}(T) = y_2
%\end{cases}.
%\end{equation}
which completes the proof.
\end{proof}
Applying now Lemma \ref{lem:radial_field}, together with the method of the characteristic curves, we find that the solution to \eqref{eq:transport}-\eqref{eq:identity_u} has to be
\begin{align}
{\bf u}(z;w) &:= \exp\bigg( \int_t^T  f\big(s,\gamma(s);T,y\big)  \dd s  \bigg) \\
& =  \exp\bigg( \int_t^T \Big[  \frac{2}{T-s} - \frac{\gamma_{%z,w,
1}^2(s)  \partial_{\gamma_1 \gamma_1}\Psi\big(s,\gamma%_{z,w}
(s);w\big)}{4} \Big]  \dd s  \bigg) ,%, \qquad t\leq T,\ x,y\in\R^n. 
%\qquad z\prec w,
\qquad z\prec w,
\label{eq:u}
\end{align}
%for any $z\prec w$, 
%Hereafter we let $u=u(z,w)$ be defined, for any $z=(t,x),w=(T,y)\in\R\times D$ with $t< T$ and $x<y$, as
where $\gamma%_{z,w}
$ denotes the optimal trajectory of the control problem \eqref{eq:control1}-\eqref{eq:optimal_curves}. In the next section we will prove that such ${\bf u}$ is well defined, and that the iterative construction described in Section \ref{sec:general_parametrix} converges to the  fundamental solution $p$ of $\Lc$, with the parametrix function $H$ given by \eqref{eq:true_parametrix}-\eqref{eq:parametrix_pre_explicit}-\eqref{eq:u}.

\subsection{Convergence of the Picard series and proofs of Theorems \ref{th:main} and \ref{th:lower}}

We prove Theorems \ref{th:main} and \ref{th:lower} by estimating the series \eqref{eq:parametrix_series} and the integral in \eqref{eq:ansaz_p} with $H$ as defined in the previous section, and by proving that $p$ as defined by \eqref{eq:ansaz_p} is indeed the fundamental solution of $\Lc$. 

We start with the following two propositions, whose proofs are deferred until Sections \ref{th:u} and \ref{prop:delta_Dirac}.
\begin{proposition}\label{th:u}
%Equation \eqref{eq:Lparametrix} holds with
%\begin{equation}\label{eq:c}
%{\bf c}(z;w) := \frac{x_1^2}{2}  \partial_{x_1 x_1}u(z;w).
%\end{equation}
The function ${\bf u}$ given by \eqref{eq:u} is well defined and we have
\begin{equation}\label{eq:rapp_u1}
{\bf u}(z;w) = v\Big(  {4}\, g^{-1}\Big(\frac{1}{{\bf h}(z;w)} \Big)  \Big), 
\end{equation}
with
\begin{equation}\label{eq:rapp_u2}
v(\eta) =\begin{cases}
   \frac{| \eta |}{2 \sqrt{3 \sqrt{\eta} \sinh(\sqrt{\eta})-6 \cosh(\sqrt{\eta})+6}} ,\quad & \eta\in\, ]-4\pi^2,+\infty[\, \setminus \{ 0 \} \\
   1 ,\quad & \eta =0
   \end{cases}.
%\end{cases}  .
\end{equation}
%In particular, the following estimate holds:
%\begin{equation}
%v(z) = 
%\begin{cases}
%             \frac{z}{2 \sqrt{3 \sqrt{z} \sinh(\sqrt{z})-6 \cosh(\sqrt{z})+6}}  ,  & \text{if } z>0  \\
%             1  ,  & \text{if } z=0  \\
%             -\frac{z}{2 \sqrt{6} \sqrt{\sin \left(\frac{\sqrt{-z}}{2}\right)} \sqrt{2 \sin \left(\frac{\sqrt{-z}}{2}\right)-\sqrt{-z} \cos \left(-\frac{\sqrt{-z}}{2}\right)}}  ,  & \text{if } z<0
%\end{cases}  .
%\end{equation}
Furthermore, %$v$ is a smooth function 
the function ${\bf u}(\cdot;\zeta)$ is smooth and solves \eqref{eq:transport}. 

Finally there exists a universal constant $\kappa>0$ such that  
\begin{align}%, \qquad \text{with}\qquad {\bf h}(z;w): = \frac{(T-t)\sqrt{x_1 y_1}}{y_2 - x_2},
& \kappa^{-1}  {\bf 1}_{]0,1]}\big(  {\bf h}(z;w) \big) {\bf h}(z;w)   \leq |{\bf u}(z;w)|  \leq \kappa \big( \sqrt{{\bf h}} + {\bf h}  \big) (z;w), \label{eq:est_u}\\
\label{eq:def_h}
%|{\bf c}(z;w)| 
&\big| { x_1^2}  \partial_{x_1 x_1} {\bf u}(z;w)  \big| \leq \kappa\, \sqrt{{\bf h}(z;w)},
\end{align}
for any $z=(t,x),w=(T,y)\in\R\times D$ with %$t< T$ and $x_2<y_2$
$z\prec w$.%, and the condition \eqref{eq:identity_u} is met in the sense that
%\begin{equation}\label{eq:terminal}
%\lim_{s\to T^-} u\big(s , \gamma_{z,w}(s) ;w\big) = 1. 
%\end{equation}
%If $x_1=y_1, x_2 = y_2 - (T-t) x_2$, then 
%\begin{equation}
%u\big(s , \gamma_{z,w}(s) ;w\big) = u\big(s , y_1, y_2 - (T-s )y_1 ;w\big) = 1, \qquad s\in [t,T[.
%\end{equation}
\end{proposition}
The next statement is crucial in order to check that parametrix $H$ enjoys the Dirac delta property ii) in Definition \ref{def:fund_sol}.%, and satisfies the boundary condition \eqref{eq:new}.
\begin{proposition}\label{prop:delta_Dirac}
For  
%$z=(T,x)%,w=(T,y)
%\in\R\times D
%$ %with $t< T$ and $x<y$, 
%and for 
any function $\tilde v\in C(\R^+)$ bounded %near zero and with with sub-linear growth at infinity
by a power function, such that $\tilde v(1)=1$, we have
%\begin{equation}
%\tilde v\big({\bf h}(\id;T,\cdot)\big) H_1(\id;T,\cdot) \longrightarrow \delta_{(1,0)},%\qquad H(z;s,\cdot)\longrightarrow \delta_{x}, 
%\qquad \text{as } T\to 0^{+},
%\end{equation}
%in the sense that
\begin{equation}\label{eq:delta_H1}
\lim_{%\substack{(T,x')\to \id \\ T>0 }
T\to 0^+}\ \int\limits_{(1,0)\prec \xi }\tilde v\big({\bf h}(%0,x'
\id;T, \xi)\big) H_1(%0,x'
\id;T, \xi) \varphi(\xi) \dd \xi  =  \varphi(1,0), \qquad \varphi\in C_b(D).%\\
%&\lim_{\substack{(s,x')\to (t,x)\\ s>t }}\int\limits_{\R^+\times ]x'_2,+\infty[}  H(t,x';s, \xi) \varphi(\xi) \dd \xi  =  \varphi(x),
\end{equation}
%holds for any $\varphi\in C_b(D)$, 
%and also
%\begin{equation}
%{\bf h}(s,\cdot;z) H_1(s,\cdot;z) \longrightarrow \delta_{x},%\qquad  H(s,\cdot;w)\longrightarrow \delta_{y}, 
%\qquad \text{as } s\to t^{-}.
%\end{equation}
%In particular
%\begin{equation}\label{eq:integral_H1}
%\int\limits_{x\prec \xi }\tilde v\big({\bf h}(0,x;T, \xi)\big) H_1(0,x;T, \xi)  \dd \xi = \int\limits_{(1,0)\prec \xi }\tilde v\big({\bf h}(\id;T, \xi)\big) H_1(\id;T, \xi)  \dd \xi \longrightarrow 1,\qquad \text{as }T\to 0^+,
%\end{equation}
%for any $x\in D$ such that $(1,0)\prec x$.
\end{proposition}
%In particular, if $\tilde v = v$ ... 

%\begin{proposition}\label{prop:zeroboundary}
%For any $T>0$ we have
%\begin{equation}
%\lim_{y_2\to 0^+  }\ \int_{0}^{T} \int_{\R^+}   H(\id;s, \xi_1,y_2) \varphi(s,\xi_1) \dd \xi_1 \dd s  =  0, \qquad \varphi\in C_b([0,T]\times \R^+).
%\end{equation}
%
%\end{proposition}

In order to carry on with the convergence analysis it is essential to provide bounds for the integrals in \eqref{eq:param_series} and \eqref{eq:ansaz_p}. In light of estimates \eqref{eq:def_h}-\eqref{eq:est_u}, such bounds are consequences of the following

\begin{keyestimate}\label{prop:estimate_CK}
%Let ${\bf c}$ be defined as in \eqref{eq:c}. Then 
Let ${\bf h}$ be as defined in \eqref{eq:def_h} and $\tilde g$ be the function
\begin{equation}
\tilde g(h) = \sqrt{h} + h.
\end{equation}
For any $\tau>0$, there exists a constant $C_{\tau}>0$ such that 
\begin{align}\label{eq:key_est_1}
\int_{x\prec\xi\prec (1,0)} \sqrt{{\bf h}(z;s,\xi) } H_1(z;s,\xi) \sqrt{ {\bf h}(s,\xi;\id)}  H_1(s,\x;\id)  \dd \xi & \leq C_{\tau}\, \sqrt{{\bf h}(z; \id)} H_1(z,\id), \qquad s\in]t,0[,\qquad\qquad \\ \label{eq:key_est_2}
\int_{x\prec\xi\prec (1,0)} \tilde{g}\big({{\bf h}(z;s,\xi) }\big) H_1(z;s,\xi)\tilde g\big({ {\bf h}(s,\xi;\id)}\big)  H_1(s,\x;\id)  \dd \xi & \leq C_{\tau}\,\tilde g\big({{\bf h}(z;\id)}\big) H_1(z,\id), \qquad s\in]t,0[,\qquad
\end{align}
for any $z=(t,x)%,w=(T,y)
\in \R\times D$ such that %$t<T$ and $y_2>x_2$
$z\prec \id$ and $t>-\tau$.
\end{keyestimate}

As already pointed out in the introduction, at the current stage we were not able to provide a proof for the bounds in Key Inequalities \ref{prop:estimate_CK}. In Section \ref{prop:estimate_CK} we collect a considerable amount of numerical evidence in favor of the claim that these estimates hold true. This makes us comfortable in conjecturing their validity. However, given that a rigorous proof is currently missing, the Key Inequalities \ref{prop:estimate_CK} are part of the hypotheses of Theorems \ref{th:main} and \ref{th:lower}.

%\begin{remark}
%Compare with Chapman-Kolmogorov equations, or reproduction formulas.
%\end{remark}
%
\begin{theorem}\label{th:convergence}
Assume that Key Inequalities \ref{prop:estimate_CK} hold true. Then the series in \eqref{eq:parametrix_series} converges, and for any $\tau>0$ there exists a positive constant $C>0$, only dependent on $\tau$, such that 
\begin{equation}\label{eq:bound_conv}
\Big| \int_{z \prec \zeta\prec w}%\limits_{\quad\qquad\qquad]0,T[\times \R^+ \times ]0,y_2[}  \hspace{-35pt} 
 H(z;\zeta)  \Phi(\zeta;w)  \dd \zeta \Big|  \leq   C (T-t) \big( \sqrt{{\bf h}(z;w)} +{\bf h}(z;w) \big) H_1(z,w),
\end{equation}
for any $z=(t,x),\zeta=(T,w)\in \R\times D$ with $z\prec w$ and $T-t<\tau$.

Furthermore, the function $p$ given by \eqref{eq:ansaz_p} is the fundamental solution of $\Lc$ with $\sigma=1$.

\end{theorem}

Before proving Theorem \ref{th:convergence}, we prove Theorems \ref{th:main} and \ref{th:lower}, which are straightforward consequences of Theorem \ref{th:convergence}.

\begin{proof}[Proof of Theorem \ref{th:main}]
In light of Remark \ref{rem:invariance}, it is not restrictive to assume $\sigma=1$. Thus the bound \eqref{eq:main_estimate} stems from Theorem \ref{th:convergence}, in particular by applying the estimates \eqref{eq:est_u}-\eqref{eq:bound_conv} to \eqref{eq:ansaz_p}.
%Therefore, the function $p$ defined by \eqref{eq:ansaz_p} satisfies
%\begin{equation}
%|p(z;\id)| \leq C_4 \big( \sqrt{{\bf h}(z;\id)} +{\bf h}(z;\id) \big) H_1(z,\id),
%\end{equation}
%which is \eqref{eq:main_estimate} with $w=\id$. In order to conclude, we only need to prove that $p(z;\id)$ is the fundamental solution of $\Lc$. 
\end{proof}

\begin{proof}[Proof of Theorem \ref{th:lower}]

In light of Remark \ref{rem:invariance}, it is not restrictive to assume $\sigma=1$ and $w=\id$. To ease notation we remove the explicit dependence on $\id$ in the functions below. %We also denote by $C$ the constant appearing in \eqref{}, which only depends on $\tau$ at most. 

All the following inequalities are meant for any $z=(t,x)\in\, ]-1,0[\,\times D$ such that %$x\prec (1,0)$ and 
\begin{equation}\label{eq:compact}
\frac{1}{\kappa}\leq \frac{\sqrt{x_1}}{-x_2}\leq \kappa. %estimates hold 
\end{equation}
By Theorem \ref{th:convergence} we obtain

\begin{equation}\label{eq:diseqp}
 p(t,x)   \geq H(t,x) - \bigg| \int_{(t,x) \prec \zeta\prec w}  H(t,x;\zeta)  \Phi(\zeta)  \dd \zeta \bigg| 
 \geq  H(t,x) -  C (-t) \big(\sqrt{{\bf h}}+{\bf h}\big)(t,x)   H_1(t,x), %\qquad t\in ]-1,0[,
\end{equation}
where $C$ is a universal constant independent of any variable. Furthermore, by \eqref{eq:compact}, we have 
\begin{equation}
{\bf h}(t,x) \leq 1,\qquad t\in [-{\kappa}^{-1},0[,
\end{equation}
and thus \eqref{eq:diseqp} together with the first equality in \eqref{eq:est_u} yield
\begin{align}
p(t,x)   & \geq \Big( {\bf h}(t,x)   -  C (-t) \big(\sqrt{{\bf h}}+{\bf h}\big)(t,x)   \Big) H_1(t,x)\\
%& = \Big( 1   -  C (-t) \big(\sqrt{{\bf h}}/{\bf h}+1\big)(t,x)   \Big) {\bf h}(t,x) H_1(t,x)\\
& = \bigg( 1-C(-t) -  C \sqrt{-t}\Big({\frac{-x_2}{\sqrt{x_1}}}\Big)^{1/2}   \bigg) {\bf h}(t,x) H_1(t,x)
\intertext{(by the first inequality in \eqref{eq:compact})}
& \geq \big( 1-C(-t) -  C \sqrt{-t}\sqrt{\kappa} \,  \big) {\bf h}(t,x) H_1(t,x)
\end{align}
for any $t\in [-{\kappa}^{-1},0[$. This completes the proof. 
\end{proof}

We conclude the section with the proof of Theorem \ref{th:convergence}.
\begin{proof}[Proof of Theorem \ref{th:convergence}]
We first prove convergence of the series in \eqref{eq:parametrix_series} and estimate \eqref{eq:bound_conv}. In light of \eqref{eq:simmetry1}-\eqref{eq:Psi_invariance}, it is not restrictive to assume $w=\id$. To ease notation we remove the explicit dependence on $\id$ in the functions below. Furthermore, we will denote by $C_1, C_2, \dots$ any positive constant that depends at most on $\tau$. 

By Proposition \ref{th:u}, in particular by applying \eqref{eq:transport} to \eqref{eq:LH}, we obtain 
\begin{equation}\label{eq:LH_ter}
\Lc H(z,\zeta) = \frac{ x_1^2}{2}  \partial_{x_1 x_1} {\bf u}(z,\zeta) H_1(z,\zeta), \qquad z\prec \zeta.
\end{equation}
Therefore, by \eqref{eq:param_series} and employing \eqref{eq:def_h}, we obtain 
\begin{equation}
|K_0(z)| \leq  C_1 \sqrt{{\bf h}(z)} H_1(z).
\end{equation}
In general, by induction we can also prove%and by iteration,
\begin{equation}\label{eq:estim_Kn}
|K_n(z)| \leq  C_2^{n} C_1^{n+1} \frac{(-t)^n}{n!} \sqrt{{\bf h}(z)} H_1(z)
\end{equation}
for any $n\in \mathbb{N}_0$. Indeed, assuming \eqref{eq:estim_Kn} true, %by \eqref{}, 
we have
\begin{align}
|K_{n+1}(z)| & \leq \frac{C_2^{n} C_1^{n+1}}{n!}   \int_{z \prec \zeta\prec \id} (-s)^n \big| \Lc H(z,\zeta) \big|   \sqrt{{\bf h}(\zeta)} H_1(\zeta)   \dd \zeta
\intertext{(by  \eqref{eq:LH_ter}-\eqref{eq:def_h})}
 & \leq \frac{C_2^{n} C_1^{n+2}}{n!} \int_{z \prec \zeta\prec \id} (-s)^n \sqrt{{\bf h}(z;\zeta)}  H_1(z;\zeta)   \sqrt{{\bf h}(\zeta)} H_1(\zeta)   \dd \zeta
 \intertext{(by \eqref{eq:key_est_1})}
 & \leq \frac{C_2^{n+1} C_1^{n+2}}{n!} \int_t^0  (-s)^n \dd s \, \sqrt{{\bf h}(z)} H_1(z)\\
 &\leq C_2^{n+1} C_1^{n+2} \frac{(-t)^{n+1}}{(n+1)!} \sqrt{{\bf h}(z)} H_1(z).
\end{align}
Summing over $n$, we obtain that the series in \eqref{eq:parametrix_series} converges and that
\begin{equation}
|\Phi(z) | \leq  C_1\, e^{- C_1 C_2 t} %e^{- C_3  t} 
\sqrt{{\bf h}(z)} H_1(z).
\end{equation}
This, together with \eqref{eq:est_u}-\eqref{eq:key_est_2}, yields \eqref{eq:bound_conv}.
%\begin{equation}
%\Big| \int_{z \prec \zeta\prec \id}%\limits_{\quad\qquad\qquad]0,T[\times \R^+ \times ]0,y_2[}  \hspace{-35pt} 
% H(z;\zeta)  \Phi(\zeta;\id)  \dd \zeta \Big|  \leq   C_3 (-t) \big( \sqrt{{\bf h}(z;\id)} +{\bf h}(z;\id) \big) H_1(z,\id). 
%\end{equation}

\vspace{5pt}

We now go on to prove the second part of the statement, namely that $p$ as defined by \eqref{eq:ansaz_p} is the fundamental solution of $\Lc$ with $\sigma=1$.  

For any $x'\in D$ and $t<s$, and for any $\varphi\in C_b(D)$, we have
\begin{align}
\int_{x'\prec y} H(t,x';T,y) \varphi(y) \dd y& = \int_{x'\prec y} {\bf u}(t,x';T,y )  H_1(t,x';T,y) \varphi(y) \dd y  
\intertext{(by \eqref{eq:simmetry1}-\eqref{eq:Psi_invariance})}
& = \int_{x'\prec y} {\bf u}\big(\id; (t,x')^{-1}\circ(T,y) \big) (x'_1)^{-2}  H_1\big(\id; (t,x')^{-1}\circ(T,y) \big) \varphi(y) \dd y \\
& = \int_{(1,0)\prec \xi} {\bf u}(\id; T-t, \xi)  H_1(\id;  T-t, \xi) \varphi_{x'}( \xi) \dd \xi ,
\end{align}
where we used the notation
\begin{equation}
\varphi_{x}(\xi) := \big( \xi_1 x_1^{-1}, x_2 +  \xi_2 x_1^{-1}  \big).
\end{equation}
Therefore, employing Proposition \ref{prop:delta_Dirac}, it is straightforward to show that 
\begin{equation}\label{eq:delta_H}
\lim_{\substack{(t,x')\to z\\ t<T} } \int_{x'\prec y} H(t,x';T,y) \varphi(y) \dd y = \varphi(x)%\qquad \varphi\in C_b(D),
\end{equation}
for any $z=(T,x)\in\R\times D$ and $\varphi\in C_b(D)$. In analogous way, we can employ estimate \eqref{eq:bound_conv} together with Theorem \ref{prop:delta_Dirac}, to prove
\begin{equation}
\lim_{\substack{(t,x')\to z\\ t<T} } \int_{x'\prec y} \bigg(    \int_{(t,x') \prec \zeta\prec (T,y)}%\limits_{\quad\qquad\qquad]0,T[\times \R^+ \times ]0,y_2[}  \hspace{-35pt} 
 H(t,x';\zeta)  \Phi(\zeta;T,y)  \dd \zeta  \bigg) \varphi(y) \dd y = 0
\end{equation}
for any $z=(T,x)\in\R\times D$ and $\varphi\in C_b(D)$. This together with \eqref{eq:delta_H} imply that $p$ satisfies property ii) in Definition \ref{def:fund_sol}.% for $p$. 

We now prove property i) in Definition \ref{def:fund_sol}, namely that $p(\cdot,w)$ as defined by \eqref{eq:ansaz_p} satisfies \eqref{eq:Lsolved}. By \eqref{eq:simmetry1}-\eqref{eq:Psi_invariance} it is easy to show that
\begin{equation}
y_1^2 p(z;w) = %\frac{1}{x_1^2}
 p(w^{-1}\circ z;\id). %= p\Big( t-T, \frac{x_1}{y_1}, \frac{x_2-y_2}{y_1} ;  \id  \Big)
\end{equation}
In light of this, and of the left invariance of $\Lc$ (see \eqref{eq:left_invariance}), it is not restrictive to set $w=\id$. Once more, to ease notation we remove the explicit dependence on $\id$ in the functions below. We need to show that
\begin{equation}\label{eq:pde_distrib}
\Lc p(z)=0,\qquad z=(t,x)\in \Omega,
\end{equation}
with 
%Since the operator $\Lc$ is hypoelliptic, it is enough to show that $p=p(z)$ is a solution to $\Lc p=0$ on 
$\Omega: = \R^- \times \R^+ \times \R^-$. Since the operator $\Lc$ is hypoelliptic, it is enough to show that $p$ solves \eqref{eq:pde_distrib} in the sense of distributions, namely
\begin{equation}\label{eq:pde_distrib_bis}
\int_{\Omega} p(z) \tilde\Lc \phi(z) \dd z = 0, \qquad \phi\in C_{0}^{\infty}(\Omega),
\end{equation}
where the operator $\tilde\Lc$ denotes the formal adjoint of $\Lc$, i.e.
\begin{equation}
\tilde\Lc \phi(z) = -\partial_t \phi(z)-  x_1 \partial_{x_2} \phi(z)  + \frac{1}{2} \partial_{x_1 x_1} \big(  x_1^2  \phi(z)  \big).
\end{equation}

We first prove that, for any $\zeta=(s,\xi) \in \Omega$, we have
\begin{equation}\label{eq:key_equality}
\int_{z\prec \zeta }   H(z;\zeta) \tilde\Lc \phi (z)  \dd z = \int_{z\prec \zeta }  \phi (z)  \Lc H(z;\zeta)  \dd z  - \phi(\zeta)  ,\qquad \phi\in C_0^{\infty}(\Omega).
\end{equation}
Note that, proceeding like we did above to prove \eqref{eq:delta_H} (we skip the details for brevity), we obtain %by applying \eqref{}
\begin{equation}\label{eq:delta_H_bis}
\lim_{t\to s^- } \int_{x\prec \eta} H(t,x;\zeta) \varphi(x) \dd x = \varphi(\eta),\qquad \varphi\in C_b(D).
\end{equation}
%for any $(t,x_1)\in\R\times \R^+$ and $\varphi\in C_b(D)$. %(scrivere che non lo mostriamo)
Furthermore, since ${\bf h}(t,x_1,x_2;\zeta)\to +\infty$ as $x_2\to \xi_2^-$, \eqref{eq:rep_psi}-\eqref{eq:asymptotic_G_right} together with \eqref{eq:est_u} simply yield 
\begin{equation}\label{eq:new_bis}
H(t,x_1,x_2;\zeta)\to 0 \quad \text{as } x_2\to \xi_2^-, \qquad t<s,\, x_1\in \R^+.
\end{equation}
For any $\delta>0$ we obtain
\begin{align}
 \int_{-\infty}^{s-\delta}  \int_{x\prec \xi}  %\int_{-\infty}^{\xi_2}   
 H(t,x;s,\xi) \tilde\Lc \phi (t,x)& \dd x  \dd t        
 \intertext{(integrating by parts w.r.t. $x$ and using the boundary condition \eqref{eq:new_bis})}
& =   \int_{-\infty}^{s-\delta}  \int_{x\prec \xi}  %\int_{-\infty}^{\xi_2}   
 \Lc H(t,x;s,\xi) \phi (t,x) \dd x  \dd t +   \int_{x\prec \xi}   H(s-\delta,x;s,\xi) \phi (s-\delta,x) \dd x .
\end{align}
Passing to the limit as $\delta\to 0^+$, together with \eqref{eq:delta_H_bis}, yields \eqref{eq:key_equality}.
 %Furthermore, 
%employing again the symmetries \eqref{eq:simmetry1}-\eqref{eq:Psi_invariance} together with Proposition \ref{prop:zeroboundary}, we obtain
%\begin{equation}\label{eq:boundarynull}
%\lim_{\xi_2\to x_2^+}\int_{t}^{0} \int_{\R^+}   H(z; s, \xi_1,\xi_2) \varphi(s,\xi_1) \dd \xi_1 \dd s = 0,
%\end{equation}
%for any $z=(t,x)\in \R\times D$ with $t<0$, %and for any %$\xi_2 > x_2$, we have 
%and for any $\varphi\in C_b(\,]t,0[ \times \R_+)$.

We now go on to prove \eqref{eq:pde_distrib_bis}. Integrating by parts, we clearly obtain
\begin{equation}\label{eq:testH}
\int_{\Omega}  H(z;\id) \tilde\Lc \phi (z)     \dd z = \int_{\Omega}   \phi (z)\Lc H(z;\id)     \dd z, \qquad \phi\in C_{0}^{\infty}(\Omega).
%Applying now \eqref{eq:delta_H_bis}-\eqref{eq:boundarynull} we obtain
\end{equation}
Furthermore,
\begin{align}
&\int_{\Omega}  \bigg(  \int_{z\prec \zeta \prec \id}  H(z;\zeta) \Phi(\zeta;\id) \dd \zeta \bigg) \tilde\Lc \phi (z) \dd  z   
\intertext{(by Fubini's Theorem)}
& =  \int_{\Omega}   \bigg(  \int_{z\prec \zeta }   H(z;\zeta) \tilde\Lc \phi (z)  \dd z \bigg) \Phi(\zeta;\id) \dd  \zeta   
\intertext{(%integrating by parts and applying \eqref{eq:delta_H_bis}-\eqref{eq:boundarynull}
by \eqref{eq:key_equality})}
& =  \int_{\Omega}   \bigg(  \int_{z\prec \zeta }  \phi (z)  \Lc H(z;\zeta)  \dd z - \phi(\zeta) \bigg) \Phi(\zeta;\id) \dd  \zeta  
\intertext{(again by Fubini's Theorem)}
& =\int_{\Omega}   \bigg(  \int_{z\prec \zeta \prec \id }    \Lc H(z;\zeta)\, \Phi(\zeta;\id)  \dd \zeta  \bigg)  \phi (z) \dd  z - \int_{\Omega}  \phi(\zeta)  \Phi(\zeta;\id) \dd  \zeta
\intertext{(by construction $\Phi$ solves \eqref{eq:picard_PHI})}
& = - \int_{\Omega}    \phi (z) \Lc H(z;\id)   \dd z   
\end{align}
for any $\phi\in C_{0}^{\infty}(\Omega)$. 

This, together with \eqref{eq:testH} and \eqref{eq:ansaz_p}, proves \eqref{eq:pde_distrib_bis} and completes the proof.

%$\lfloor$
%For any $z\in \R\times D$, we denote by 
%\begin{equation}
%\psi_{z}(h)  = \big(t+ h, \psi_{z,1}(h), \psi_{2,z}(h) \big)  , \qquad h\in \R,
%\end{equation}
%the (smooth) integral curve of the (smooth) vector field $\Lc$, starting at $z$. For any $z\prec\id$ and $h$ near $0$, we have the following telescoping sum:
%\begin{equation}
%\frac{1}{h} \Big(    p\big(  \psi_z(h) \big)  - p(z)  \Big)   =   \frac{1}{h} \sum_{j=1}^3 I_j,
%\end{equation}
% where
% \begin{align}
% I_1  &=  H\big(  \psi_z(h) \big)  - H(z)        ,     \\
% I_2  &=  
%\int_{z \prec \zeta\prec \id} \Big( H\big(  \psi_z(h)  ;\zeta\big) - H(z;\zeta)   \Big)  \Phi(\zeta)  \dd \zeta     ,     \\
%  I_3  &= \int_{\psi_z(h) \prec \zeta\prec \id}  H\big(  \psi_z(h)  ;\zeta\big)   \Phi(\zeta)  \dd \zeta  -  \int_{z \prec \zeta\prec \id}  H\big(  \psi_z(h)  ;\zeta\big)   \Phi(\zeta)  \dd \zeta    .  %,     \\
%%   I_4  &=         ,     \\
%%    I_5  &=         ,     \\
%%     I_6  &=         ,     \\
% \end{align}
%Now, since $H$ is smooth on its domain, we have
%\begin{equation}
%\lim_{h\to 0} \frac{H\big( \psi_z(h)  ; \zeta   \big) - H(z;\zeta)}{h}   =   \Lc H(z,\zeta), \qquad z\prec \zeta.
%\end{equation}

\end{proof}

%\begin{notation}
%We will denote by $\sigma(t,x)$ an invertible $(n\times n)$-matrix such that ${a(t,x)}=(\sigma^\top \sigma)(t,x)$.
%\end{notation}

%Set $a(x)=\sigma^2(x)$ and let 

%\subsection{Proof of Theorems \ref{th:main} and \ref{th:lower}}

%\section{Proofs}
\subsection{Proof of Proposition \ref{th:u}}

We start with the following
\begin{lemma}\label{lemm:lim_integrand_u}
For any $z=(t,x),w=(T,y)\in\R\times D$ with $z\prec \zeta$, denoting by $\gamma=\gamma(s)=\big(\gamma_1(s),\gamma_2(s)\big)$ the optimal trajectory of the control problem \eqref{eq:control1}-\eqref{eq:optimal_curves}, we have
%\begin{equation}
%\gamma_{1}^2(s)  \partial_{\gamma_1 \gamma_1}\Psi\big(s,\gamma(s);T,y\big) =\frac{ w\Big(  {4}\frac{(T-s)^2}{(T-t)^2}\, g^{-1}\big(\frac{1}{{\bf h}(z;w)} \big)  \Big)}{T-s},
%\end{equation}
\begin{equation}
\gamma_{1}^2(s)  \partial_{\gamma_1 \gamma_1}\Psi\big(s,\gamma(s);T,y\big) =\frac{ h\Big(  \frac{(T-s)^2}{(T-t)^2}\, g^{-1}\big(\frac{1}{{\bf h}(z;w)} \big)  \Big)}{T-s}, \qquad s\in [t,T[,
\end{equation}
with
%\begin{equation}
%w(z) =
%\begin{cases}
%        \frac{2 \sqrt{z} \coth \left(\frac{1}{2} \sqrt{z} \right)-z  \left(\text{\emph{csch}}^2\left(\frac{1}{2} \sqrt{z} \right)+2\right)}{2-\sqrt{z}  \coth \left(\frac{1}{2} \sqrt{z} \right)}, &\quad z\in\, ]-4\pi^2,+\infty[\, \setminus \{ 0 \}   \\ 
%   8, &\quad z=0 
%\end{cases}.
%\end{equation}
%\begin{equation}
%w(z) =
%\begin{cases}
%     2   \frac{\sqrt{z} \coth ( \sqrt{z})-z  \left(\coth^2( \sqrt{z})+1\right)}{1-\sqrt{z}  \coth( \sqrt{z})}, &\quad z\in\, ]-\pi^2,+\infty[\, \setminus \{ 0 \}   \\ 
%   8, &\quad z=0 
%\end{cases}.
%\end{equation}
\begin{equation}\label{eq:w}
h(\eta) =
\begin{cases}
     2   \sqrt{\eta} \coth ( \sqrt{\eta}) - \frac{2\eta}{1-\sqrt{\eta}  \coth( \sqrt{\eta})}   , &\quad \eta\in\, ]-\pi^2,+\infty[\, \setminus \{ 0 \}   \\ 
   8, &\quad z=0 
\end{cases}.
\end{equation}

%and where $E=E_{z,w}$ is as defined in \eqref{def:E}-\eqref{def:g}.
\end{lemma}
\begin{proof} 

In \cite{cibelli2019sharp} it was shown that the optimal curve $\gamma$ is given by
\begin{equation}\label{eq:optimal_curve_expl}
\gamma%_{z,w}
(s) =\begin{cases}
   \Big(   \frac{ 4 y_1 }{(   k(T-s) +2 )^2 }   ,  y_2 - \frac{  2(T-s)  y_1  }{  k(T-s) +2  }   \Big),     &  \text{if }   E=0       \vspace{5pt}\\   
   \bigg(   \frac{ E y_1 }{\big(    \sqrt{E} \cosh \big( \frac{ T-s }{ 2 } \sqrt{E}  \big) + k \sinh \big(  \frac{ T-s }{ 2 } \sqrt{E} \big)   \big)^2 }   ,  
   y_2 - \frac{  2  \sinh\big(   \frac{ T-s }{ 2 } \sqrt{E}  \big)  y_1  }{  \sqrt{E} \cosh \big( \frac{ T-s }{ 2 } \sqrt{E}  \big) + k \sinh \big(  \frac{ T-s }{ 2 } \sqrt{E}\big)  }   \bigg),     &  \text{if }   E>0    \vspace{5pt}  \\ 
  \bigg(   \frac{ - E y_1 }{\big(    \sqrt{-E} \cos \big( \frac{ T-s }{ 2 } \sqrt{-E}  \big) + k \sin \big(  \frac{ T-s }{ 2 }\sqrt{-E} \big)   \big)^2 }   ,  
   y_2 - \frac{  2  \sin\big(   \frac{ T-s }{ 2 } \sqrt{-E}  \big)  y_1  }{  \sqrt{-E} \cos \big( \frac{ T-s }{ 2 } \sqrt{-E}  \big) + k \sin \big(  \frac{ T-s }{ 2 } \sqrt{-E} \big)  }   \bigg),     &  \text{if }   E<0      
\end{cases}
\end{equation}
where $E=E_{z,w}$ is as defined in \eqref{def:E}-\eqref{def:g}, and where
\begin{equation}
k = k_{z,w} := \begin{cases}
        \frac{ 2 y_1 }{  y_2 - x_2  } - \sqrt{E + \frac{ 4 x_1 y_1 }{ (y_2 - x_2)^2  } } ,   &  \text{if }  E\geq  -\frac{ \pi^2 }{( T-t)^2 }               \\  
        -  \frac{ 2 y_1 }{  y_2 - x_2  } - \sqrt{E + \frac{ 4 x_1 y_1 }{ (y_2 - x_2)^2  } } ,   &  \text{if }  -\frac{ 4\pi^2 }{ ( T-t)^2 }<  E <  -\frac{ \pi^2 }{ ( T-t)^2 }
\end{cases}.
\end{equation}
We also have
\begin{equation}
\frac{ y_2 - \gamma_{2}(s)}{(T-s)\sqrt{\gamma_{1}(s) y_1}} = \frac{ 2 \sinh\Big(  \frac{(T-s) \sqrt{E}}{2}  \Big) }{ (T-s) \sqrt{E}} ,
\end{equation}
which implies
\begin{equation}\label{eq:constant_energy}
g^{-1}\bigg( \frac{ y_2 - \gamma_{2}(s)}{(T-s)\sqrt{\gamma_{1}(s) y_1}} \bigg) = 
%\begin{cases}
%  1,       &  \text{if }   E=0         \\  
%  \frac{(T-s)^2 E}{4} , &   \text{if }   E>0
%\end{cases}
\frac{(T-s)^2 E}{4}.
\end{equation}
%In particular, the energy $E_{\gamma(s),w}$ is a constant of motion along the optimal curves.

The statement now follows from a direct computation of $\partial_{x_1 x_1}\Psi(t,x;T,y)$, together with \eqref{eq:optimal_curve_expl} and \eqref{eq:constant_energy}.
\end{proof}
We are now in the position to prove Theorem \ref{th:u}.

\begin{proof}[Proof of Theorem \ref{th:u}]

Fix $z=(t,x), w=(T,y) \in \R\times D$ such that $z\prec w$, and denote as usual by $\gamma=\gamma(s)$ the optimal curve connecting $z$ to $w$.
By Lemma \ref{lemm:lim_integrand_u} together with \eqref{eq:fg_bis} we obtain 
\begin{equation}
f\big(s,\gamma(s);T,y\big) := \frac{2 - \frac{1}{4} h\Big(  \frac{(T-s)^2}{(T-t)^2}\, g^{-1}\big(\frac{1}{{\bf h}(z;w)} \big)  \Big)}{T-s}, \qquad s\in [t,T[.  
\end{equation}
A simple change of variables now yields
\begin{equation}\label{eq:int_change_tau}
\int_t^T f\big(t,\gamma(s);T,y\big) \dd s = \int_{0}^1 \tilde f (\tau)   \dd \tau ,
\end{equation}
with
\begin{equation}
 \tilde f (\tau) = \frac{2 - \frac{1}{4} h\Big(  \tau^2 g^{-1}\big(\frac{1}{{\bf h}(z;w)} \big)  \Big)}{\tau} 
\end{equation}
and $h$ as in \eqref{eq:w}.
By Taylor-expanding the function $h(z)$ around $z=0$ we obtain
\begin{equation}
2 - \frac{1}{4}h (z) = -\frac{16}{15} z+O\big(z^{3/2}\big)\qquad \text{as } z\to 0.
\end{equation}
Furthermore, it can be checked that 
\begin{equation}
1-\sqrt{\eta}  \coth( \sqrt{\eta}) \neq 0, \qquad \eta\in]-\pi^2,+\infty[\setminus \{ 0 \},
\end{equation}
and thus $\tilde f$ is continuous on $[0,1]$. 
Therefore, the integral in \eqref{eq:int_change_tau} is well defined, and so is the function ${\bf u}$ as given by \eqref{eq:u}. 

We now go on to prove that ${\bf u}$ can be represented as in \eqref{eq:rapp_u1}-\eqref{eq:rapp_u2}, and that it is smooth. Set
\begin{equation}
a:=4 g^{-1}\bigg(\frac{1}{{\bf h}(z;w)} \bigg).
\end{equation}
If ${\bf h}(z;w) =1$, then $a =0$ and \eqref{eq:rapp_u1}-\eqref{eq:rapp_u2} is trivially satisfied. If ${\bf h}(z;w) \neq 1$, then $a \neq 0$ and it is easy to check that the function 
\begin{equation}
\tilde F(\tau):= 2 \log \tau-\frac{1}{2} \log \Big(\sqrt{ a} \tau \sinh ( \sqrt{ a} \tau)-2\cosh( \sqrt{ a} \tau)+2\Big)%, \qquad a:=4 g^{-1}\bigg(\frac{1}{{\bf h}(z;w)} \bigg),
\end{equation}
is a primitive of $\tilde f$ on $]0,\infty[$. Thus we obtain
\begin{align}
\int_{0}^1 \tilde f (\tau)   \dd \tau &= \tilde F(1) - \lim_{\tau\to 0^+} \tilde F(\tau) = \\
& = -\frac{\log \big(\sqrt{a} \sinh(\sqrt{a})-2 \cosh(\sqrt{a})+2\big)}{2} + \frac{1}{2}\log \left(\frac{a^2}{12}\right).
\end{align}
This yields
\begin{equation}\label{eq:rep_u_proof}
{\bf u} (z;\zeta) = \frac{|a|}{2 \sqrt{3 \sqrt{a} \sinh\big(\sqrt{a}\big)-6 \cosh\big(\sqrt{a}\big)+6}}, \qquad a:=4 g^{-1}\bigg(\frac{1}{{\bf h}(z;w)} \bigg),
\end{equation}
which is \eqref{eq:rapp_u1}-\eqref{eq:rapp_u2}. Furthermore, by Taylor expanding the functions $\sinh$ and $\cosh$ around zero we obtain 
\begin{equation}
\frac{1}{v(\eta)} =\sqrt{ 1 + \sum_{n=1}^{\infty}  \frac{12}{(2n+3)!} \frac{n+1}{n+2} \eta^n  }, \qquad \eta\in ]-4\pi^2, + \infty [.
\end{equation}
This shows that $v$ is smooth on $]-4\pi^2,+\infty[$, and since $g^{-1}$ and ${\bf h}$ are smooth, then ${\bf u}$ is smooth on its existence domain. 

We now prove that ${\bf u}(\cdot;\zeta)$ solves \eqref{eq:transport}. We employ the method of the characteristics. Fix again $z=(t,x), w=(T,y) \in \R\times D$ such that $z\prec w$, and recall that $\gamma=\gamma(s)$ denotes the optimal curve connecting $z$ to $w$.
By \eqref{eq:u} we have
\begin{equation}
{\bf u}\big(   r, \gamma(r)  ; w  \big)  =  e^{\int_r^T  f\left( s, \gamma(s)  ; w \right)  \dd s},  \qquad r\in [t,T[.
\end{equation}
%Noting that $f\left( s, \gamma(s)  ; w \right)$ is continuous on $]-\infty, T [$
This yields
\begin{equation}\label{eq:uuu}
\frac{\dd }{\dd r} {\bf u}\big(   r, \gamma(r)  ; w  \big) = - f\big( r, \gamma(r)  ; w \big) {\bf u}\big(   r, \gamma(r)  ; w  \big), \qquad r\in [t,T[. 
\end{equation}
On the other hand the smoothness of ${\bf u}$, together with Lemma \ref{lem:radial_field}, implies
\begin{equation}\label{eq:uuuu}
\frac{\dd }{\dd r} {\bf u}\big(   r, \gamma(r)  ; w  \big) = \big\langle \big(1, \dot\gamma(r)\big)  , \nabla  {\bf u}\big(   r, \gamma(r)  ; w  \big) \big\rangle  =
\big(\partial_t + \gamma_1(r) \partial_{x_2} - g(\gamma(r);w) \partial_{x_1} \big)  {\bf u}\big(\gamma(r);w\big),  \qquad r\in [t,T[. 
\end{equation}
Setting $r=t$ in \eqref{eq:uuu} and \eqref{eq:uuuu} we obtain \eqref{eq:transport}.

We now prove the bounds in \eqref{eq:est_u}. %-\eqref{eq:def_h}.
Set %$\varsigma,\theta,\chi\omega$
\begin{equation}
r(\rho):= \frac{g^{-1}(\rho)}{\log^2 \rho}, \qquad \rho>>1, 
\end{equation}
and obtain 
\begin{equation}\label{eq:expr_u_new}
v\big( 4 g^{-1} ( \rho) \big) =  v\big( 4  r(\rho) \log^2 \rho \big) 
= \rho^{-\sqrt{r(\rho)}}( \log \rho)^{\frac{3}{2}} R\big(\rho, r(\rho) \big),
\end{equation}
with 
%\begin{equation}
%R(h,r)  :=  \frac{2r}{\sqrt{
%3 \sqrt{r} \big(1-h^{-4\sqrt{r}}\big) - \frac{3}{\log h} \big(1+ h^{-4\sqrt{r}}\big) + 6 \frac{h^{-2 \sqrt{r}}}{\log h}
%}} 
%\end{equation}
\begin{equation}
R(\rho,r)  := {2r}{\bigg({
3 \sqrt{r} \big(1-\rho^{-4\sqrt{r}}\big) - \frac{3}{\log \rho} \big(1+ \rho^{-4\sqrt{r}}\big) + 6 \frac{\rho^{-2 \sqrt{r}}}{\log \rho}
}\bigg)^{-\frac{1}{2}}} .
\end{equation}
By \eqref{eq:asymp_ginv}, which is
\begin{equation}
r(\rho) \longrightarrow 1\qquad \text{as } \rho\to +\infty,
\end{equation}
we obtain
\begin{equation} \label{eq:lim_R}
R\big(\rho, r(\rho)\big) \longrightarrow \frac{2}{\sqrt{3}}\qquad \text{as } \rho\to +\infty,
\end{equation}
and
\begin{equation}\label{eq:lim_zero}
\rho^{-\sqrt{r(\rho)}+\frac{1}{2}}( \log \rho)^{\frac{3}{2}} \longrightarrow 0\qquad \text{as } \rho\to +\infty.
\end{equation}
Furthermore the definition of $g$ yields 
\begin{equation}
-\sqrt{g^{-1}(\rho)}+\log \rho = - \log \bigg(     \sqrt{g^{-1}(\rho)}  +  \frac{e^{-\sqrt{g^{-1}(\rho)}}}{\rho}     \bigg),
\end{equation}
which in turn implies
\begin{equation}
\rho^{-\sqrt{r(\rho)}+1} = e^{-\sqrt{g^{-1}(\rho)}+\log \rho} = \frac{1}{\sqrt{g^{-1}(\rho)}  +  \frac{e^{-\sqrt{g^{-1}(\rho)}}}{\rho}}.
\end{equation}
Applying again \eqref{eq:asymp_ginv}, the latter yields 
\begin{equation}
\rho^{-\sqrt{r(\rho)}+1} \sim \frac{1}{\log \rho}\qquad \text{as } \rho\to +\infty
\end{equation}
and thus 
\begin{equation}\label{eq:lim_infty}
\rho^{-\sqrt{r(\rho)}+1}( \log \rho)^{\frac{3}{2}} \longrightarrow +\infty \qquad \text{as } \rho\to +\infty.
\end{equation}
Now, plugging \eqref{eq:lim_R}-\eqref{eq:lim_zero}-\eqref{eq:lim_infty} into \eqref{eq:expr_u_new} yields 
\begin{equation}\label{eq:asymp_v_big}
\rho^{-1} < v\big( 4 g^{-1} ( \rho) \big) < \rho^{-1/2},\qquad \rho>>1.
\end{equation}
We now prove 
\begin{equation}\label{eq:asymp_v}
 v\big( 4 g^{-1} ( \rho) \big) < \rho^{-1},\qquad \rho<<1,
\end{equation}
which is equivalent to 
\begin{equation}\label{eq:asymp_v_b}
g(\eta) v(    4 \eta  )  <1,\qquad 0>\eta + \pi^2 <<1.
\end{equation}
Noting that 
\begin{equation}
g(\eta), \frac{1}{ v(    4 \eta  )}\longrightarrow 0 \qquad \text{as } \eta \to   -\pi^2_+,
\end{equation}
we study the limit of the derivatives. We find  
\begin{align}
\frac{d}{d \eta}  g(\eta) & = \frac{ \sqrt{\eta} \cosh \sqrt{\eta} - \sinh \sqrt{\eta}}{2 \eta^{{3}/{2}} } \longrightarrow \frac{1}{2 \pi ^2}  %\qquad \text{as } \eta \to   -\pi^2_+ 
,  \\
\frac{d}{d \eta} \frac{1}{ v(    4 \eta  )} &= \frac{\sqrt{\frac{3}{2}}\, \Big(\!-5 \sqrt{\eta } \sinh \left(2 \sqrt{\eta }\right)+2 (\eta +2) \cosh \left(2 \sqrt{\eta }\right)-4 \Big)}{4 \eta  \sqrt{\eta ^2 \left(\sqrt{\eta } \sinh \left(2 \sqrt{\eta }\right)-\cosh \left(2 \sqrt{\eta }\right)+1\right)}} \longrightarrow +\infty  % \qquad \text{as } \eta \to   -\pi^2_+ .
,
\end{align}
as $\eta \to   -\pi^2_+$. Therefore, applying L'H\^opital's rule yields 
\begin{equation}
g(\eta)v(    4 \eta  ) \longrightarrow 0 \qquad \text{as } \eta \to   -\pi^2_+,
\end{equation}
which in turn implies \eqref{eq:asymp_v_b} and thus \eqref{eq:asymp_v}. Finally, \eqref{eq:asymp_v_big}-\eqref{eq:asymp_v}, together with the continuity of $v$ and Weierstrass extreme value  theorem, proves \eqref{eq:est_u}. 

Eventually, analogous arguments allow to prove \eqref{eq:def_h}, using the representation
\begin{equation}
{ x_1^2}  \partial_{x_1 x_1} {\bf u}(z;w) = \Big(  \rho \big(g^{-1}\big)'(\rho) + \rho^2 \big(g^{-1}\big)''(\rho)  \Big) v'\big(   4 g^{-1}(\rho)   \big)  + \Big( 2 \rho \big(g^{-1}\big)'(\rho) \Big)^2 v''\big(   4 g^{-1}(\rho)   \big), \qquad \rho=\frac{1}{{\bf h}(z;w)}.
\end{equation}
The details are left to the reader for the sake of brevity.
\end{proof}

%\begin{lemma}
%For any $z=(t,x),w=(T,y)\in\R\times D$ with $t\leq T$, we have
%\begin{equation}
%u(z,w) = v\big((T-t)^2 E_{z,w} \big), 
%\end{equation}
%with
%\begin{equation}
%v(z) = 
%\begin{cases}
%             \frac{z}{2 \sqrt{3 \sqrt{z} \sinh(\sqrt{z})-6 \cosh(\sqrt{z})+6}}  ,  & \text{if } z>0  \\
%             1  ,  & \text{if } z=0  \\
%             -\frac{z}{2 \sqrt{6} \sqrt{\sin \left(\frac{\sqrt{-z}}{2}\right)} \sqrt{2 \sin \left(\frac{\sqrt{-z}}{2}\right)-\sqrt{-z} \cos \left(-\frac{\sqrt{-z}}{2}\right)}}  ,  & \text{if } z<0
%\end{cases}  .
%\end{equation}
%
%\end{lemma}
%
%We are now ready to prove Theorem \ref{th:u}.
%
%\begin{proof}[Proof of Theorem \ref{th:u}]
%In particular,
%\begin{equation}
%\lim_{s\to T^-}  \frac{2}{T-s} - \frac{\gamma_{z,w,1}^2(s)  \partial_{\gamma_1 \gamma_1}\Psi\big(s,\gamma_{z,w}(s);T,y\big)}{4}   = 0.
%\end{equation}
%
%\end{proof}

\subsection{Proof of Proposition \ref{prop:delta_Dirac}}

%\begin{lemma}
%Let $\Phi:\R_{>0}^3\to \R_{>0}$ be defined as
%\begin{equation}
%\Phi(\tau,\xi,\zeta) :=  \frac{\sqrt{12}}{\pi \tau}  \exp\bigg( -2\frac{g^{-1}\big(\xi^{-1}\big)  + \xi \big( \zeta + \zeta^{-1} \big) - 2\,\text{\emph{sgn}}\big( 4 g^{-1}(\xi^{-1}) + \pi^2 \big)  \sqrt{g^{-1}(\xi^{-1}) + \xi^2  }  }{\tau}    \bigg).
%\end{equation}
%Then we have
%\begin{equation}
%\lim_{\tau\to 0^+} \int_{\R_{>0}^2}  f(\xi,\zeta)  \Phi(\tau,\xi,\zeta)  \dd \xi \dd \zeta  =  f(1,1)
%\end{equation}
%for any $f$ continuous and such that ...
%\end{lemma}
%\begin{proof}
%...
%\end{proof}

\begin{proof}[Proof of Proposition \ref{prop:delta_Dirac}]
For %fixed $z=(t,x)\in \R\times D$ and 
$T>0$ %and $x'\in D$ 
we set %consider the integral
\begin{equation}\label{int_H1}
I(T%,x'
):=\int_{(1,0)\prec \xi} g\big( {\bf h}(\id;T,\xi) \big) H_1(\id;T,\xi)  \varphi(\xi) \dd \xi . %= \frac{\sqrt{12}}{2\pi (T-t)^2} \int_{\R^+\times ]x_2,\infty[} {\bf h}(z;T,y) \frac{ e^{-\frac{1}{2}\Psi(z;T,y)}}{ y^2_1}  \dd y_1 \dd y_2.
\end{equation}
By \eqref{eq:rep_psi}-\eqref{eq:G}, and by the change of variables
\begin{equation}
\begin{cases}
 \chi = \sqrt{\xi_1}                 \\  
 \eta = {\bf h}(\id;T,\xi)% \frac{\sqrt{x_1 y_1}(T-t)}{y_2 - x_2}
\end{cases},\qquad (1,0)\prec \xi,%\in \R^+ \times ]x_2,\infty[,
\end{equation}
the integrals in \eqref{int_H1} can be written as
\begin{equation}
I(T%,x'
)= \frac{\sqrt{12}}{\pi  T}
\int\limits_{\R^+\times \R^+} \eta^{-2} \chi^{-2} g(\eta)
\exp\bigg( -2\,\frac{  \eta\big(\chi + \chi^{-1} -2\big) +    G( \eta ) }  { T}    \bigg) \bar\varphi(\chi,\eta) \dd \eta \dd \chi,
\end{equation}
with
\begin{equation}
\bar\varphi(\chi,\eta) = \varphi \bigg(  \chi^2  ,  \frac{T \chi}{\eta}   \bigg).
\end{equation}
Employing now $G(1)=G'(1)=0$ together with $G''(1)=6$ (see \eqref{eq:def_Gn}-\eqref{eq:def_a_b} with $a_2 = 3$), we can write the $2$nd order Taylor expansion 
%\begin{equation}
%\eta\big(\chi + \chi^{-1} -2\big) +    G( \eta )
%\end{equation}
%around $(\chi,\eta)=(1,1)$, which reads as
\begin{equation}
\eta\big(\chi + \chi^{-1} -2\big) +    G( \eta ) =  (\chi-1)^2 + 3 (\eta-1)^2 + R(\chi,\eta),  
\end{equation}
where $R$ is a continuous function such that 
\begin{equation}
R(\chi,\eta) = o(\chi^2) + o(\eta^2),\qquad \text{as } (\chi,\eta)\to (1,1).
\end{equation}
Thus we have
\begin{equation}
I(T) = I_1(T)+I_2(T)+I_3(T),
\end{equation}
with
\begin{align}
I_1(T) &= \frac{\sqrt{12}}{\pi  T}
\int_{([\frac{1}{2},\frac{3}{2}])^2} \eta^{-2} \chi^{-2} g(\eta)
\exp\bigg(\! -\frac{  4(\chi-1)^2 + 12 (\eta-1)^2 }  {2  T}    \bigg) \bar\varphi(\chi,\eta) \dd \eta \dd \chi,\\
I_2(T) &= \frac{\sqrt{12}}{\pi  T}
\int_{([\frac{1}{2},\frac{3}{2}])^2} \eta^{-2} \chi^{-2} g(\eta)
\exp\bigg(\! -\frac{  4(\chi-1)^2 + 12 (\eta-1)^2 }  {2  T}    \bigg) \Big(e^{-\frac{2 R(\chi,\eta) }  { T}} -1  \Big) \bar\varphi(\chi,\eta) \dd \eta \dd \chi,\\
I_3(T) &= \frac{\sqrt{12}}{\pi  T}
\int_{(\R^+)^2\setminus([\frac{1}{2},\frac{3}{2}])^2} \eta^{-2} \chi^{-2} g(\eta)
\exp\bigg(\! -2\,\frac{  \eta\big(\chi + \chi^{-1} -2\big) +    G( \eta ) }  { T}    \bigg) \bar\varphi(\chi,\eta) \dd \eta \dd \chi.
\end{align}
We now study each term separately. Regarding $I_1$, it is enough to observe that the function
\begin{equation}
(\chi,\eta)\mapsto\frac{\sqrt{12}}{\pi  T} \exp\bigg(\! -\frac{  4(\chi-1)^2 + 12 (\eta-1)^2 }  {2  T}    \bigg)
\end{equation}
is a Gaussian probability density. Owing to the continuity of $g$ and $\varphi$, and to the hypothesis $g(1)=1$, it is standard to show that 
\begin{equation}
\lim_{T\to 0^+} I_1(T) = \bar\varphi(1,1) = \varphi(1,0).
\end{equation}
We now address $I_2$. By using
\begin{equation}
e^{-\frac{2 R(\chi,\eta) }  { T}} -1 = o\big((\chi-1)^2\big) + o\big((\eta-1)^2\big),\qquad \text{as } (\chi,\eta)\to (1,1),
\end{equation}
 it is again standard to prove that
 \begin{equation}
\lim_{T\to 0^+} I_2(T) = 0.
\end{equation}
We now employ
\begin{equation}
\frac{1}{T} \exp\bigg(\! -2\,\frac{  \eta\big(\chi + \chi^{-1} -2\big) +    G( \eta ) }  { T}    \bigg) \longrightarrow 0 \quad \text{as } T\to 0^{+},\qquad (\chi,\eta)\neq (1,1),
\end{equation}
together with 
\begin{equation}
\eta\big(\chi + \chi^{-1} -2\big) +    G( \eta ) \geq c>0,\qquad (\chi,\eta)\in (\R^+)^2\setminus([{1}/{2},{3}/{2}])^2
\end{equation}
and the fact that $\varphi\in C_b(D)$, to apply Lebesgue dominated convergence theorem so as to obtain 
 \begin{equation}
\lim_{T\to 0^+} I_3(T) = 0.
\end{equation}
This proves \eqref{eq:delta_H1} and concludes the proof. %Finally, $I(T,x)$ does not depend on $x$ if $\varphi\equiv 1$, and the limit in \eqref{eq:integral_H1} is a particular case of \eqref{eq:delta_H1}.%by setting $\varphi\equiv 1$.
%\begin{align}
%&\int\limits_{\R^+\times ]x_2,\infty[} {\bf h}(z;T,y) \frac{ e^{-\frac{1}{2}\Psi(z;T,y)}}{ y^2_1}  \dd y_1 \dd y_2 \\
%& = 2 (T-t)\hspace{-18pt} \int\limits_{\R^+\times \R^+} \hspace{-18pt}\xi^{-1} \zeta^{-2}  
%\exp\bigg( -2\frac{g^{-1}\big(\xi^{-1}\big)  + \xi \big( \zeta + \zeta^{-1} \big) - 2\,\text{sgn}\big( 4 g^{-1}(\xi^{-1}) + \pi^2 \big)  \sqrt{g^{-1}(\xi^{-1}) + \xi^2  }  }{T-t}    \bigg) \dd \xi \dd \zeta.
%\end{align}
%This shows that the integral in \eqref{int_H1} only depends on $T-t$, while it is independent of $x=(x_1,x_2)$. We now show that it converges to $1$ as $(T-t)\to 0^{+}$.
\end{proof}

%\subsection{Proof of Proposition \ref{prop:zeroboundary}}
%\eqref{eq:rep_psi}
%
%We write $H(\id;s, \xi_1,y_2)$ as 
%\begin{align}
%H(\id;s, \xi_1,y_2) & \leq \frac{\kappa \sqrt{12}}{2\pi } \frac{s^2}{y_2^4} e^{- \frac{2}{\s^2 y_2}(\xi_1 +1) }  \\
%& \qquad \times  \exp\bigg({-\frac{2 G\big( {\bf h}(\id;s, \xi_1,y_2) \big) - 4 {\bf h}(\id;s, \xi_1,y_2)}{\s^2 s}   } \bigg) \bigg(\frac{ \sqrt{{\bf h}} + {\bf h}}{{\bf h}^4}  \bigg) (\id;s, \xi_1,y_2)
%\end{align}
%
%\begin{proof}[Proof of Proposition \ref{prop:zeroboundary}]
%
%\end{proof}

\section{Numerical evidences for Key Inequalities \ref{prop:estimate_CK}}\label{sec:numerical_evidence}

In this section we collect numerical evidence in favor of the validity of the Key Inequalities \ref{prop:estimate_CK}. Unfortunately, so far, we were unable to provide a full mathematical proof. %of their validity. 
However, given the stability and the extensiveness of the numerical tests that we performed, we feel comfortable to conjecture that the Key Inequalities \ref{prop:estimate_CK} are true statements. The completion of a rigorous proof seems challenging and remains an open problem, which we defer to further research.

From the numerical point of view, the two issues that one has to overcome in order to verify \eqref{eq:key_est_1}-\eqref{eq:key_est_2} are the following:
\begin{enumerate}
\item[a)] The cost function $\Psi$ involves the inverse of a trigonometric function. Thus the complexity of the numerical inversion adds up to the complexity of the numerical integration. The result is a loss of precision and an increase of the computational time.
\item[b)] The integrals are to be computed on unbounded domains. This makes the numerical integration harder to manage. 
\end{enumerate}

To tackle Point a) above, we utilize the approximation of the cost function given by
\begin{equation}
\tilde{\Psi}_N (z,w)  =\frac{4}{T-t}\bigg[ {\bf h}(z,w)\bigg(\sqrt{\frac{y_1}{x_1}} + \sqrt{\frac{x_1}{y_1}} -2\bigg) +  \sum_{n=2}^{N}  \tilde{G}_n\big( {\bf h}(z,w) \big) \bigg],  \qquad N\geq 2,
\end{equation}
where
\begin{equation}\label{eq:def_Gn_tilde}
\tilde{G}_n(\eta): = a_n {\bf 1}_{]1,\infty[} (\eta)  \frac{ (\eta -1)^n}{\eta^{n-1} } +\tilde{b}_n {\bf 1}_{]0,1[}(\eta)  \frac{ (-\log \eta )^n}{\big(1-\log \eta\big)^{n-1}} ,
\end{equation}
and where the coefficients $a_n,\tilde{b}_n$ are recursively determined by solving the equations
\begin{equation}\label{eq:def_a_b_tilde}
\frac{\dd^n}{\dd \eta^n}{G}(\eta)\Big|_{\eta=1} = \frac{\dd^n}{\dd\eta^n} \sum_{k=2}^{n}  \tilde{G}_k (\eta ) \Big|_{\eta=1}, \qquad n\geq 2 .
\end{equation}
The $N$-th order approximation $\tilde{\Psi}_N$ is clearly very close to the one in Definition \ref{def:N_approx}. In fact, the functions $\tilde{G}_n$ coincide with the functions $G_n$ on $]1,\infty[$, while on $]0,1[$ the two are only slightly different. The choice $\tilde{G}_n$ clearly does not preserve the asymptotic properties enjoyed by $G_n$ described in Section \ref{sec:novel_repres}. However, the advantage of using here $\tilde{\Psi}_N$ in place of $\Psi_N$ is that the former approximates $\Psi$ from below, namely
\begin{equation}\label{eq:Psitil_ineq}
\tilde{\Psi}_N (z,w) \leq \Psi(z,w),\qquad z,w\in \R\times D, \ z\prec w.
\end{equation}
Indeed, in Section \ref{sec:novel_repres} we have already showed that 
\begin{equation}\label{eq:GtillessG}
\sum_{n=2}^{N}  \tilde{G}_n(\eta) < G(\eta)
\end{equation}
for any $\eta>1$. This comes from the fact that the coefficients $a_n$ are positive and $G_n$ converges point-wise to $G$. On the other hand, the plots in Figure \ref{fig:ste3} provide numerical evidence of the fact that \eqref{eq:GtillessG} is satisfied for $\eta\in]0,1[$ too. In particular, we test it for $N=50$, which is the order used in the numerical tests below. 
\begin{figure}[htb]
%\title{\centering Difference between $G(\eta)%-\sum_{n=2}^N G_n(\eta)
%$ and its $N$-th order approximation\\ ${}$}
\centering
\includegraphics[width=1\textwidth,height=0.28\textheight]{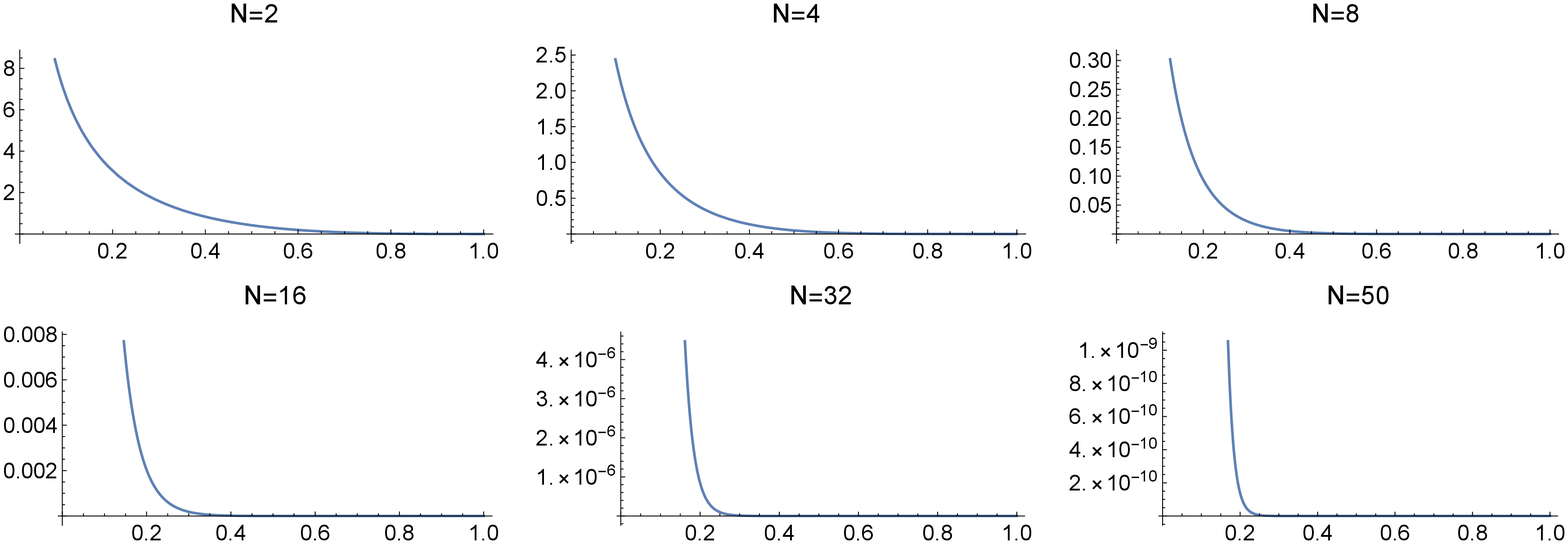} 
\caption{Plot of the difference $G(\eta)-\sum_{n=2}^N b_n \tilde{G}_n(\eta)$, with $\tilde{G}_n$ and $\tilde{b}_n$ as defined by \eqref{eq:def_Gn_tilde}-\eqref{eq:def_a_b_tilde}.} 
\label{fig:ste3}
\end{figure} 

In order to address Point b) above, we make use of the Dirac Delta property of the function $H_1$ proved in Proposition \ref{prop:delta_Dirac}, in order to reduce the integration domains to bounded sets. We have the following 
\begin{proposition}\label{prop:reduction}
The Key Inequalities \ref{prop:estimate_CK} are satisfied if and only if there exists a constant $\kappa_{\tau}>0$ such that 
\begin{align}
\int_{x\prec \xi \prec (1,0)} {\bf 1}_{D(t,s,x)}(\xi) \sqrt{{\bf h}(z;s,\xi) } H_1(z;s,\xi) \sqrt{ {\bf h}(s,\xi;\id)}  H_1(s,\x;\id)  \dd \xi & \leq \kappa_{\tau}\, \sqrt{{\bf h}(z; \id)} H_1(z,\id),\qquad
\label{eq:key_est_1_bis} \\%\qquad s\in]t,0[, 
\int_{x\prec \xi \prec (1,0)} {\bf 1}_{D'(t,s,x)}(\xi) \tilde{g}\big({{\bf h}(z;s,\xi) }\big) H_1(z;s,\xi) \tilde{g}\big({ {\bf h}(s,\xi;\id)}\big)  H_1(s,\x;\id)  \dd \xi & \leq \kappa_{\tau}\, \tilde{g}\big({{\bf h}(z;\id)}\big) H_1(z,\id),\qquad
\label{eq:key_est_2_bis}% \qquad s\in]t,0[,
\end{align}
for any $s\in]t,0[$, and for any $z=(t,x)%,w=(T,y)
\in \R\times D$ such that %$t<T$ and $y_2>x_2$
$z\prec \id$ and $t>-\tau$,
where $D(x),D'(x)\subset D$ are the bounded domains
\begin{align}
  D(t,s,x)      &:= \Big\{ \xi\in D :\Big( \sqrt{{\bf h}(z;s,\xi) }  H_1(z;s,\xi) \Big) \wedge \Big( \sqrt{ {\bf h}(s,\xi;\id)}  H_1(s,\x;\id) \Big) >  \sqrt{{\bf h}(z; \id)} H_1(z,\id) \Big\} , \qquad \label{eq:Dtsx} \\
   D'(t,s,x)    & := \Big\{ \xi\in D :\Big( \tilde{g}\big({{\bf h}(z;s,\xi) }\big)  H_1(z;s,\xi) \Big) \wedge \Big( \tilde{g}\big({ {\bf h}(s,\xi;\id)}\big)  H_1(s,\x;\id) \Big) >  \tilde{g}\big({{\bf h}(z; \id)}\big) H_1(z,\id) \Big\}   .
\end{align}
\end{proposition}
\begin{proof}
By \eqref{eq:Dtsx} we have 
%Proposition \ref{prop:delta_Dirac}, we have that 
\begin{align}
& \int_{x\prec \xi \prec (1,0)} {\bf 1}_{D\setminus D(t,s,x)}(\xi)  \sqrt{{\bf h}(z;s,\xi) }  H_1(z;s,\xi) \sqrt{ {\bf h}(s,\xi;\id)}  H_1(s,\x;\id) \,  \dd \xi \\
& \leq  \sqrt{{\bf h}(z; \id)} H_1(z,\id) \bigg(    \int_{x\prec \xi }  \sqrt{{\bf h}(z;s,\xi) }  H_1(z;s,\xi) \dd \xi    
  +   \int_{ \xi \prec (1,0)}  
  \sqrt{ {\bf h}(s,\xi;\id)}\, {H_1(s,\x;\id)}   \dd \xi \bigg) .  \label{eq:integrali}
\end{align}
Now, by Proposition \ref{prop:delta_Dirac} we obtain
\begin{equation}
\int_{x\prec \xi }  \sqrt{{\bf h}(z;s,\xi) }  H_1(z;s,\xi)\, \dd \xi %&=  \int_{x\prec \xi }  \sqrt{{\bf h}(0,x;s-t,\xi) }  H_1(0,x;s-t,\xi) \,\dd\xi
%\intertext{(by Proposition \ref{prop:delta_Dirac})}
% &
= \int_{(1,0)\prec \xi } \sqrt{{\bf h}(\id;s-t, \xi)} H_1(\id;s-t, \xi) \, \dd \xi \longrightarrow 1,\qquad \text{as }s-t\to 0^+, 
\end{equation}
and analogously 
\begin{equation}
\int_{ \xi \prec (1,0) }   \sqrt{ {\bf h}(s,\xi;\id)}  \,{H_1(s,\x;\id)}\, \dd \xi %&=  \int_{x\prec \xi }  \sqrt{{\bf h}(0,x;-s,\xi) }  H_1(0,x;-s,\xi) \,\dd\xi
%\intertext{(by ...)}
 %&
 = \int_{(1,0)\prec \xi } \sqrt{{\bf h}(\id;-s, \xi)} H_1(\id;-s, \xi) \, \dd \xi \longrightarrow 1,\qquad \text{as }s\to 0^-.
\end{equation}
In particular, we have proved that the two integrals in \eqref{eq:integrali} are bounded, for any $-\tau<t<s<0$ and $x\in D$, by a constant that only depends on $\tau$. Therefore, \eqref{eq:key_est_1_bis} is equivalent to \eqref{eq:key_est_1}. The proof that \eqref{eq:key_est_2_bis} is equivalent to \eqref{eq:key_est_2} is identical. 
%This concludes the proof. 
 \end{proof}

\subsection{Numerical tests and results}

In force of \eqref{eq:Psitil_ineq} and of Proposition \ref{prop:reduction} above, in order to prove the Key Inequalities \ref{prop:estimate_CK} it is enough to check that  
\begin{align}
I_1(t,s,x):=\frac{\int_{]x_2,0[ \times ]0,\bar\xi[}  \sqrt{{\bf h}(z;s,\xi) } \tilde{H}_1(z;s,\xi) \sqrt{ {\bf h}(s,\xi;\id)}  \tilde{H}_1(s,\x;\id)  \dd \xi}
{\sqrt{{\bf h}(z; \id)} H_1(z,\id)} & \leq \kappa_{\tau},
\label{eq:key_est_1_ter} \\%\qquad s\in]t,0[, 
I_2(t,s,x):= \frac{\int_{]x_2,0[ \times ]0,\bar{\bar\xi}[}  %\tilde{g}
\big(\sqrt{\bf h} + {\bf h} \big) (z;s,\xi)  \tilde{H}_1(z;s,\xi)\, \big(\sqrt{\bf h} + {\bf h} \big)(s,\xi;\id)  \tilde{H}_1(s,\x;\id)  \dd \xi}{\big(\sqrt{\bf h} + {\bf h} \big)(z;\id) H_1(z,\id)} & \leq \kappa_{\tau},
\label{eq:key_est_2_ter}% \qquad s\in]t,0[,
\end{align}
for any $z=(t,x)%,w=(T,y)
\in \R\times D$ such that %$t<T$ and $y_2>x_2$
$z\prec \id$ and $t>-\tau$, where $\tilde{H}_1$ is defined as $H_1$ with $\Psi$ replaced by $\tilde\Psi_N$ for a given $N>2$, and where $\bar\xi , \bar{\bar\xi}>0$ are such that
\begin{equation}
]x_2,0[ \times ]0,\bar\xi[\, \supset D(t,s,x)\,\cap  \big(\, ]x_2,0[ \times \R^+ \big) ,\qquad ]x_2,0[ \times ]0,\bar{\bar\xi}[\, \supset  D'(t,s,x)\,\cap  \big(\, ]x_2,0[ \times \R^+ \big).
\end{equation}
It is a direct computation to show that a possible choice for such $\bar\xi , \bar{\bar\xi}$ is given by 

\begin{align}
\bar\xi &= \max\bigg\{  1,     -  \frac{x_2}{4} \Big(\Psi (t,x_1,x_2; \id)-3 \log \frac{t-s}{t}+\log 2\Big)                       ,       \frac{x_2}{8}    v^*  \Big(  \frac{8x_1}{x_2} 
e^{-2\big(  \Psi (t,x_1,x_2; \id)   - 3\log \frac{s}{t}   +\log 2  \big)  + \frac{8x_1}{x_2}}            \Big)             \bigg\}, \\ \label{eq:xibar} \\
%\end{equation}
%\begin{equation}
\bar{\bar\xi} &= \max\bigg\{  1,     -  \frac{x_2}{4} \Big(\Psi (t,x_1,x_2; \id)-3 \log \frac{t-s}{t}\Big)                       ,       \frac{x_2}{8}    v^*  \Big(  \frac{8x_1}{x_2} 
e^{-2\big(  \Psi (t,x_1,x_2; \id)   - 3\log \frac{s}{t}     \big)  + \frac{8x_1}{x_2}}            \Big)  ,     \\
 \label{eq:xibarbar}
&                     \qquad   \qquad     \frac{x_2}{4}    v^*  \Big(  \frac{4x_1}{x_2} 
e^{-\big(  \Psi (t,x_1,x_2; \id)   - 3\log \frac{s}{t}     \big)  + \frac{4x_1}{x_2}}            \Big)    \bigg\},
\end{align}
with $v^*(\nu)$ being the lowest real solution\footnote{The equation $\nu = v e^v$ has two negative real solutions if $\nu = c \,a e^a $ with $a<0$ and $c\in]0,1[$. Thus the values of $v^*$ in \eqref{eq:xibar}-\eqref{eq:xibarbar} are well-defined.} to $\nu = v e^v$. The scope for turning the domains of integration into rectangles is facilitating the numerical integration.

 %Indeed, 
%\begin{equation}
%\nu = v e^v.
%\end{equation}
%Here above, $\tilde{H}_1$ is defined as $H_1$ with $\Psi$ replaced by $\tilde\Psi_N$ for a given $N>2$.

We provide numerical evidence of the fact that \eqref{eq:key_est_1_ter}-\eqref{eq:key_est_2_ter} are satisfied for $\tau=1$, with $\kappa_{\tau} \approx 2.5$ and $N=50$. We test the inequalities for multiple choices of the parameters $s, t,  x $, generated at random  
%For six different choices of $\varsigma>0$ we compute the integrals above for $10^6$ different choices of 
%the parameters $s, t,  x $ %in the following way: %generated at random 
according to the following distributions:
\begin{equation}\label{eq:distrib_unif}
t \sim \text{Unif}_{[-1, 0]},\qquad s \sim \text{Unif}_{[t, 0]},
\end{equation}
and 
\begin{equation}\label{eq:distrib_lognormal}
x_1 = \chi^2 ,\qquad x_2 = \frac{t \chi }{\eta } , \qquad \text{with } (\chi,\eta) \sim \text{LogNorm}_{0, 2} \otimes \text{LogNorm}_{0,2}.
\end{equation}
The distribution above is justified by the representation \eqref{eq:rep_psi}, which can be written here as 
\begin{equation}%\label{eq:rep_psi}
\Psi(t,x;\id) =\frac{4}{- t}\bigg( \eta\Big(\chi + \frac{1}{\chi} -2\Big) +    G( \eta ) \bigg),\qquad \eta = {\bf h}(t,x;\id) = \frac{t \sqrt{x_1}}{x_2}, \ \chi = \sqrt{x_1},
\end{equation}
%with $\eta = {\bf h}(t,x;\id) = \frac{t \sqrt{x_1}}{x_2}$ and $\chi = \sqrt{x_1}$, 
and which shows that ${\Psi}$ reaches its minimum at $(\chi,\eta)=(1,1)$. %For the %four different 
%value $\omicron = 2$% $\omicron =0.5$, $1$, $1.5$ and $2$%, $2.5$ and $3$
%, w
We generated $2\times 10^7$ vectors $(s, t,  x) $ and computed the integrals in the LHS of \eqref{eq:key_est_1_ter}-\eqref{eq:key_est_2_ter} using Monte Carlo integration. In Table \ref{tab:max} we report the maximum value of $I_1(t,s,x),I_2(t,s,x)$ and the corresponding $\arg\max$ obtained for the batch of variables. In Figure \ref{fig:ste4} we plot the $2\times 10^7$ samples of the pair $(\log\eta,\log\chi)$ generated at random with distribution $\mathcal{N}_{0,2} \otimes \mathcal{N}_{0,2}$, emphasizing the two argmax at which the maxima of $I_1$ and $I_2$ in Table \ref{tab:max} are attained. 

The results do not present evidence of the fact that the functions $I_1$ and $I_2$ are unbounded. Note that, due to the chosen distributions for $(\eta,\chi)$ and $(t,s)$, the values of $I_1$ and $I_2$ were computed in a region that includes the extreme tails of the exponential kernel $e^{-\Psi(t,x;\id)/2}$, and hence in a region where the denominators in \eqref{eq:key_est_1_ter}-\eqref{eq:key_est_2_ter} can be very small. For instance, for $t=-1/2,s=-1/4$ and $\chi=\xi = e^8$ one has $e^{-\Psi(t,x;\id)/2}\approx3.24\times10^{-126}$.

%With regards to this, we stress that the corresponding values of $(\chi,\eta)$ allow to test the inequalities \eqref{eq:key_est_1_ter}-\eqref{eq:key_est_2_ter} for a wide range of values of $x=(x_1,x_2)$, including those that correspond to the extreme tails of the exponential kernel $e^{\Psi(t,x;\id)}$: for instance, for $t=-1/2,s=-1/4$ and $\chi=\xi = e^8$ one has $e^{\Psi(t,x;\id)}\approx3.24\times10^{-126}$.

%The results show that this maximum is not sensitive to increasing the parameter $\omicron $, which is equivalent to choosing extreme values of $x\in \R^-\times \R^+$ with a higher probability. Figure \ref{fig:ste4} contains the plot of the pairs $(\log\chi, \log \eta)%\in \R^-\times \R^+ $ in each batch.  
 The numerical integration was performed using Wolfram Mathematica built-in numerical integration routine \verb+NIntegrate+ with the method \verb+AdaptiveMonteCarlo+. The Mathematica notebook used to generate the results reported above can be found in the supplementary material. 
 \begin{table}
\centering
\begin{tabular}{c|ccccccc}
$n$ & $t^*$  &  $s^*$  &  $x^*_1$ & $x^*_2$  & $\log\chi^*$  & $\log\eta^*$  &  $I_n(t^*,s^*,x^*)$     \\
\hline
1&      -0.81948 & -0.54040 & 43.64581 & -0.12497 & 1.88805  & 3.76863  & 1.74841
\\
2 &      -0.78426 & -0.67014 & 0.00002 & -0.00605 & -5.52146 & -0.54472 & 2.48050  
\end{tabular}
    \caption{The maxima $I_n(t^*,s^*,x^*) = \max_{t,s,x}I_n(t,s,x)$ and %$I_2(t^*,s^*,x^*) = \max_{t,s,x}I_2(t,s,x)$ 
    relative $\text{argmax}$, with $t,s,x$ ranging over a batch counting $2\times 10^7$ samples generated at random in accordance with \eqref{eq:distrib_unif}-\eqref{eq:distrib_lognormal}.% and $\omicron = 2$.
    }
\label{tab:max}
\end{table}

% \begin{table}
%\centering
%\begin{tabular}{c|ccccccc}
%$\omicron$ & $t^*$  &  $s^*$  &  $x^*_1$ & $x^*_2$  & $\log\chi^*$  & $\log\eta^*$  &  $I_2(t^*,s^*,x^*)$     \\
%\hline
%0.5&      &    & & & &   & 
%\\
%1 &      &    & & & &   & 
%\\
%1.5 &      &    & & & &   &  
%\\
%2 &      &    & & & &   &   
%%\\
%%2.5 &      &    & & & &   &    
%%\\
%%3 &      &    & & & &   &    
%\end{tabular}
%    \caption{The maximum $I_2(t^*,s^*,x^*) = \max_{t,s,x}I_2(t,s,x)$ with $t,s,x$ ranging over a batch counting $10^6$ samples generated at random in accordance with \eqref{eq:distrib_unif}-\eqref{eq:distrib_lognormal}. In particular, the parameter $\omicron$ is the volatility of the lognormal distributions for $\chi,\eta$, which determine $x=(x_1,x_2)$.}
%\label{tab:max}
%\end{table}

\begin{figure}[htb]
\centering
\includegraphics[width=0.7\textwidth,height=0.3\textheight]{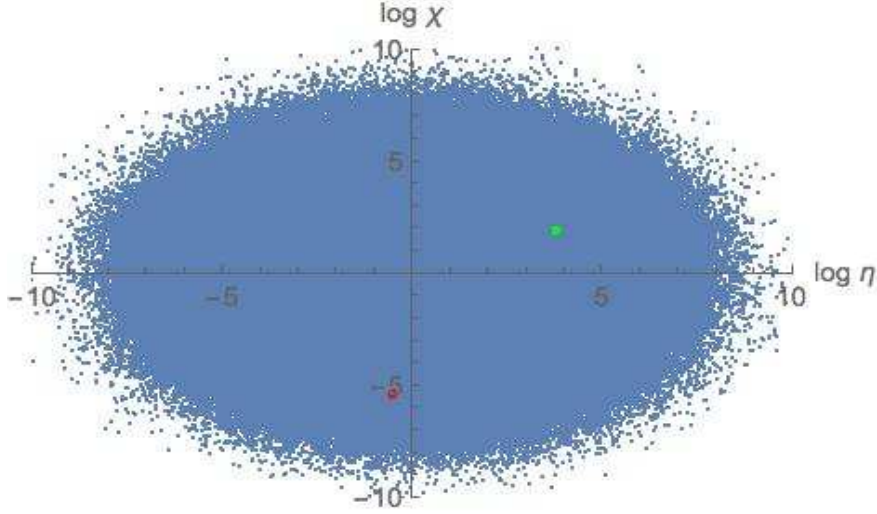} 
\caption{Blue dots: $2\times 10^7$ samples of the pairs $(\log\eta,\log\chi)$ generated at random with distribution $\mathcal{N}_{0,2}\otimes \mathcal{N}_{0,2}$. Green dot: pair $(\log\eta^*,\log\chi^*)$ at which the maximum $I_1$ in Table \ref{tab:max} is attained. Red dot: pair $(\log\eta^*,\log\chi^*)$ at which the maximum $I_2$ in Table \ref{tab:max} is attained.} 
\label{fig:ste4}
\end{figure}

\appendix

\section{A topological lemma}\label{sec:app_top_lemma}

Let $B = \big\{ (x,y) \in \R^2 \mid x^2 + y^2 < 1 \big\}$ denote the unit open ball of $\R^2$, and let $f:\overline{B} \to \R^{2}$ be a continuous map such that $f|_{\partial B}$ is a homeomorphism. In accordance with the Jordan-Schoenflies separation theorem, in the sequel $\internal$ will denote the ``inside" of $f(\partial B)$, meaning the bounded set of $\R^{2}$, whose border is $f(\partial B)$, homeomorphic to an open ball. We will also denote by $\external$ the ``outside" of $f(\partial B)$, namely the complementary of $f(\partial B)\cup\internal$, which is homeomorphic to the complementary of a closed ball.

\begin{lemma}\label{lem:lem_topol}
Let $B$ be an open ball of $\R^2$, and let $f:\overline{B} \to \R^{2}$ be such that: 
\begin{itemize}
\item[(i)] $f$ is a local homeomorphism;
\item[(ii)] $f|_{\partial B}$ is a homeomorphism.
\end{itemize}
Then $f|_B$ is a homeomorphism between $B$ and $\internal$.
\end{lemma}

\begin{proof}
We first prove that %the image of 
\begin{equation}\label{eq:surjective}
%Im(f|_B) 
f(B) = \internal.
\end{equation}
%$f|_B$ is $\overset{\circ}{f(\partial B)}$. 
Assume that $\external\cap f(\overline B) \neq \emptyset$. Since $f(\overline B)$ is a compact set, we also have $\external\cap \big(f(\overline B) \big)^c\neq \emptyset$. Therefore, being $\external$ an open connected set, we have $\external\cap \partial \left( f(\overline B) \right) \neq \emptyset$. Let $y\in \external\cap \partial \left( f(\overline B) \right)$, again by compactness of $f(\overline B)$ there exists $x \in \overline B$ such that $f(x)=y$. Note that $x\in B$, as $y \in \external$. However, by (i), there exists a neighborhood $U$ of $x$ such that $f(U)$ is a neighborhood of $y$, which contradicts the fact that $y\in \partial \left( f(\overline B) \right)$. We thus proved that $f(\overline B) \subset \overline{\internal}$. 

An analogous argument shows that $f(\overline B) \supset \overline{\internal}$. Assume by contradiction that $f(\overline B) \cap  \external \neq \emptyset$. Then $\internal \cap \partial \left( f(\overline B) \right) \neq \emptyset$, and the conclusion follows. Thus,
\begin{equation}
f(\overline B) = \overline{\internal}.
\end{equation}
Moreover, if there were $x\in B$ such that $f(x)\in f(\partial B)$, then (i) would be violated because $f(\partial B) = \partial f(B)$. This proves \eqref{eq:surjective}. 

We now note that the compact subsets in the subspace topologies on $B$ and $\internal$ are all the closed subsets of $\R^2$ contained in $B$ and $\internal$, respectively. Therefore, $f|_B:B\to \internal$ is a proper function.

Now, in order to conclude the proof it is enough to apply Hadamard-Caccioppoli Theorem, a particular instance of which states that a local homeomorphism between two open and simply connected sets of $\R^n$ is a global homeomorphism if and only if it is a proper function. 
\end{proof}

%{\blue Risultato pi\`u forte, con teoria del grado.}

%\begin{remark}\label{prop1}
%It seems to us that the topological degree theory gives the following result stronger than Lemma \ref{lem:lem_topol}. 
%
%Let $B$ be an open ball of $\R^2$, and let $f:\overline{B} \to \R^{2}$ be a continuous map such that $f|_{\partial B}$ is a homeomorphism. Then 
%\begin{equation}\label{eq:imm}
%f(B) \supset \internal.
%\end{equation}
%Furthermore, if $f$ is also local homeomorphism, then $f|_B$ is a homomorphism between $B$ and $\internal$.
%\end{remark}

\section{Proof of Lemma \ref{lem:HJB_eq}}\label{sec:proof_lemma_HJB}
\begin{proof}[Proof of Lemma \ref{lem:HJB_eq}]
First note that, by \eqref{eq:Psi_invariance}, it is enough to prove \eqref{eq:fund_identity_dim_n} for $w=\id$% and $\sigma=1$
. Now, we observe that it suffices to prove that
\begin{equation}\label{eq:HJB_inf}
0 = \inf_{\omega\in\R} \big\{  \omega^2 + \big( \omega x_1 \partial_{x_1} + Y \big) \Psi(z,\id)   \big\}, \qquad z=(t,x)\in \R\times D, \ z\prec \id .
\end{equation}
Indeed, the function
\begin{equation}
 \omega^2 + \big( \omega x_1 \partial_{x_1} + Y \big) \Psi(z,\id)  
\end{equation}
has a global minimum at $\omega = - \frac{x_1}{2}\partial_{x_1}\Psi(z,\id)$, which is 
\begin{equation}
-\Big( \frac{x_1}{2}\partial_{x_1}\Psi(z,\id)\Big)^2 + Y \Psi(z,\id).
\end{equation}

Consider now the following extended control problem. For any $z=(t,x)\in \R\times D$ with $z\prec \id$, and for any $x_0\in\R$,  find
\begin{equation}\label{eq:control1_ext}
 \bar{\Psi}(t,x_0,x_1,x_2) = \min_{\omega%\atop \gamma(0)=x,\, \gamma(1)=y
}  \gamma_0(T),
\end{equation}
where the minimum is taken over all the controls $\omega\in L^2([t,T]) $ for which \eqref{eq:optimal_curves} with $w=\id$ are satisfied, and
\begin{align}\label{eq:optimal_curves_ext}
\dot{\gamma}_0(s) = \omega^2(s) ,\quad
t<s<T,&& 
\gamma_0(t) = x_0.
\end{align}
It is trivial to see that $\omega$ is optimal for this problem if and only if it is optimal for the problem \eqref{eq:control1}-\eqref{eq:optimal_curves}. Also, we have 
\begin{equation}
\bar{\Psi}(t,x_0,x_1,x_2) ={\Psi}(t,x_1,x_2;\id) + x_0. 
\end{equation}
Therefore, \eqref{eq:HJB_inf} is equivalent to
\begin{equation}
0 = \inf_{\omega\in\R} \big\{ \big( \omega^2 \partial_{x_0}\bar\Psi(t,x_0,x)+ \omega x_1 \partial_{x_1} + Y \big) \bar\Psi(t,x_0,x)   \big\}, \qquad z=(t,x)\in \R\times D, \ z\prec \id,\quad x_0\in\R .
\end{equation}
Eventually, the latter holds true by Theorem IV-4.1 in \cite{fleming2012deterministic}, which applies to the our problem as $\Psi(\cdot,\id)$ is smooth on its domain and the optimal control $\omega$ is piecewise continuous. Theorem IV-4.1 in \cite{fleming2012deterministic} also states \eqref{eq:control}, and this concludes the proof. 
%We now prove \eqref{eq:control}. We reinforce the notation and denote by $\omega_{z,w}$ and $\gamma_{z,w}$ the optimal control and the optimal trajectory of the control problem \eqref{eq:control1}-\eqref{eq:optimal_curves}. We have
%\begin{align}
%\Psi(z;w) & =\Psi\big(t,\gamma_{z,w}(t);w\big) = - \Big(   \Psi\big(T,\gamma_{z,w}(T);w\big) - \Psi\big(t,\gamma_{z,w}(t);w\big)   \Big)= - \int_t^T \frac{\dd}{\dd s} \Psi\big(s,\gamma_{z,w}(s);w\big) \dd s  \\
%&  = - \int_t^T \Big(  \partial_s \Psi\big(s,\gamma_{z,w}(s);w\big) +\nabla_\gamma  \Psi\big(s,\gamma_{z,w}(s);w\big)  \dot{\gamma}_{z,w}(s)  \Big) \dd s
%\intertext{(by \eqref{eq:optimal_curves})}
%& = -  \int_t^T \bigg(   \partial_s \Psi\big(s,\gamma_{z,w}(s);w\big) + \nabla_\gamma  \Psi\big(s,\gamma_{z,w}(s);w\big) \left(\begin{matrix} 
%\omega_{z,w}(s) & 0 \\
%1 & 0
%\end{matrix} \right)
%\gamma_{z,w}(s) \bigg) \dd s 
%\intertext{(by \eqref{eq:fund_identity_dim_n})}
%& = -  \int_t^T \bigg(   \Big( \frac{\gamma_{z,w,1}(s)}{2} \partial_{x_1}  \Psi\big(\gamma_{z,w}(s);w\big)  \Big)^2 + 
%\omega_{z,w}(s) \gamma_{z,w,1}(s) \partial_{\gamma_1}  \Psi\big(s,\gamma_{z,w}(s);w\big) \bigg) \dd s
%.
%\end{align}
%This and \eqref{eq:control1} yield
%\begin{equation}
%|\omega_{z,w}(s)|^2   + \omega_{z,w}(s) \gamma_{z,w,1}(s) \partial_{\gamma_1}  \Psi\big(s,\gamma_{z,w}(s);w\big) + \Big( \frac{\gamma_{z,w,1}(s)}{2} \partial_{\gamma_1}  \Psi\big(\gamma_{z,w}(s);w\big)  \Big)^2 = 0,
%\end{equation}
%which in turn yields \eqref{eq:control}.
\end{proof}

\bibliographystyle{siam}
\bibliography{Bibtex-Final}

\end{document}